\documentclass[aap]{imsart}

\RequirePackage{amsthm,amsmath,amsfonts,amssymb}
\RequirePackage[authoryear]{natbib}
\RequirePackage[colorlinks,citecolor=blue,urlcolor=blue]{hyperref}
\RequirePackage{graphicx}

\usepackage{tikz}
\usepackage{pgfplots}
\usepackage{tkz-euclide}
\usepgfplotslibrary{fillbetween}
\usetikzlibrary{calc}
\pgfplotsset{compat=1.12}

\definecolor{babyblue}{rgb}{0.54, 0.81, 0.94}
\definecolor{rosebonbon}{rgb}{0.98, 0.26, 0.62}

\startlocaldefs
\theoremstyle{plain}

\newtheorem{thm}{Theorem}[section]
\newtheorem{lemma}[thm]{Lemma}
\newtheorem{prop}[thm]{Proposition}

\theoremstyle{remark}

\newtheorem{rem}[thm]{Remark}

\def\B{\mathbb{B}}

\def\P{\mathbb{P}}
\def\R{\mathbb{R}}
\def \F {\mathcal{F}}
\def\E{\mathbb{E}}

\def \1nd{\mathds{1}}

\def \I{\mathbb{I}}

\def \dd{\textup{d}}
\def\e{\mathbf{e} }

\endlocaldefs

\begin{document}

\begin{frontmatter}
\title{$L^p$ optimal prediction of the last zero \\of a spectrally negative L\'evy process}
\runtitle{$L^p$ optimal prediction of the last zero}

\begin{aug}
\author[A]{\fnms{Erik J.} \snm{Baurdoux}\ead[label=e1]{e.j.baurdoux@lse.ac.uk}},
\and
\author[B]{\fnms{Jos\'e M.} \snm{Pedraza}\ead[label=e2]{josemanuel.pedraza-ramirez@uwaterloo.ca}}
\address[A]{Department of Statistics, London School of Economics and Political Science, \printead{e1}}

\address[B]{Department of Statistics and Actuarial Science, University of Waterloo, \printead{e2}}
\end{aug}

\begin{abstract}
Given a spectrally negative L\'evy process $X$ drifting to infinity, (inspired on the early ideas of Shiryaev (2002)) we are interested in finding a stopping time that minimises the $L^p$ distance ($p>1$) with $g$, the last time $X$ is negative. The solution is substantially more difficult compared to the case $p=1$, for which it was shown by Baurdoux and Pedraza (2020) that it is optimal to stop as soon as $X$ exceeds a constant barrier. In the case of $p>1$ treated here, we prove that solving this optimal prediction problem is equivalent to solving an optimal stopping problem in terms of a two-dimensional strong Markov process that incorporates the length of the current positive excursion away from $0$. We show that an optimal stopping time is now given by the first time that $X$ exceeds a non-increasing and non-negative curve depending on the length of the current positive excursion away from $0$. We further characterise the optimal boundary and the value function as the unique solution of a non-linear system of integral equations within a subclass of functions. As examples, the case of a Brownian motion with drift and a Brownian motion with drift perturbed by a Poisson process with exponential jumps are considered.
\end{abstract}



\end{frontmatter}

\tableofcontents


\section{Introduction}

In recent years, last passage times have received considerable attention in the literature. For instance, in risk theory, the capital of an insurance company over time is modelled by a stochastic process $X=\{X_t,t\geq 0\}$. In the classical risk theory, $X$ is modelled by the Cram\'er--Lundberg process, defined as a compound Poisson process with drift. This model assumes that a premium is collected continuously at rate $c>0$. In contrast, the claims arrive according to a Poisson process, with the claims sizes being independent and identically distributed. In more recent literature, $X$ is considered to be a more general L\'evy process (see e.g. \cite{huzak2004ruin} or \cite{kluppelberg2004ruin}) with the motivation that the surplus of the company is the superposition of several independent small and large claims, and that some uncertainty should be added in the premium income. A critical quantity of interest is the moment of ruin, which is classically defined as the first passage time below zero.\\

\cite{GERBER1990115} and \cite{chiu2005passage} propose the following extension by considering that $X$ represents the capital of an individual company portfolio. After the moment of ruin, the company may have funds to endure a negative surplus of this portfolio for some time (possibly with the influx of capital from other portfolios) with the hope of having a positive surplus in the long term. Then, the last passage time below level zero is regarded as the last recovery time, so after that, there will no longer be ruin. Note that prior knowledge of this last passage time may have important implications in the risk management of the company (or even a start-up company). Indeed, when launching a new product or portfolio, the insurance company naturally expects losses in the first few months (or even years), aiming that the project is profitable in the long term. Hence, after the last moment of ruin has occurred, fewer funds are needed on their reserves, and they can be destined for other projects or portfolios within the company. In \cite{chiu2005passage}, the Laplace transform of the last passage time is derived in this framework.\\

Secondly, \cite{paroissin2013first} consider spectrally positive L\'evy processes as a degradation model. In particular, the ageing of a device is modelled by a process $D=\{D_t,t\geq 0 \}$, given by a subordinator perturbed by a Brownian motion. In this framework, large values of $D$ represent a significant deterioration of the device, so the effect of the subordinator means constant degradation, whilst the Brownian motion component represents minor repairs made to the device. In a traditional setting, the failure time of a device is the first time the model hits a specific critical level $b$. However, another approach has been used in the literature. For example, \cite{BARKER200933} considered the failure time as the last time the process is below $b$. After the last passage time, the process can never go back to this level, meaning that the device is ``beyond repair''. \\

 Thirdly, \cite{egami2017time} studied the last passage time of a general time-homogeneous transient diffusion with applications to credit risk management. They proposed the leverage process (the ratio of a company asset process over its debt) as a geometric Brownian motion over a process that grows at a risk-free rate. It is shown there that the last passage time of the leverage ratio is equivalent to a last passage time of a Brownian motion with drift. In this setting, the last passage represents the situation where the company cannot recover to normal business conditions after this time has occurred. \\

	 An important feature of last passage times is that they are random times that are not stopping times. 
In the recent literature, the problem of finding a stopping time that approximates last passage times has been solved. There are, for example, various papers in which the approximation is in $L_1$ sense. To mention a few: \cite{du2008predicting} predicted the last zero of a Brownian motion with drift in a finite horizon setting; \cite{duToit2008} predicted the time of the ultimate maximum at time $t=1$ for a Brownian motion with drift is attained; \cite{shiryaev2009} focused on the last time of the attainment of the ultimate maximum of a Brownian motion and proceeded to show that it is equivalent to predicting the last zero of the process in this setting; \cite{glover2013three} predicted the time in which a transient diffusion attains its ultimate minimum; \cite{glover2014optimal} predicted the last passage time of a level $z > 0$ for an arbitrary non-negative time-homogeneous transient diffusion; \cite{baurdoux2014predicting} predicted the time at which a L\'evy process attains its ultimate supremum and \cite{baurdoux2016optimal} predicted when a positive self-similar Markov process attain its path-wise global supremum or infimum before hitting zero for the first time and \cite{baurdoux2018predicting} predicted the last zero of a spectrally negative L\'evy process.\\

Note that in \cite{Shiryaev2002}, the author states some general optimal prediction problems that are natural for the ``technical analysis'' of financial data. In particular, among other problems, it is proposed to predict the time in which a process attains its maximum (over a finite interval) in an $L_p$ sense. However, no solution to the problem is provided. Moreover, to the best of our knowledge, optimal prediction problems for last passage times have been only solved in an $L_1$ sense. Hence, inspired by this, we consider the problem of predicting the last zero of a spectrally negative L\'evy process (drifting to infinity) in an $L^p$ sense, i.e. we are interested in solving
\begin{align*}
\inf_{\tau \in \mathcal{T}} \E(|\tau-g|^p),
\end{align*}
%
where $g=\sup\{t\geq 0: X_t\leq 0 \}$ is the last time a spectrally negative L\'evy process drifting to infinity is below the level zero and $p> 1$. The case when $p=1$ was solved in \cite{baurdoux2018predicting} for the spectrally negative case. An optimal stopping time, in this case, is the first time the process crosses above a fixed level $a^*\geq 0$, which is characterised in terms of the distribution function of the infimum of the process. The case $p>1$ is substantially more complex, as an optimal stopping time now depends on the length of the current excursion above the level zero given by $U_t=t-\sup\{0\leq s\leq t: X_s\leq 0 \}$. The process $(U,X)$ is a Markov process taking values in $E=[(0,\infty)\times (0,\infty)] \cup [\{0 \} \times (-\infty,0)]$.\\

We show that an optimal stopping time (when $p>1$) is given by $\tau_D=\inf\{t>0: (U_t, X_t)\in D \}=\inf\{t\geq 0: X_t\geq b(U_t) \}$, where $b$ is a non-negative, non-increasing and continuous curve. That is, it is not optimal to stop when $(U, X)$ is in the (continuation) set $C:=E\setminus D$ whilst we should stop as soon as the process enters the (stopping) set $D$ (see Figure \ref{fig:stoppingsetinE}). In other words, given the strong dependence of $U$ on $X$, the latter has the following interpretation in terms of the sample paths of $X$: It is optimal to stop when $X$ is sufficiently large or has stayed for a sufficiently large  amount of time above zero, whereas we will never stop when $X$ is in the negative half-line (see Figure \ref{fig:brownianmotionoptimalboundary}). \\
 \begin{figure}
\begin{center}
\begin{tikzpicture}
\begin{axis}%
    [
        x=10mm,
        y=10mm,
        xticklabels={,,}
        xtick={},   
        xmin=-0.3,
        xmax=7,
        axis x line=middle,
        yticklabels={,,},
        ymin=-3,
        ymax=5,
        axis y line=middle,
        x label style={at={(axis description cs:0.96,0.35)},anchor=north},
        y label style={at={(axis description cs:-0.01,0.96)},anchor=south},
        xlabel={\tiny{$U$}},
        ylabel={\tiny{$X$}},
        no markers,
        samples=100,
        domain=-1:7,
        restrict y to domain=-3:5
    ]

\addplot[thick,name path=A, domain=0:7,samples=400] (x,{1/x});
\addplot[draw=none,name path=B,] {0}; 
\addplot[draw=none,name path=C,] {4.78}; 
\draw[fill=babyblue, opacity=0.3]  (0,0) -- (0,-3) -- (0.1,-3) -- (0.1,0) -- cycle;
\addplot[babyblue, opacity=0.3] fill between[of=B and C,soft clip={domain=0.01:0.243}];
\addplot[babyblue, opacity=0.3] fill between[of=A and B,soft clip={domain=0.243:7}];
\path[inner sep=0pt, node font=\footnotesize]
node[anchor=west] at (0.6,0.6){$C$}
node[anchor=west] at (2.5,2){$D$}
;
\end{axis} 
\end{tikzpicture}
\end{center}
\caption{Stopping and continuation set in the $(U,X)$ plane}
\label{fig:stoppingsetinE}
\end{figure}
In the figure below we include a plot of a sample path of $X_t$ and $b(U_t)$, where we calculated numerically the function $b$ for the Brownian motion with drift case (see Section \ref{sbsec:BMwdriftexampleLp} and Figure \ref{pic:BMexample}).\\
\begin{figure}
\centering
\includegraphics[scale=0.32]{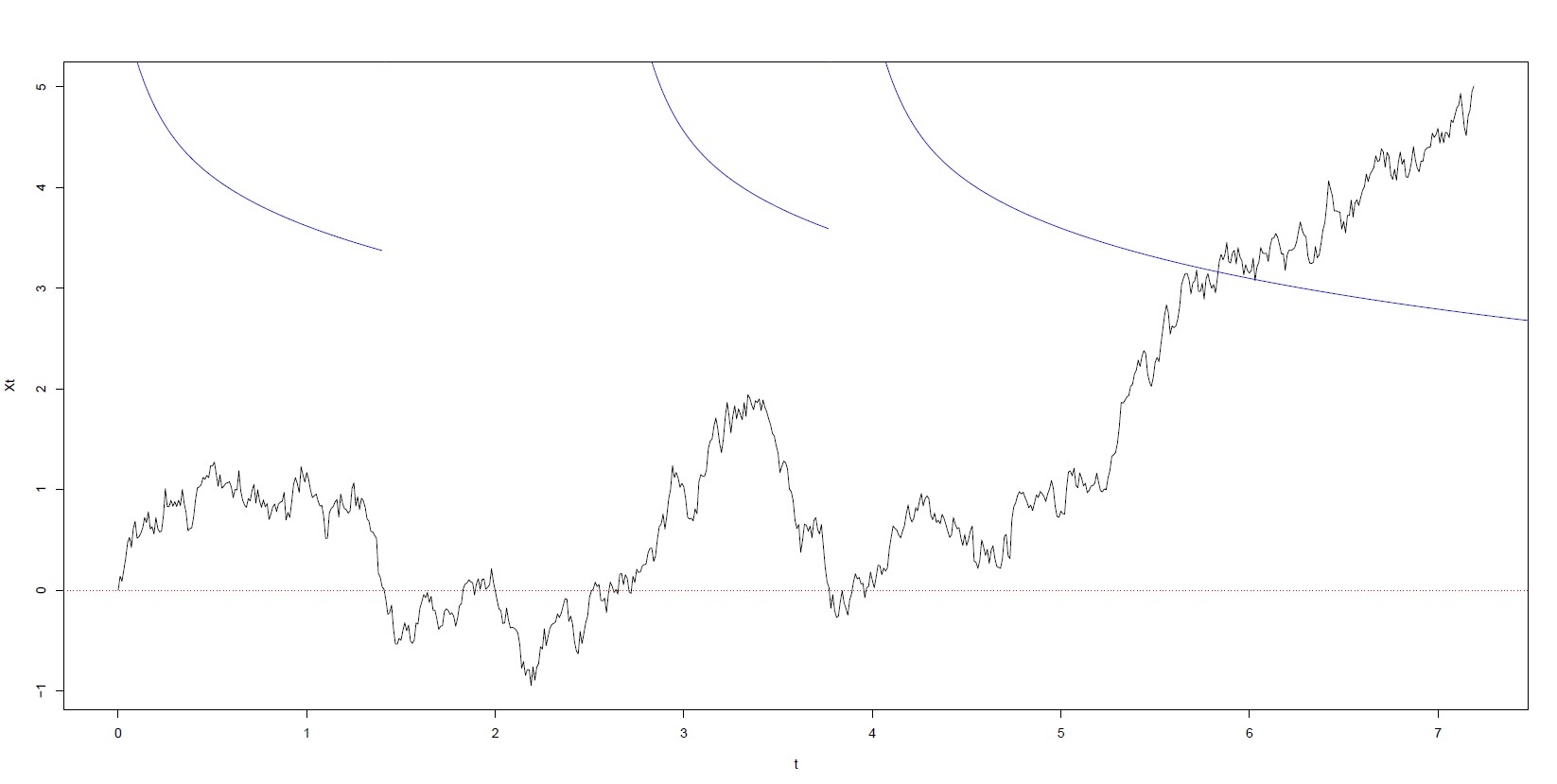}
\label{fig:brownianmotionoptimalboundary}
\caption{Black line: $t\mapsto X_t$; Blue line: $t\mapsto b(U_t)$.}
\end{figure}

The paper is organised as follows. Section \ref{sec:Preliminaries} gives a short overview of the main results and notation on the fluctuation theory of spectrally negative L\'evy processes. 
In Section \ref{sec:optimalpredictionproblem}, we formulate the optimal prediction problem, and we show that it is equivalent to an optimal stopping problem whose solution is described in Theorem \ref{thm:solutiontotheoptimalstoppingproblem}. Since the proof of Theorem \ref{thm:solutiontotheoptimalstoppingproblem} is rather long, we dedicate Section \ref{sec:Solutionoptimalstopping} for that purpose. In particular, we show that an optimal stopping time is given by the first time $X$ exceeds a boundary $b$, which depends on the length of the current excursion above zero. We derive various properties of $b$. For example, in Lemma \ref{lemma:continuous}, we show that $b$ is continuous and in Lemma \ref{lemma:smoothfit}, we show that smooth fit holds at the boundary. The main result of this paper is Theorem \ref{thm:characterisationofbandV}, which proof is devoted to Section \ref{sec:proofofmainmainthm}, providing a characterisation of $b$ and the value function of the optimal stopping problem. In Section \ref{sec:Examples}, we provide two numerical examples: Firstly, when $X$ is a Brownian motion with drift, and secondly, when $X$ is a Brownian motion perturbed by a compound Poisson process with exponential jumps. Finally, some of the more technical proofs are deferred to the Appendices \ref{sec:Appendix} and \ref{appendix:variationaliniequalities}.

\section{Preliminaries}
\label{sec:Preliminaries}
A L\'evy process $X=\{X_t,t\geq 0 \}$ is an almost surely c\`adl\`ag process that has independent and stationary increments such that $\P(X_0=0)=1$. We take it to be defined on a filtered probability space $(\Omega,\F, \mathbb{F}, \P)$ where $\mathbb{F}=\{\F_t,t\geq 0 \}$ is the filtration generated by $X$ which is naturally enlarged (see Definition 1.3.38 of \cite{bichteler2002stochastic}). From the stationary and independent increments property, the law of $X$ is characterised by the distribution of $X_1$. We hence define the characteristic exponent of $X$, $\Psi(\theta):=-\log(\E(e^{i\theta X_1}))$, $\theta\in\mathbb{R}$. The L\'evy--Khintchine formula guarantees the existence of constants, $\mu \in \R$, $\sigma\geq 0$ and a measure $\Pi$ concentrated in $\R\setminus \{0\}$ with the property that $\int_{\R} (1\wedge x^2) \Pi(\dd x)<\infty$ (called the L\'evy measure) such that for any $\theta \in \R$,

\begin{align*}
\Psi(\theta)= i \mu\theta +\frac{1}{2}\sigma^2 \theta^2-\int_{\R} (e^{i \theta y}-1-i \theta y\I_{\{|y|<1 \}})\Pi(\dd y).
\end{align*}
We now state some properties and facts about L\'evy processes. The reader can refer, for example, to \cite{bertoin1998levy}, \cite{sato1999levy} and \cite{kyprianou2014fluctuations} for more details. Every L\'evy process $X$ is also a strong Markov $\mathbb{F}$-adapted process. For all $x\in \R$, denote $\P_x$ as the law of $X$ when started at the point $x\in \R$, that is, $\E_x(\cdot)=\E(\cdot|X_0=x)$. Due to the spatial homogeneity of L\'evy processes, the law of $X$ under $\P_x$ is the same as that of $X+x$ under $\P$.\\

From the L\'evy--It\^o decomposition we can write for any $t\geq 0$,

\begin{align*}
X_t=\sigma B_t-\mu t+\int_{0}^t\int_{(-\infty,\infty)\setminus(-1,1)} x N(\dd s, \dd x)+\int_{0}^t\int_{(-1,1)} x( N(\dd s	, \dd x)-\dd s\Pi(\dd x)),
\end{align*}
where $B$ is a standard Brownian motion and $N$ is an independent Poisson random measure on $\R^+\times \R$ with intensity $\dd t \times \Pi(\dd x)$. Note that the Poisson random measure $N$ describes the jumps of the process $X$. \\

In the following sections, we often use the so-called compensation formula for Poisson random measures (see e.g. Theorem 4.4 in \cite{kyprianou2014fluctuations}). Let $\xi:[0,\infty)\times \R \times \Omega \mapsto [0,\infty)$ a measurable function such that: for each $t\geq 0$, the random variable $\xi(t,x)$ is $\B(\R)\times \F_t$-measurable, where $\B(\R)$ is the Borel sigma algebra on $\R$; for each $x\in \R$, the stochastic process $\{\xi(t,x), t\geq 0 \}$ is almost surely left continuous. Then we have that 
\begin{align}
\label{eq:compensationformulaPRM}
\E\left( \int_{0}^t \int_{\R} \xi(s,x) N(\dd s, \dd x)  \right)=\E\left( \int_{0}^t \int_{\R} \xi(s,x) \Pi(\dd x) \dd s  \right).
\end{align}
 
The process $X$ is a spectrally negative L\'evy process if it has no positive jumps ($\Pi(0,\infty)=0$) with no monotone paths. We now state some important properties and fluctuation identities of spectrally negative L\'evy processes, which will be useful to us later in this paper. We refer to \cite{bertoin1998levy}, Chapter VII or  Chapter 8 in \cite{kyprianou2014fluctuations} for details.\\

Due to the absence of positive jumps, we can define the Laplace transform of $X_1$. We denote $\psi(\beta)$ as the Laplace exponent of the process, that is, $\psi(\beta)=\log(\E(e^{\beta X_1}))$ for $\beta \geq 0$. For such $\beta$ we have that

\begin{align*}
\psi(\beta)=-\mu\beta +\frac{1}{2}\sigma^2 \beta^2+\int_{(-\infty,0)} (e^{\beta y}-1-\beta y\I_{\{y>-1 \}})\Pi(\dd y).
\end{align*}
The function $\psi$ is infinitely often differentiable and strictly convex function on $(0,\infty)$ with $\psi(\infty)=\infty$. In particular, $\psi'(0+)=\E(X_1)\in [-\infty,\infty)$ determines the behaviour of $X$ at infinity. When $\psi'(0+)>0$ the process $X$ drifts to infinity, i.e., $\lim_{t \rightarrow \infty} X_t=\infty$; when $\psi'(0+)<0$, $X$ drifts to minus infinity and the condition $\psi'(0+)=0$ implies that $X$ oscillates, that is, $\limsup_{t\rightarrow \infty} X_t =-\liminf_{t\rightarrow \infty} X_t=\infty$. We denote by $\Phi$ the right-inverse of $\psi$, i.e.
\begin{align*}
\Phi(q)=\sup\{\beta\geq 0: \psi(\beta)=q \}, \qquad q\geq 0.
\end{align*}
In the particular case that $X$ drifts to infinity, we have that $\psi'(0+)>0$ which implies that $\psi$ is strictly increasing and then $\Phi$ is the usual inverse with $\Phi(0)=0$.\\

It turns out that the path variation of L\'evy processes is characterised by its L\'evy triplet. Indeed, for each $t>0$, the paths of $X$ have finite variation on $(0,t]$, if and only if $\sigma=0$ and $\int_{(-1,0)} |x| \Pi(\dd x)<\infty$. Since the path variation of $X$ does not depend on $t> 0$, we just simply say that $X$ is of (in)finite variation. \\

Denote by $\tau_a^+$ the first passage time above the level $a>0$,
\begin{align*}
\tau_a^+=\inf\{t>0: X_t>a \}.
\end{align*}
The Laplace transform of $\tau_a^+$ is given by
\begin{align}
\label{eq:laplacetransformtau0}
\E(e^{-q \tau_a^+}\I_{\{\tau_a^+<\infty \}})=e^{-\Phi(q)a}, \qquad a> 0.
\end{align}

An important family of functions for spectrally negative L\'evy processes consists of the scale functions, usually denoted by $W^{(q)}$ and $Z^{(q)}$. For all $q\geq 0$, the scale function $W^{(q)}:\R \mapsto \R_+$ is such that $W^{(q)}(x)=0$ for all $x<0$ and it is characterised on the interval $[0,\infty)$ as the strictly increasing and continuous function with Laplace transform given by

\begin{align*}
\int_0^{\infty} e^{-\beta x} W^{(q)}(x)\dd x=\frac{1}{\psi(\beta)-q}, \qquad \text{ for } \beta >\Phi(q).
\end{align*}
The function $Z^{(q)}$ is defined for all $q\geq 0$ by

\begin{align*}
Z^{(q)}(x):= 1+q\int_0^x W^{(q)}(y) \dd y, \qquad \text{ for } x\in \R.
\end{align*}
For the case $q=0$ we simply denote $W=W^{(0)}$. When $X$ has paths of infinite variation, $W^{(q)}$ is continuous on $\R$ and $W^{(q)}(0)=0$ for all $q\geq 0$, otherwise $W^{(q)}(0)=1/d$ for all $q\geq 0$, where
\begin{align*}
d:=-\mu -\int_{(-1,0)} x\Pi(\dd x).
\end{align*}
Note that since processes with monotone paths are excluded from the definition of spectrally negative processes, we necessarily have that $d>0$ when $X$ is of bounded variation.\\

For all $q\geq 0$, $W^{(q)}$ has left and right derivatives. Moreover, when $X$ is of infinite variation we have that $W^{(q)}\in C^1((0,\infty))$ with right-derivative at zero given by $W^{(q)'}(0)=2/\sigma^2$. When $X$ is of finite variation $W^{(q)}\in C^1((0,\infty))$ when $\Pi$ has no atoms. Henceforth, we will assume that when $X$ is of finite variation, the L\'evy measure $\Pi$ has no atoms. Furthermore, for each $x\geq 0$ and $q\geq 0$, $W^{(q)}$ has the following representation
\begin{align}
\label{eq:convolutionrepresentationofWq}
W^{(q)}(x)=\sum_{k=0}^{\infty} q^k W^{*(k+1)}(x),
\end{align}
where $W^{*(k+1)}$ is the $(k+1)$-th convolution of $W$ with itself. Various fluctuation identities for spectrally negative L\'evy processes have been found in terms of the scale functions. Here we list some that will be useful in later sections. Denote by $\tau_x^-$ the first passage time below the level $x\leq 0$, i.e.,
\begin{align*}
\tau_x^-=\inf\{t>0: X_t<x \}.
\end{align*}
For any $q\geq 0$ and $x\leq a$ we have
\begin{align}
\label{eq:laplacetransformoftaua+beforecrossingthelevelzero}
\E_x\left(e^{-q \tau_a^+} \I_{\{\tau_0^->\tau_a^+ \}} \right)=\frac{W^{(q)}(x)}{W^{(q)}(a)}.
\end{align}
For any $x\in \R$ and $q\geq 0$,
\begin{align}
\label{eq:laplacetransoformoftau0minus}
\E_x(e^{-q \tau_0^-}\I_{\{\tau_0^-<\infty \}})=Z^{(q)}(x)-\frac{q}{\Phi(q)} W^{(q)}(x),
\end{align}
where we understand $q/\Phi(q)$ in the limiting sense when $q=0$. Since $X$ has only negative jumps, we have that it only creeps upwards, that is,
\begin{align}
    \P(X_{\tau_x^+}=x ,\tau_x^+<\infty )=1
\end{align}
for any $x>0$. Moreover, $X$ creeps downwards if and only if $\sigma>0$ with probability given by
\begin{align}
\label{eq:probabilityofcreepingdownwards}
    \P_x(X_{\tau_0^-}=0,\tau_0^-<\infty)=\frac{\sigma^2}{2}\left(W'(x)-\Phi(0)W(x) \right)
\end{align}
for any $x>0$.\\

Denote by $\underline{X}_t=\inf_{0\leq s\leq t} X_s$ and $\overline{X}_t=\sup_{0\leq s\leq t} X_s$ the running infimum and running maximum of the process $X$ up to time $t>0$, respectively. For $q\geq 0$, let $\e_q$ be an exponential random variable with mean $1/q$ independent of $X$, where we understand that $\e_q=\infty$ almost surely when $q=0$.  Then $ \overline{X}_{\e_q}$ is exponentially distributed with parameter $\Phi(q)$ and the Laplace transform of $\underline{X}_{\e_q}$ is given by
\begin{align}
\label{eq:laplacetransformofrunninginfimumexptime}
\E(e^{\beta \underline{X}_{\e_q}})=\frac{q}{\Phi(q)} \frac{\Phi(q)-\beta}{q-\psi(\beta)}, \qquad \beta\geq 0.
\end{align}
Denote by $\sigma_x^-$ the first time the process $X$ is below or equal to the level $x$, i.e.,

\begin{align}
\label{eq:definitionofsigmax-}
\sigma_x^-=\inf\{t>0: X_t \leq x \}.
\end{align}
It is easy to show that the mapping $x\mapsto \sigma_x^-$ is non-increasing, right-continuous with left limits. The left limit is given by $\lim_{h\downarrow 0} \sigma_{x-h}^-=\tau_x^-$ for all $x\in \R$. Moreover, since

\begin{align*}
\E(e^{-q \sigma_{x}^-}\I_{\{\sigma_x^- <\infty\}})=\P(\e_q >\sigma_x^-)=\P(\underline{X}_{\e_q} \leq -x)
\end{align*}
for all $x \leq 0$, and the fact that the random variable $\underline{X}_{\e_q}$ is continuous on $(-\infty,0)$, we have that the stopping times $\sigma_x^-$ and $\tau_x^-$ have the same distribution, for any $x>0$. When $X$ is of infinite variation, $X$ enters instantly to the set $(-\infty,0)$ whilst in the finite variation case, there is a positive time before the process enters it. That implies that in the infinite variation case, $\tau_0^-=\sigma_0^-=0$ almost surely. Note that in the finite variation case, since the time $t=0$ is excluded from the definition of $\sigma_0^-$, from the fact that $0$ is irregular for $(-\infty,0]$ (see discussion in \cite{kyprianou2014fluctuations} on p. 157) and due to equation \eqref{eq:probabilityofcreepingdownwards} we have that $\sigma_0^-=\tau_0^->0$ a.s. \\

 Let $q>0$ and $a\in \R$.
The $q$-potential measure of $X$ killed on exiting $[0,a]$,
\[\int_0^{\infty} e^{-qt} \P_x(X_t \in \dd y,t<\tau_a^+ \wedge \tau_0^-)\dd t\] is absolutely continuous with respect to Lebesgue measure and it has a density given by
\begin{align}
\label{eq:qpotentialdensitytkillingonexiting0,a}
\frac{W^{(q)}(x)W^{(q)}(a-y)}{W^{(q)}(a)} -W^{(q)}(x-y), \qquad x,y\in[0,a].
\end{align}
Similarly, the $q$-potential measure of $X$ killed on exiting $(-\infty,a]$ and the $q$-potential measure of $X$ are absolutely continuous with respect to Lebesgue measure with a density given by
\begin{align}
\label{eq:qpotentialdensitytkillingonexitinga}
 e^{-\Phi(q)(a-x)}W^{(q)}(a-y) -W^{(q)}(x-y), \qquad  x,y\leq a,
\end{align}
and
\begin{align}
\label{eq:qpotentialdensitywithoutkilling}
\Phi'(q)e^{-\Phi(q)(y-x)}-W^{(q)}(x-y), \qquad x,y\in\R,
\end{align}
respectively.
In the case when $X$ drifts to infinity these expression are also valid for $q=0$.\\

For any $t\geq 0$ and $x\in \R$, we denote by $g_t^{(x)}$ the last time that the process is below $x$ before time $t$, i.e.,
\begin{align}
\label{eq:gt}
g_t^{(x)}=\sup\{0\leq s\leq t: X_s \leq x \},
\end{align}
with the convention $\sup \emptyset=0$. We simply denote $g_t:=g_t^{(0)}$ for all $t\geq 0$.
Note that when $\mathbb{P}(X_t\geq 0)=\rho$ for some $\rho\in(0,1)$, then $g_t/t$ follows the generalised arcsine law with parameter $\rho$, see Theorem 13 in \cite{bertoin1998levy}.
The last-hitting time of zero is of key importance in the study of Az\'ema's martingale (see \cite{AzemaYor1989}).
We also define, for each $t\geq 0$, $U_t^{(x)}$ as the time spent by $X$ above the level $x$ before time $t$ since the last visit to the interval $(-\infty,x]$, i.e.,

\begin{align*}
U_t^{(x)}:=t-g_t^{(x)}, \qquad t\geq 0.
\end{align*}
It turns out that for our optimal prediction problem
\begin{align*}
\inf_{\tau \in \mathcal{T}} \E(|\tau-g|^p),
\end{align*}
for $p>1$, the process $U_t=U_t^{(0)}$ plays a vital role. 
It can be readily seen that for all $x\in \R$, the process $\{U_t^{(x)},t\geq 0\}$ is not a Markov process. We now list a number of results from \cite{baurdoux2019Itolastzero} concerning $U=\{U_t,t\geq 0 \}$. The strong Markov property holds for the two dimensional process $(U,X)=\{(U_t,X_t),t\geq 0 \}$ with respect to the filtration $\{\F_t,t\geq 0 \}$ and state space given by
\begin{align*}
E=\{(u,x): u>0 \text{ and } x>0 \}\cup \{(u,x): u=0 \text{ and } x\leq 0 \}.
\end{align*}
Then, there exists a family of probability measures  $\{ \P_{u,x},  (u,x)\in E\}$ such that for any $A \in \B(E)$, Borel set of $E$, we have that $\P_{u,x}((U_{\tau +s}, X_{\tau+s})\in A|\F_{\tau} )=\P_{U_{\tau},X_{\tau}}((U_s,X_s)\in A)$. For each $(u,x)\in E$, $\P_{u,x}$ can be written in terms of $\P_x$ via

\begin{align}
\label{eq:initialmeausrePux}
\E_{u,x}(h(U_s,X_s)):=\E_x(h(u+s, X_s)\I_{\{ \sigma_0^->s\}})+\E_x(h(U_s, X_s)\I_{\{ \sigma_0^-\leq s\}}),
\end{align}
for any positive measurable function $h$. Note that the stochastic process $(U,X)$ is a semimartingale so that It\^o formula is known (see e.g. Theorem IV.71 in \cite{protter2005}). However, given the strong dependence between $U$ and $X$, we can give a more explicit formula in terms of the dynamics of $X$ (see Theorem 3.3 in \cite{baurdoux2019Itolastzero}). Let $F:E\mapsto \R$ a continuous function that satisfies:

\begin{enumerate}
\item[i)] The mapping $x \mapsto F(0,x)$ is $C^{1}$ on $(-\infty,0)$ such that, when $X$ is of infinite variation, the second derivative $\frac{\partial^2}{\partial x^2 }F(0,x)$ exists and is continuous on $(-\infty,0)$;

\item[ii)] The mapping $(u,x)\mapsto F(u,x)$ is $C^{1,1}$ on $(0,\infty)\times(0,\infty)$ such that, when $X$ is of infinite variation, the second derivative $\frac{\partial^2}{\partial x^2 }F(u,x)$ exists and is continuous on $(0,\infty)$, for all $u\geq 0$;

\item[iii)] In the case that $\sigma>0$, $F$ is such that $\lim_{h\downarrow 0} F(u,h)=F(0,0)$ for all $u> 0$ and 
\begin{align}
\label{eq:pastingatzeroItoformula}
\frac{\partial}{\partial x}F(0,0+)=\frac{\partial}{\partial x}F(0,0-).
\end{align}
\end{enumerate}
Then we have the following version of It\^o formula

\begin{align}
&F(U_{t}, X_t)\nonumber\\
&=F(U_0,X_0)+\int_{0}^{t} \frac{\partial }{\partial u} F(U_s,X_{s}) \I_{\{X_s> 0 \}}\dd s \nonumber\\
&\qquad + \int_{0}^{t} \frac{\partial }{\partial x} F(U_{s-},X_{s-})\dd X_s+\frac{1}{2} \sigma^2 \int_{0}^{t } \frac{\partial^2 }{\partial x^2} F(U_{s},X_{s})\dd s \nonumber\\
\label{eq:ItoformulaforUtXt}
& \qquad +\int_{[0,t] } \int_{(-\infty,0)} \left( F(U_{s},X_{s^-}+y) -F(U_{s-},X_{s^-})-y\frac{\partial }{\partial x} F(U_{s-},X_{s^-}) \right)N(\dd s, \dd y)
\end{align}
Moreover, if in addition $F$ is a bounded function, the infinitesimal generator $\mathcal{A}_{U,X}$ of the process $(U,X)$ is given by
\begin{align}
&\mathcal{A}_{U,X}(F)(u,x) \nonumber\\
&=\frac{\partial }{\partial	u} F(u,x) \I_{\{x> 0 \}} -\mu \frac{\partial }{\partial	x} F(u,x) +\frac{1}{2} \sigma^2  \frac{\partial^2}{\partial x^2} F(u,x)\nonumber \\
& \qquad +\int_{(-\infty,0)} \left( F(u,x+y) -F(u,x)-y\I_{\{y>-1\}}\frac{\partial }{\partial x} F(u,x) \right)\I_{\{x+y >0 \}}\Pi(\dd y)\nonumber\\
& \qquad + \int_{(-\infty,0)} \left( F(0,x+y) -F(0,x)-y\I_{\{y>-1\}}\frac{\partial }{\partial x} F(0,x) \right)\I_{\{x\leq 0 \}}\Pi(\dd y)\nonumber\\
& \qquad +  \int_{(-\infty,0)} \left( F(0,x+y)-F(u,x)-y\I_{\{y>-1\}} \frac{\partial }{\partial x} F(u,x) \right)\I_{\{0<x <-y \}}\Pi(\dd y)\nonumber\\
&=\frac{\partial }{\partial	u} \widetilde{F}(u,x)  -\mu \frac{\partial }{\partial	x} \widetilde{F}(u,x) +\frac{1}{2} \sigma^2  \frac{\partial^2}{\partial x^2} \widetilde{F}(u,x)\nonumber \\
\label{eq:infinitesimalgeneratorUX}
& \qquad +\int_{(-\infty,0)} \left( \widetilde{F}(u,x+y) -\widetilde{F}(u,x)-y\I_{\{y>-1\}}\frac{\partial }{\partial x} \widetilde{F}(u,x) \right)\Pi(\dd y),
\end{align}
where $\tilde{F}$ is a function that extends $F$ to the set $\R_+ \times \R$ given by

\begin{align}
\label{eq:extensionofftoR+R}
\widetilde{F}(u,x)=\left\{
\begin{array}{cc}
F(u,x), & \text{for } u>0  \text{ and } x>0,\\
F(0,x), & \text{for } u\geq 0 \text{ and } x\leq 0,\\
F(0,0), & \text{for } u=0 \text{ and } x>0.
\end{array}
\right.
\end{align}

\noindent In addition, we provide a formula to calculate an integral involving the process $\{(U_t,X_t), t\geq 0 \}$ with respect to time in terms of the excursions of $X$ above and below zero (see Theorem 3.6 in \cite{baurdoux2019Itolastzero}). Let $K:E\mapsto \R$ be a left-continuous function in each argument. Assume that there exists a non-negative function $C:\R_+ \times \R \mapsto \R$ such that $u\mapsto C(u,x)$ is a monotone function for all $x\in \R$, $|K(u,x)|\leq C(u,x)$ and $\E_{u,x}\left( \int_0^{\infty} e^{-qr } C(U_r,X_r+y) \dd r\right)<\infty$ for all $(u,x)\in E$ and $y\in \R$. Then we have that 
\begin{align}
\label{eq:calculationofUsXsintegral}
\E\left(\int_0^{\infty}e^{-qr} K(U_r,X_r) \dd r \right)=   \lim_{\varepsilon \downarrow 0 }\frac{ \E_{\varepsilon}\left(  \I_{\{ \tau_0^-<\infty\}} e^{-q \tau_0^-} K^-(X_{\tau_0^-}-\varepsilon) \right)+K^+(0,\varepsilon) }{ \psi'(\Phi(q)) W^{(q)}(\varepsilon)},
\end{align}
where $K^+$ and $K^-$ are given by
\begin{align*}
K^+(u,x)&=\displaystyle{\E_x\left( \int_0^{\tau_0^-}e^{-qr}K(u+r,X_r)\dd r \right)},\\
  K^-(x)&=\displaystyle{\E_x\left( \int_0^{\tau_0^+}e^{-qr}K(0,X_r)\dd r \right)},
\end{align*}
for all $(u,x)\in E$. As a direct application of the aforementioned formula, we can calculate a density of the $q$-potential measure of $(U,X)$ (see Corollary 3.10 in \cite{baurdoux2019Itolastzero}). For any $v,y>0$ we have that
\begin{align}
\label{eq:potentialmeasureUXatzero}
\int_0^{\infty} e^{-qr }\P(U_r \in \dd v, X_r \in \dd y) \dd r
&=\Phi'(q)   \frac{y}{v} e^{-qv}\P(X_{v} \in \dd y) \dd v \\
& =\Phi'(q)e^{-qv}\P(\tau_y^+ \in \dd v)\dd y \nonumber,
\end{align}
where the last equality follows from Kendall's identity (see e.g. Exercise 6.10 in \cite{kyprianou2014fluctuations}).\\
%
%

We conclude this section by collecting some additional results about the last passage time
\begin{align}
\label{eq:lastzeroofX}
g=g_\infty=\sup\{t\geq 0 : X_t\leq 0\}.
\end{align}
%
The Laplace transform of $g$ was found in \cite{chiu2005passage} as
\begin{align}
\label{eq:laplacetransformofg}
\E_x(e^{-qg})=e^{\Phi(q)x}\Phi'(q)\psi'(0+)+\psi'(0+)(W(x)-W^{(q)}(x)), \qquad q\geq 0.
\end{align}
The distribution function of $g$ under $\P_x$ is found by observing that
\begin{align}
\P_x(g\leq \gamma)
&=\P_x(X_{u+\gamma} >0 \text{ for all } u\in(0,\infty))\nonumber\\
&=\E_x( \P_x(X_{u+\gamma} >0 \text{ for all } u\in(0,\infty)|\F_{\gamma}))\nonumber\\
&=\E_x( \P_{X_{\gamma}}( \sigma_0^-=\infty))\nonumber\\
&=\E_x( \P_{X_{\gamma}}( \tau_0^-=\infty))\nonumber\\
&=\E_x(\psi'(0+)W(X_{\gamma})),\label{distributiong}
\end{align}
where we used the tower property of conditional expectation in the second equality, the Markov property of L\'evy processes in the third and that $\sigma_0^-$ and $\tau_0^-$ have the same distribution (see discussion below equation \eqref{eq:definitionofsigmax-}). Note that the law of $g$ under $\P_x$ may have an atom at zero given by
\begin{align*}
\P_x(g=0)=\P_x(\sigma_0^-=\infty)=\P_x(\tau_0^-=\infty)=\psi'(0+)W(x).
\end{align*}

For our optimal prediction problem, we require the $p$-th moment of $g$ to be finite. 
The following result is from \cite{doney2004moments} (see Theorem 1, Theorem 4, Theorem 5 and Remark (ii)).
\begin{lemma}
\label{lemma:finitenessofEgn}
Let $X$ be a spectrally negative L\'evy process drifting to infinity. Then, for a fixed $p > 0$, the following are equivalent:
\begin{enumerate}
\item $\E_{x}(g^{p})<\infty$ for some (hence every) $x\leq 0$;
\item $\int_ {(-\infty,-1)} |x|^{1+p} \Pi(\dd x)<\infty$;
\item $\E((-\underline{X}_{\infty}^p))<\infty$;
\item $\E_x((\tau_0^+)^{p+1})<\infty$ for some (hence every) $x\leq 0$;
\item $\E_x((\tau_0^-)^{p}\I_{\{\tau_0^-<\infty\}})<\infty$ for some (hence every) $x\geq  0$.
\end{enumerate}
\end{lemma}
The next lemma states that when $\tau_0^+$ has finite $p$-th moment under $\P_x$, then the function $ \E_x((\tau_0^+)^p)$ has a polynomial bound in $x$. It will be of use later to deduce a lower bound for the value function of our optimal prediction problem. Its proof can be found in Appendix \ref{sec:Appendix}.
\begin{lemma}
\label{lemma:boundaryformomentsoftau0+p}
Let $p> 0$ and suppose $\E_x((\tau_0^+)^{p+1})<\infty$ for some $x\leq 0$. Then, for each $0\leq r \leq p  $, there exist non-negative constants $A_r$ and $C_r$ such that
\begin{align*}
\E_x( (\tau_0^+)^r)\leq A_r+C_r|x|^{r} \qquad \text{and} \qquad \E_x (g^{r}) \leq 2^{r} [\E(g^{r})+A_{r} ] +2^{r} C_{r} |x|^{r}, \qquad x\leq 0.
\end{align*}
Here $\lfloor p \rfloor$ denotes the integer part of $p$.

\end{lemma}

The following lemma shows some properties of the function $x\mapsto \E_x(g^p)$. The proof is included in Appendix \ref{sec:Appendix}.

\begin{lemma}
\label{lemma:propertiesofExgn}
Let $p>0$ and assume that $\int_{(-\infty,-1)}|x|^{p+1} \Pi(dx)<\infty$. Then $x\mapsto \E_x(g^p)$ is a non-increasing, non-negative and continuous function. Moreover,

\begin{align*}
\lim_{x\rightarrow -\infty} \E_x(g^p)=\infty \qquad \text{and} \qquad \lim_{x\rightarrow \infty} \E_x(g^p)=0.
\end{align*}
\end{lemma}

We state a basic inequality which is used throughout the paper. 

\begin{lemma}
\label{lemma:basicinequality}
Let $q>0$. Then we have that $(a+b)^q \leq 2^q( a^q+b^q)$ for any $a>0$ and $b>0$. 
\end{lemma}
\begin{proof}
Since the function $x\mapsto x^q$ is increasing on $(0,\infty)$, we have that 
\begin{align*}
(a+b)^q\leq (2\max\{a,b \})^q\leq (2\max\{a,b \})^q+(2\min\{a,b \})^q=2^q( a^q+b^q).
\end{align*}
The proof is complete.
\end{proof}
We conclude this section with a technical result extracted from \cite{baurdoux2014predicting} (see Lemma 5)  that will be useful later. 
\begin{lemma}
\label{lemma:technicalresultOS}
Let $X$ be any L\'evy process drifting to $-\infty$. Denote $T_+(0)=\inf\{t\geq 0: X_t\geq 0 \}$. Consider, for $a>0$ and $b<0$, the optimal stopping problem 
\begin{align*}
P(x)=\inf_{\tau \in \mathcal{T}} \E_x[a\tau +\I_{\{\tau \geq T_+(0) \}} b], \qquad \text{for } x\in \R.
\end{align*}
Then there is an $x_0\in (-\infty,0)$ so that $P(x)=0$ for all $x\leq x_0$. 
\end{lemma}

\section{Optimal prediction problem}
\label{sec:optimalpredictionproblem}

Denote by $V_*$ the value of the optimal prediction problem, i.e.,
\begin{align}
\label{eq:poptimalprediction}
V_*=\inf_{\tau \in \mathcal{T}} \E(|\tau-g|^p),
\end{align}
where $\mathcal{T}$ is the set of all stopping times with respect to $\mathbb{F}$, $p> 1$ and $g$ is the last zero of $X$ given in (\ref{eq:lastzeroofX}). Since $g$ is only $\F$ measurable standard techniques of optimal stopping times are not directly applicable. However, there is an equivalence between the optimal prediction problem (\ref{eq:poptimalprediction}) and an optimal stopping problem. The next lemma, inspired by the work of \cite{urusov2005property}, states such equivalence.

\begin{lemma}
\label{lemma:optimalstoppingequivalency}
Let $p>1$ and let $X$ be a spectrally negative L\'evy process drifting to infinity such that $\int_{(-\infty,-1)}|x|^{p+1}\Pi(\dd x)<\infty$. Consider the optimal stopping problem

\begin{align}
\label{eq:equivalentoptimalstoppingproblem}
V=\inf_{\tau \in \mathcal{T}}\E\left( \int_0^{\tau} G(s-g_s,X_s) \dd  s \right),
\end{align}
where the function $G$ is given  by

\begin{align*}
G(u,x)=u^{p-1}\psi'(0+)W(x)-\E_{x}(g^{p-1})
\end{align*}
for $u\geq 0$ and $x\in \R$. Then we have that $V_*=pV+\E(g^{p})$ and a stopping time minimises (\ref{eq:poptimalprediction}) if and only if it minimises (\ref{eq:equivalentoptimalstoppingproblem}).

\end{lemma}
\begin{proof}
Let $\tau \in \mathcal{T}$. Then the following equality holds

\begin{align}
\label{eq:integralrepofabsvalue}
|\tau-g|^p&= \int_0^{\tau}  \varrho(s-g)  \dd s +g^p,
\end{align}
where the function $\varrho$ is defined by

\begin{align*}
\varrho(x)= p\left[ \frac{(-x)^{p}}{x} \I_{\{x< 0 \}} +x^{p-1}\I_{\{x\geq 0 \}}\right].
\end{align*}
 Taking expectations in equation (\ref{eq:integralrepofabsvalue}), using Fubini's theorem and the tower property for conditional expectations we obtain

\begin{align*}
\E(|\tau-g|^p)
&= \int_0^{\infty}  \E\left(\varrho(s-g) \I_{\{ s\leq \tau\}}\dd s\right)+ \E(g^p) \\
&= \int_0^{\infty}  \E\left[\I_{\{ s\leq \tau\}} \E\left(\varrho(s-g) |\F_s\right) \dd s\right]+ \E(g^p)  \\
&=\E\left(\int_0^{\tau} \E\left( \varrho(s-g) |\F_s\right)\dd s \right)+ \E(g^p).
\end{align*}
To evaluate the conditional expectation inside the last integral, note that for all $t\geq 0$ we can write the time $g$ as
\begin{align*}
g=g_t \vee \sup\{s \in (t,\infty): X_s \leq  0 \},
\end{align*}
recalling that $g_t=g_t^{(0)}$ defined in (\ref{eq:gt}). Hence, using the Markov property for L\'evy processes and the fact that $g_s$ is $\F_s$ measurable we have that
\begin{align*}
\E(\varrho(s-g)|\F_s)
&=\E\left(\varrho \left(s- \left[g_s \vee \sup\{r \in (s,\infty): X_{r} \leq  0 \}\right] \right)|\F_s\right)\\
&=\varrho(s-g_s)\E( \I_{\{ X_{r} >  0 \text{ for all } r\in (s,\infty) \}} |\F_s)\\
&\qquad+\E(\varrho(s-\sup\{r \in (s,\infty): X_{r} \leq  0 \})  \I_{\{ X_{r} \leq  0 \text{ for some } r\in (s,\infty) \}}|\F_s )\\
&=\varrho(s-g_s)\P_{X_s}(g=0 )+\E_{X_s}(\varrho(-g) \I_{\{g>0\}})\\
&=p(s-g_s)^{p-1}\psi'(0+)W(X_s)-p\E_{X_s}(g^{p-1} ).
\end{align*}
Then we have that

\begin{align*}
\E(|\tau-g|^p)=p\E\left( \int_0^{\tau} G(s-g_s,X_s)\dd s\right)+\E(g^p).
\end{align*}

\end{proof}

\begin{rem}
A close inspection of the proof of Lemma \ref{lemma:optimalstoppingequivalency} tells us that the function $\varrho$ corresponds to the right derivative of the function $f(x)=|x|^p$. Therefore, using similar arguments we can actually extend the result to any convex function $d:\R_+ \times \R_+ \mapsto \R_+$. That is, under the assumption that $\E(d(0,g))<\infty$, the optimal prediction problem
\begin{align*}
V_d=\inf_{\tau \in \mathcal{T}} \E(d(\tau,g))
\end{align*}
is equivalent to the optimal stopping problem
\begin{align*}
\inf_{\tau \in \mathcal{T}} \E\left[ \int_0^{\tau} G_d(g_{s},s,X_{s}) \dd s \right],
\end{align*}
where $G_d(\gamma,t,x)=\varrho_d(s,\gamma)\psi'(0+)W(x)+\E_{x}(\varrho_d(s,g+s) \I_{\{g>0\}})$ and $\varrho_d$ is the right derivative with respect the first argument of $d$.

\end{rem}
The next theorem states the solution to the optimal prediction problem. Note that its proof is rather lengthy so the next section is entirely dedicated to that purpose.

\begin{thm}
\label{thm:solutiontotheoptimalstoppingproblem}
Let $p>1$ and let $X$ be a spectrally negative L\'evy process drifting to infinity such that $\Pi$ has no atoms and that $\int_{(-\infty,-1)}|x|^{p+1}\Pi(\dd x)<\infty$. Then there exists a non-decreasing and continuous function $b:(0,\infty)\mapsto [0,\infty)$ such that $b(u)\geq h(u):=\inf\{x\in \R: G(u,x)\geq 0 \}$ for all $u\geq 0$, $\lim_{u\downarrow 0} b(u)=\infty$, $\lim_{u\rightarrow \infty} b(u)=0$ and the infimum in (\ref{eq:equivalentoptimalstoppingproblem}) (and hence in (\ref{eq:poptimalprediction})) is attained by 
\begin{align}
    \tau_D=\inf\{ t>0: X_t\geq b(U_t)\}.
\end{align}
\end{thm}

Moreover, the function $b$ is uniquely characterised as in Theorem \ref{thm:characterisationofbandV}.

\section{Solution to the optimal stopping problem}
\label{sec:Solutionoptimalstopping}
Throughout this section we are going to assume that $p>1$ and that $X$ is a spectrally negative L\'evy process drifting to infinity such that $\Pi$ has no atoms and $\int_{(-\infty,-1)}|x|^{p+1}\Pi(\dd x)<\infty$. To solve the optimal stopping problem (\ref{eq:equivalentoptimalstoppingproblem}) using the general theory of optimal stopping (see e.g. \cite{peskir2006optimal}), we have to extend it to an optimal stopping problem driven by a strong Markov process. For every $(u,x)\in E$, we define the optimal stopping problem
\begin{align}
\label{eq:optimalstoppingproblem}
V(u,x)=\inf_{\tau \in \mathcal{T}} \E_{u,x}\left[ \int_0^{\tau} G(U_s,X_{s})\dd s \right],
\end{align}
where the function $G$ is given by $G(u,x)= u^{p-1} \psi'(0+)W(x)-\E_{x}(g^{p-1} )$ for any $u\geq 0$ and $x\in \R$. Therefore we have that $V_*=pV(0,0)+\E(g^p)$. Note that using the definition of $\E_{u,x}$ we have that (\ref{eq:optimalstoppingproblem}) takes the form
\begin{align}
\label{eq:VintermsofPx}
V(u,x)=\inf_{\tau \in \mathcal{T}}\E_x\left( \int_0^{\tau} \left\{ G(u+s,X_s)\I_{\{\sigma_0^->s\}}+G(U_s,X_s)\I_{\{\sigma_0^-\leq s\}} \right\} \dd s\right).
\end{align}

As a consequence of Lemma \ref{lemma:propertiesofExgn} we have the following behaviour of the function $G$. For all $x\in \R$, the function $u \mapsto G(u,x)$ is non-decreasing. In particular, when $x< 0$, $u \mapsto G(u,x)=-\E_x(g^{p-1})$ is a strictly negative constant. For fixed $u\geq 0$, $x\mapsto G(u,x)$ is a non-decreasing right-continuous function  which is continuous everywhere apart from possibly at $x=0$ (since $W$ is discontinuous at zero when $X$ is of finite variation) such that for all $u\geq 0$,

\begin{align*}
\lim_{x \rightarrow -\infty} G(u,x)=-\infty \qquad \text{ and } \qquad \lim_{x \rightarrow \infty} G(u,x)=u^{p-1} \geq 0.
\end{align*}
Moreover, we have that $\lim_{u \rightarrow \infty} G(u,x)=\infty$ and $G(0,x)=-\E_x(g^{p-1})<0$ for all $x\geq 0$. Recall that for any $u\geq 0$,
\begin{align}
\label{eq:functionh}
h(u)&=\inf\{x\in \R: G(u,x)\geq 0 \}.
\end{align}
From the description of $G$ above we have that $h$ is a non-negative and non-increasing function such that $h(u)<\infty$ for all $u\in (0,\infty)$, $h(0)=\infty$ and $\lim_{u\rightarrow \infty}h(u)=0$. Moreover, since $W$ is strictly increasing on $(0,\infty)$, the function
\begin{align*}
T(x):= \frac{\E_x(g^{p-1})}{\psi'(0+)W(x)}
\end{align*}
is continuous and strictly decreasing on $[0,\infty)$. Then, there exists an inverse function $T^{-1}$ which is continuous and strictly decreasing on $(0, u_h^*]$ with
\begin{align}
\label{eq:ustar}
u_h^*:=\frac{\E(g^{p-1})}{\psi'(0+)W(0)},
\end{align}
where we understand $1/0=\infty$ when $X$ is of infinite variation. Hence, we can write
\begin{align*}
h(u)= \left\{
\begin{array}{lr}
T^{-1}(u^{p-1}), & u< (u_h^*)^{\frac{1}{p-1}},\\
0, & u\geq (u_h^*)^{\frac{1}{p-1}}.
\end{array}
\right. 
\end{align*}
Therefore, since $T^{-1}(u_h^*-)=0$, we conclude that $h$ is a continuous function on $[0,\infty)$. From the definition of $h$ we clearly have that $G(u,x)\geq 0$ if and only if $x\geq h(u)$.\\

 The facts above give us some intuition about the optimal stopping rule for the optimal stopping problem (\ref{eq:optimalstoppingproblem}). Since we are dealing with a minimisation problem, before stopping, we want the process $(U, X)$ to be in the set in which $G$ is negative as much as possible. Then, the fact that $G(U_t, X_t)$ is strictly negative when $X_t<h(U_t)$ suggests that it is never optimal to stop on this region. When $X_t>h(U_t)$, we have that $G(U_t,X_t)\geq 0$ but with strictly positive probability $(U,X)$ can enter the set in which $G$ is negative. Moreover, $t\mapsto U_t$ is strictly increasing when $X$ is in the positive half line so that $t\mapsto h(U_t)$ gets closer to zero when the current excursion away from $(-\infty,0]$ is sufficiently large. Then, $G(U_t,X_t)\geq 0$ even when $X_t$ is relatively close to zero. That suggests that stopping is optimal when the current excursion away from $(-\infty,0]$ is large, or $X$ takes sufficiently large values. Then we infer the existence of a non-negative curve $b \geq h$ such that it is optimal to stop when $X$ crosses above $b(U_t)$. We will formally show in the next Lemmas the existence of such a boundary. \\

Note that if there exists a stopping time $\tau$ for which the expectation of the right-hand side of (\ref{eq:optimalstoppingproblem}) is minus infinity, then $V$ would also be minus infinity. The following Lemma provides the finiteness of a lower bound of $V$ that will ensure that $V$ only takes finite values. Its proof is included in Appendix \ref{sec:Appendix}.

\begin{lemma}
\label{lemma:finitenessofintExgn}
We have that

\begin{align*}
0\leq \E_x\left( \int_0^{\infty} \E_{X_s} (g^{p-1}) \dd s\right)<\infty \qquad \text{for all } x\in \R.
\end{align*}
\end{lemma}
We now prove the finiteness of the function $V$.
\begin{lemma}
\label{lemma:finitnessofVandnegativevaluesh}
For every $(u,x)\in E$ we have that $V(u,x) \in (-\infty,0]$. In particular, $V(u,x)<0$ for $(u,x)\in B:=\{(u,x)\in E: x<h(u)\} $, where $h$ is defined in (\ref{eq:functionh}).
\end{lemma}

\begin{proof}
By taking the stopping time $\tau =0$ we deduce that for all $(u,x)\in E$, $V(u,x)\leq 0$. In order to check that $V(u,x)>-\infty$ we use that $G(u,x)\geq -\E_x(g^{p-1})$ to get
\begin{align*}
V(u,x)=\inf_{\tau \in \mathcal{T}} \E_{u,x}\left[ \int_0^{\tau} G(U_s,X_{s})\dd s \right]\geq - \sup_{\tau \in \mathcal{T}} \E_{x}\left[ \int_0^{\tau} \E_{X_s}(g^{p-1}) \dd s \right],
\end{align*}
for all $(u,x)\in E$. Hence by Lemma \ref{lemma:finitenessofintExgn} we have that
\begin{align}
\label{eq:lowerboundforVintermsofgp}
V(u,x) \geq -  \E_{x}\left[ \int_0^{\infty} \E_{X_s}(g^{p-1})\dd s \right]>-\infty
\end{align}
for all $(u,x)\in E$. Using standard arguments we can prove that $V(u,x)<0$ when $(u,x)\in B$. Indeed, from the definition of $h$, we have that if $(u,x) \in B$ then $G(u,x)<0$. Take $(u,x)\in B$ and consider the stopping time
\begin{align*}
\tau_{B}:=\inf\{t\geq 0: (U_t,X_t) \in E \setminus B \}.
\end{align*}
Note that under the measure $\P_{u,x}$ we have $\tau_{B}>0$. Then, for all $s<\tau_{B}$ we have that $(U_s, X_s) \in B$ which implies that $G(U_s,X_s)<0$. Hence, by the definition of $V$, we see that
\begin{align*}
V(u,x)\leq  \E_{u,x}\left[ \int_0^{\tau_{B}} G(U_s,X_{s})\dd s \right]<0.
\end{align*}
\end{proof}
\begin{rem}
\label{rem:boundsforV}
Note that we have that $h(0)=\infty$, which implies that $(0,0)\in B$ and then, from the Lemma above, $V(0,0)<0$. Moreover, from Lemma \ref{lemma:optimalstoppingequivalency} we have that $pV(0,0)+\E(g^{p-1})=V_* \geq 0$ which implies that 
\begin{align*}
-\frac{\E(g^{p-1})}{p} \leq V(0,0)< 0.
\end{align*}
\end{rem}

Now we prove some basic properties of $V$.
\begin{lemma}
\label{lemma:propertiesofVgeneralOS}
We have the following monotonicity property of $V$. For all $(u,x), (v,y)\in E$ such that $u\leq v$ and $x\leq y$ we have that $V(u,x)\leq V(v,y)$.
\end{lemma}

\begin{proof}
From equation (\ref{eq:VintermsofPx}) we have that
\begin{align*}
V(u,x)&= \inf_{\tau \in \mathcal{T}}\E_x\left( \int_0^{\tau} \left\{ G(u+s,X_s)\I_{\{\sigma_0^->s\}}+G(U_s,X_s)\I_{\{\sigma_0^-\leq s\}} \right\}\dd s\right) \\
&= \inf_{\tau \in \mathcal{T}}\E\left( \int_0^{\tau} \left\{ G(u+s,X_s+x)\I_{\{\sigma_{-x}^->s\}}+G(U_{s}^{(-x)},X_s+x)\I_{\{\sigma_{-x}^-\leq s\}} \right\}\dd s\right),
\end{align*}
where $\sigma_{-x}^-=\inf \{t\geq 0: X_t\leq  -x\}$ and $U_s^{(-x)}=s-\sup\{0\leq t\leq s: X_t\leq  -x \}$. Recall that for all $s\geq 0$, $x\mapsto U_s^{(-x)}$ and $x\mapsto \sigma_{-x}^-$ are non-decreasing and that the function $G$ is non-decreasing in each argument. Define the function
\begin{align*}
 G^*(u,x):=G(u+s,X_s+x)\I_{\{\sigma_{-x}^->s \}}+G(U_s^{(-x)},X_s+x)\I_{\{\sigma_{-x}^-\leq s \}}.
\end{align*}
We show by cases that the function $G^*$ is non-decreasing in each argument. Take $ x\leq y$ and $0\leq u\leq v$. First, we suppose that $\omega \in  \{\sigma_{-x}^->s\} \subset \{\sigma_{-y}^->s\}$. Since $ G$ is non-decreasing in each argument we then have
\begin{align*}
 G^*(u,x)(\omega)=G(u+s,X_s(\omega)+x)\leq  G(v+s,X_s(\omega)+y)=G^*(v,y)(\omega).
\end{align*}
Similarly, if $\omega \in  \{\sigma_{-x}^-\leq s\}\cap \{\sigma_{-y}^- \leq s\}$ we have that
\begin{align*}
G^*(u,x)(\omega)=G(U_s^{(-x)}(\omega),X_s(\omega)+x)\leq G(U_s^{(-y)}(\omega),X_s(\omega)+y)=G^*(v,y)(\omega).
\end{align*}
Lastly, take $\omega \in  \{\sigma_{-x}^-\leq s\} \cap \{\sigma_{-y}^->s\}$. Then using the fact that $U_s^{(-x)}=s-g_s^{(-x)}\leq s \leq v+s$ and the monotonicity of $G$ we get
\begin{align*}
G^*(u,x)(\omega)=G(U_s^{(-x)}(\omega),X_s(\omega)+x) \leq  G(v+s,X_s(\omega)+y)=G^*(v,y)(\omega).
\end{align*}
All this together implies that the function $ G^*(u,x)$ is non-decreasing in each argument for all $u\geq 0$ and $x\in \R$, in particular for all $(u,x)\in E$. Hence, the claim on $V$ holds.
\end{proof}
In the following Lemma, we give an expression for $V(0,x)$ when $x<0$, in terms of $V(0,0)$, and we use it to give a lower bound for $V$.
\begin{lemma}
For any $x\leq 0$ we have that 
\begin{align}
V(0,x)&=\E_{x}\left( \int_0^{ \tau_0^+} G(0,X_s) \dd s\right)+V(0,0)\nonumber\\
\label{eq:expressionforV0xnegative}
&=-\int_{0}^{-x} \int_{0}^{\infty}  \E_{-u-z}(g^{p-1})W'(u)\dd u  \dd z +V(0,0).
\end{align}
Moreover, for all $(u,x)\in E$ we have that there exist non-negative constants $A'_{p-1}$ and $C'_{p-1}$ such that
\begin{align}
\label{eq:lowerboundforV}
V(u,x)\geq -A'_{p-1}-C'_{p-1} |x|^p+V(0,0).
\end{align}
\end{lemma}
\begin{proof}
Let $x<0$, using the Markov property and a dynamic programming argument we can write for all $x<0$,

\begin{align*}
V(0,x)&=\inf_{\tau \in \mathcal{T}}  \E_{x}\left( \int_0^{\tau \wedge \tau_0^+} G(0,X_s) \dd s +\I_{\{\tau_0^+<\tau \}} \int_{\tau_0^+}^{\tau} G(U_s,X_s)\dd s  \right)\\
&=\inf_{\tau \in \mathcal{T}}  \E_{x}\left( \int_0^{\tau \wedge \tau_0^+} G(0,X_s) \dd s +\I_{\{\tau_0^+<\tau \}} V(0,0)   \right)\\
&=  \E_{x}\left( \int_0^{\tau_0^+} G(0,X_s) \dd s     \right)+V(0,0),
\end{align*}
where the last equality follows since $G(0,x)\leq 0$ for all $x\leq 0$ and $V(0,0)\leq 0$, and hence the infimum is attained for any $\tau \geq \tau_0^+$. Using the fact that $G(0,x)=-\E_x(g^{p-1})$, for all $x<0$, and Fubini's theorem we get that
\begin{align*}
V(0,x)&=-\E_x\left(\int_0^{\tau_0^+} \E_{X_s}(g^{p-1})\dd s\right)+V(0,0)\\
&=-\int_{(-\infty,0)} \E_z(g^{p-1})\int_0^{\infty} \P_x(X_s\in \dd z,s<\tau_0^+)\dd s +V(0,0).
\end{align*}
Using the $0$-potential measure of $X$ killed on exiting the interval $(-\infty,0]$ (see equation (\ref{eq:qpotentialdensitytkillingonexitinga})) and Fubini's theorem, we obtain that
\begin{align*}
V(0,x)&=-\int_0^{\infty} \E_{-z}(g^{p-1})[W(z)-W(x+z)] \dd z +V(0,0)\nonumber\\
&=-\int_0^{\infty} \E_{-z}(g^{p-1}) \int_{x+z}^{z}W'(u)\dd u \dd z +V(0,0)\nonumber\\
&=- \int_{0}^{\infty} W'(u) \dd u \int_{u}^{u-x} \E_{-z}(g^{p-1}) \dd z +V(0,0)\nonumber \\
&=-\int_{0}^{-x} \int_{0}^{\infty}  \E_{-u-z}(g^{p-1})W'(u) \dd u  \dd z +V(0,0).
\end{align*}
From equation (\ref{eq:expressionforV0xnegative}) and the fact that $x\mapsto \E_x(g^{p-1})$ is non-increasing and bounded from above by a polynomial (see Lemmas \ref{lemma:boundaryformomentsoftau0+p} and \ref{lemma:propertiesofExgn}) we have the inequalities for $x<0$,
\begin{align*}
 V(0,x) &\geq x \int_{0}^{\infty}  \E_{x-u}(g^{p-1})W'(u) \dd u  +V(0,0) \\
 &\geq \frac{1}{\psi'(0+)}  2^{p-1}[\E(g^{p-1})+A_{p-1} ]x+\frac{1}{\psi'(0+)}2^{p-1} C_{p-1} x \E(|x+\underline{X}_{\infty}|^{p-1})\\
 &\qquad +V(0,0) \\
  &\geq  \frac{1}{\psi'(0+)} 2^{p-1}[\E(g^{p-1})+A_{p-1}+2^{p-1} C_{p-1} \E((-\underline{X}_{\infty})^{p-1}) ]x\\
  &\qquad -\frac{1}{\psi'(0+)} 2^{p-1} C_{p-1} |x|^p+V(0,0),
 \end{align*}
where we used that $\P(-\underline{X}_{\infty} \in \dd u)=\psi'(0+)W'(u) \dd u$. Hence (\ref{eq:lowerboundforV}) follows for $x<0$. The general statement holds since $V$ is non-decreasing in each argument.
\end{proof}
Define the set $D:=\{(u,x)\in E: V(u,x)=0 \}$. From Lemma \ref{lemma:finitnessofVandnegativevaluesh} we know that $V(u,x)<0$ for all $(u,x)\in E$ such that $x<h(u)$. Hence if $(u,x)\in D$ we have that $x\geq h(u)\geq 0$. We then define the function $b:(0,\infty) \mapsto \R$ by
\begin{align*}
b(u)=\inf\{x>0 : V(u,x)=0 \},
\end{align*}
where $\inf \emptyset=\infty$ and $\inf (0,\infty)=0$. Then it directly follows that $b(u)\geq h(u)\geq 0$ for all $u>0$. Moreover, since $h(0)=\infty$ we have that $\lim_{u\downarrow 0} b(u)=\infty$. Furthermore, since $V$ is monotone in each argument we deduce that $u\mapsto b(u)$ is non-increasing and $V(u,x)=0$ for all $x>b(u)$.	We then have the following Lemma.

\begin{lemma}
\label{lemma:basicpropertiesofb}
The function $b:\R_+ \mapsto \R$ is non-increasing with $0\leq h(u) \leq b(u)$. We have that $\lim_{u \downarrow 0}b(u)=\infty$ and $b(u)<\infty$ for all $u>0$. 
\end{lemma}

\begin{proof}
%
We show that for each $u>0$, $b(u)<\infty$. Fix $u>0$ and take $x>y>0$. By a dynamic programming argument we obtain that 
\begin{align*}
V(&u,x)\\
&=\inf_{\tau \in \mathcal{T}} \E_{u,x} \left( \int_0^{\tau \wedge \sigma_y^-} G(u+s,X_s)\dd s+\I_{\{ \sigma_y^-<\tau \} }V(U_{\sigma_y^-}, X_{\sigma_y^-} )\right)\\
&\geq \inf_{\tau \in \mathcal{T}} \E_{x} \left( \int_0^{\tau \wedge \sigma_y^-} G(u+s,X_s)\dd s+\I_{\{ \sigma_y^-<\tau \}}V(0,0 )+\I_{\{ \sigma_y^-<\tau, X_{\sigma_y^-}\leq 0  \} }V(0, X_{\sigma_y^-} )\right),
\end{align*}
where the inequality follows since $V$ is non-positive and non-decreasing. By the compensation formula for Poisson random measures (see \eqref{eq:compensationformulaPRM}), we have that for any stopping time $\tau$ (we assume without loss of generality that $\tau<\infty$ a.s.),
\begin{align*}
 \E_{x} &\left( \I_{\{ \sigma_y^-<\tau, X_{\sigma_y^-}<0  \} }V(0, X_{\sigma_y^-} )\right)\\
 & \qquad= \E_{x} \left( \int_{0}^{\infty} \int_{(-\infty,0)} V(0, X_{s-}+z ) \I_{\{X_{s-}+z\leq  0 \}}\I_{\{s\leq \tau\wedge \sigma_y^- \}} N(\dd s,\dd z) \right)\\
 & \qquad =\E_{x} \left( \int_{0}^{\tau\wedge \sigma_y^- } \int_{(-\infty,0)} V(0, X_{s}+z ) \I_{\{X_{s}+z\leq  0 \}} \Pi(\dd z) \dd s \right).
\end{align*}
Hence, from the equation above, since $G$ and $V$ are non-decreasing in each argument, $V\leq 0$ and $X_s> y$ for all $s<\sigma_y^-$ we have that 
\begin{align*}
V(u,x)\geq \inf_{\tau \in \mathcal{T}} \E_{x} \left( (\tau \wedge \sigma_y^-)\left[ G(u,y)+  \int_{(-\infty,-y)} V(0,z )\Pi(\dd z) \right]+  \I_{\{ \sigma_y^-<\tau \}}V(0,0 )\right).
\end{align*}
Note that from equation (\ref{eq:lowerboundforV}) and Lemma \ref{lemma:finitenessofEgn}, the integral with respect to $\Pi(\dd z)$ above is finite so we can choose $y$ sufficiently large such that $a:=G(u,y)+  \int_{(-\infty,-y)} V(0,z )\Pi(\dd z)\geq 0$. Take any stopping time $\tau \in \mathcal{T}$, since $\tau \wedge \sigma_y^-$ is also a stopping time we have that
\begin{align*}
\inf_{\tau \in \mathcal{T}} \E_{x}& \left(a \tau +  \I_{\{ \tau \geq \sigma_y^-  \}}V(0,0 )\right)\\
&\leq \E_{x} \left( a(\tau \wedge \sigma_y^-)+\I_{\{ \tau \wedge \sigma_y^- \geq \sigma_y^-  \}}V(0,0 )\right)\\
&\leq \E_{x} \left( a(\tau \wedge \sigma_y^-)+\I_{\{ \tau > \sigma_y^-  \}}V(0,0 )\right),
\end{align*}
where the last inequality follows since $V(0,0)\leq 0$. Hence, we deduce that for $x>y>0$ and $u>0$ sufficiently large,
\begin{align*}
V(u,x)&\geq \inf_{\tau \in \mathcal{T}} \E_{x} \left(a (\tau \wedge \sigma_y^-)+  \I_{\{ \sigma_y^-<\tau \}}V(0,0 )\right)\\
&\geq \inf_{\tau \in \mathcal{T}} \E_{x}\left(a \tau +  \I_{\{ \tau \geq \sigma_y^-  \}}V(0,0 )\right).
\end{align*}

Then from Lemma \ref{lemma:technicalresultOS} we have that (since $V(0,0)\leq 0$ and $-X$ drifts to $-\infty$) there exists a value $x_0(u)<0$ such that the right-hand side of the equation above vanishes for all $y-x \leq x_0(u)$. Hence, we have that $V(u,x)=0$ for all $x\geq y-x_0(u)$ and then $b(u)<\infty$.

\end{proof}
Let $(u,x)\in E$. We define, under the measure $\P_{u,x}$, the stopping times
\begin{align}
\tau_D&=\inf\{t\geq 0: (U_t,X_t)\in D \}=\inf\{ t\geq 0 : X_t \geq b(U_t) \}\nonumber,\\
\label{eq:auxiliarystoppingtimetaugamma}
\tau_b^{v,y}&=\inf\{t>0: X_t+y\geq b(v+t) \}, \qquad v>0 \text{ and } y\in \R,
\end{align}
and for any $x\in \R$, under the measure $\P_x$, the stopping time
\begin{align}
\label{eq:auxiliarystoppingtimetaug}
\tau_b^{g,y}&=\inf\{t>0: X_t+y\geq b(U_t^{(-y)}) \}, \qquad  \text{ } y\in \R.
\end{align}
Note that for any $y\in \R$ and $v>0$, the stopping time $\tau_b^{v,y}$ does not depend on the process $U$ and hence for any measurable function $f$, we have that $\E_{u,x}(f(\tau_b^{v,y}))=\E_x(f(\tau_b^{v,y}))$. Hence, for any $(u,x)\in E$ and any measurable function $f$, it can be seen that 
\begin{align*}
\E_{u,x}(f(\tau_D))
&=\E_{x}( f(\tau_b^{u,0}) \I_{\{ \tau_b^{u,0}\leq \sigma_0^- \}})+\E_{x}(f(\tau_b^{g,0}) \I_{\{ \tau_b^{u,0}> \sigma_0^- \}}).
\end{align*}
Now we introduce a technical lemma that ensures that the stopping time $\tau_D$ has moments of order $p$. The proof can be found in Appendix \ref{sec:Appendix}.

\begin{lemma}
\label{lemma:finitenesofEtauD}
For all $(u,x)\in E$ we have that
\begin{align*}
\E_{u,x}((\tau_{D})^p)<\infty.
\end{align*}
\end{lemma}

Now we are ready to show (using the general theory of optimal stopping) that $\tau_D$ is an optimal stopping time for \eqref{eq:optimalstoppingproblem}.
 	\begin{lemma}
An optimal stopping time for (\ref{eq:optimalstoppingproblem}) is given by $\tau_D$, the first entrance of $(U,X)$ to the closed set $D$, i.e.,
\begin{align*}
\tau_D=\inf\{ t\geq 0: (U_t,X_t)\in D\}.
\end{align*}
Then the function $V$ takes the form
\begin{align*}
V(u,x)=\E_{u,x}\left( \int_0^{\tau_D} G(U_s,X_s)\dd s\right), \qquad (u,x) \in E.
\end{align*}
\end{lemma}
\begin{proof}
Note that it follows from Lemma \ref{lemma:finitenesofEtauD} that $\P_{u,x}(\tau_D<\infty)=1$ for all $(u,x)\in E$. Then, using a dynamic programming argument, we deduce that 
\begin{align*}
V(u,x)
&=\inf_{\tau \in \mathcal{T}} \E_{u,x}\left( \int_0^{\tau\wedge \tau_D }G(U_s,X_s) \dd s  + \I_{\{\tau_D <\tau \}} V(U_{\tau_D}, X_{\tau_D} ) \right)\\
&=\inf_{\tau \in \mathcal{T}} \E_{u,x}\left( \int_0^{\tau\wedge \tau_D }G(U_s,X_s) \dd s   \right),
\end{align*}
where in the last equality we used that $V(u,x)=0$ on $D$. \\

Since $W(x)\leq 1/\psi'(0+)$ for all $x\in \R$, we have that $|G(u,x)|\leq u^{p-1}+\E_x(g^{p-1})$. Then, for any $(u,x)\in E$, we deduce that
\begin{align*}
\E_{u,x}&\left[ \sup_{t\geq 0}\left| \int_0^{t\wedge \tau_D} G(U_s,X_s) \right| \dd s \right]\\
&\qquad\leq \E_{u,x}\left[ \int_0^{\tau_D} [(U_s)^{p-1} +\E_{X_s} (g^{p-1}) \dd s \right]	\\
&\qquad\leq \E_{u,x}\left(  \int_0^{\tau_D} (u+s)^{p-1}\dd s\right)+ \E_{u,x}\left( \int_0^{\tau_D}  \E_{X_s} (g^{p-1}) \dd s\right)\\
&\qquad\leq 2^{p-1}[u^{p-1}+\frac{1}{p}\E_{u,x}[(\tau_D)^p]]+\E_x\left( \int_0^{\infty}  \E_{X_s} (g^{p-1}) \dd s\right)\\
&\qquad <\infty,
\end{align*}
where in the second inequality we used that $U_s\leq u+s$ for all $s\geq 0$, under the measure $\P_{u,x}$, and on the last equality follows from Lemmas \ref{lemma:finitenessofintExgn} and \ref{lemma:finitenesofEtauD}. \\

From equation (\ref{eq:expressionforV0xnegative}) we have that $V(0,x)$ is continuous on $(-\infty,0)$ and left continuous at $0$ (and hence upper semi-continuous on $(-\infty,0)$). Next, we show that $V$ is upper semi-continuous on $([0,\infty)\times [0,\infty))\cap E$. Note that since $V$ is non-decreasing in each argument, we have that for any $u>0$ and $x>0$ or $u=0=x$,
\begin{align*}
\limsup_{ E\ni (v,y)\rightarrow (u,x)} V(v,y) &=\lim_{\varepsilon \downarrow 0}[\sup\{ V(v,y): (v,y)\in  \mathbf{B}((u,x),\varepsilon) \cap E \}]\\
&\leq \lim_{\varepsilon \downarrow 0} V(u+\varepsilon,x+\varepsilon),
\end{align*}
where $\mathbf{B}((u,x),\varepsilon)$ is the ball with center $(u,x)$ with radius $\varepsilon$. Hence, it suffices to show that 
\begin{align*}
\lim_{\varepsilon \downarrow 0} V(u+\varepsilon,x+\varepsilon) \leq V(u,x)
\end{align*}
for all $(u,x)\in ([0,\infty)\times [0,\infty))\cap E$. Take $\delta>0$ and define the stopping time
\begin{align*}
\tau_{\delta}(x):=\inf\{t\geq \delta: X_t\geq b(\delta)+x \}=\inf\{t\geq 0: X_{t+\delta}\geq b(\delta)+x \}+\delta.
\end{align*}
Note that, by conditioning at the filtration at time $\delta$ and from Lemmas \ref{lemma:boundaryformomentsoftau0+p} and \ref{lemma:basicinequality}, for any $x\in \R$,
 
\begin{align*}
\E_x(\tau_{\delta}(0)^p)=\E_x( \E_{X_{\delta}}[ (\tau_{b(\delta)}^++\delta)^p])\leq 2^p( A_p+\delta) +2^pC_p \E(|X_{\delta}+x -b(\delta)|^p)<\infty,
\end{align*}
where $A_p$ and $C_p$ are non-negative constants and the last inequality follows from Theorem 3.8 in \cite{kyprianou2014fluctuations} (since $x^p\vee 1$ is sub-multiplicative), and by assumption $\int_{(-\infty,-1)} |y|^{p+1}\Pi(\dd y)<\infty$. Using once again a dynamic programming argument we have that for any $u>0$ and $x>0$,
\begin{align}
&V(u,x)\nonumber\\
&=\inf_{\tau \in \mathcal{T}} \E_{u,x} \bigg( \int_0^{\tau \wedge \sigma_{0}^- \wedge \tau_{\delta}(0)} G(U_s,X_s)\dd s+\I_{\{ \sigma_{0}^-<\tau \} }\I_{\{ \sigma_{0}^-<\tau_{\delta}(0) \} }V(0, X_{\sigma_{0}^-} )\nonumber  \\
&\qquad\qquad\qquad \qquad+ \I_{\{ \tau_{\delta}(0)<\tau \} }\I_{\{ \tau_{\delta}(0)<\sigma_{0}^- \} }V(U_{\tau_{\delta}(0) }, X_{\tau_{\delta}(0) } )\bigg)\nonumber\\
\label{eq:representationforVtoshowusc}
&=\inf_{\tau \in \mathcal{T}} \E \bigg( \int_0^{\tau \wedge \sigma_{-x}^- \wedge \tau_{\delta}(-x)} G(u+s,X_s+x)\dd s+\I_{\{ \sigma_{-x}^-<\tau \} }\I_{\{ \sigma_{-x}^-<\tau_{\delta}(-x) \} }V(0, X_{\sigma_{-x}^-}+x ) \bigg),
\end{align}
where in the last equality we used that, on the event $\{\tau_{\delta}(0)<\sigma_{0}^-\}$, we have $(U_{\tau_{\delta}(0)} , X_{\tau_{\delta}(0)} )=(u+\tau_{\delta}(0),X_{\tau_{\delta}(0) })\in D$ due to the fact that $u+\tau_{\delta}(0)\geq u+\delta\geq \delta$ and then $X_{\tau_{\delta}(0) }\geq b(\delta) \geq b(u+\tau_{\delta}(0))$. Recall that $G(u,x)\leq u^{p-1}$ and note that for any $\tau \in \mathcal{T}$, 
\begin{align*}
\E\left( \int_0^{\tau \wedge \sigma_{-x}^- \wedge \tau_{\delta}(-x)} (u+s)^{p-1}\dd s\right)&\leq \frac{1}{p}[\E( (\tau_{\delta}(-x)+u)^p-u^p]<\infty.
\end{align*}
Hence, by Fatou's Lemma, we see that for any $(u,x)\in ([0,\infty)\times [0,\infty))\cap E$ and $\tau \in \mathcal{T}$,
\begin{align*}
&\lim_{\varepsilon\downarrow 0} V(u+\varepsilon,x+\varepsilon)\\
&\leq \limsup_{\varepsilon \downarrow 0}  \E \bigg( \int_0^{\tau \wedge \sigma_{-x-\varepsilon}^- \wedge \tau_{\delta}(-x-\varepsilon)} G(u+\varepsilon+s,X_s+x+\varepsilon)\dd s\\
&\qquad\qquad\qquad+\I_{\{ \sigma_{-x+\varepsilon}^-<\tau \} }\I_{\{ \sigma_{-x-\varepsilon}^-<\tau_{\delta}(-x-\varepsilon) \} }V(0, X_{\sigma_{-x-\varepsilon}^-}+x+\varepsilon ) \bigg)\\
&\leq   \E \bigg(  \int_0^{\tau } \limsup_{\varepsilon \downarrow 0} G(u+\varepsilon+s,X_s+x+\varepsilon)\I_{\{s< \sigma_{-x-\varepsilon}^-\}} \I_{\{ s<\tau_{\delta}(-x-\varepsilon)\}}\dd s\\
&\qquad\qquad\qquad+\limsup_{\varepsilon \downarrow 0}\I_{\{ \sigma_{-x-\varepsilon}^-<\tau \} }\I_{\{ \sigma_{-x-\varepsilon}^--\tau_{\delta}(-x-\varepsilon)<0 \} }V(0, X_{\sigma_{-x-\varepsilon}^-}+x+\varepsilon ) \bigg).
\end{align*}
It is easy to show that for any $x\in \R$, we have $ \tau_{\delta}(-x-\varepsilon) \uparrow \tau_{\delta}(-x)$ and $\sigma_{-x-\varepsilon}^-\downarrow \tau_{-x}^-=\sigma_{-x}^- $ a.s., whenever $\varepsilon \downarrow 0$ (so that $\sigma_{-x-\varepsilon}^--\tau_{\delta}(-x-\varepsilon) \downarrow \sigma_{-x}^--\tau_{\delta}(-x)$ a.s.). Hence, by continuity of $y\mapsto V(0,y)$, right-continuity of $X$ and of the mappings $u,x\mapsto G(u,x)$ and $y\mapsto \I_{\{y<T \}}$, for any $T\in \R$, we have that   
\begin{align*}
&\lim_{\varepsilon\downarrow 0} V(u+\varepsilon,x+\varepsilon)\\
&\leq   \E \bigg(  \int_0^{\tau \wedge \sigma_{-x}^- \wedge \tau_{\delta}(-x)}  G(u+s,X_s+x)\dd s
+\I_{\{ \sigma_{-x}^-<\tau \} }\I_{\{ \sigma_{-x}^-<\tau_{\delta}(-x) \} }V(0, X_{\sigma_{-x}^-}+x) \bigg).
\end{align*}
Thus, taking infimum over all $\tau \in \mathcal{T}$ in the equation above and from \eqref{eq:representationforVtoshowusc} we deduce that
\begin{align*}
\lim_{\varepsilon\downarrow 0} V(u+\varepsilon,x+\varepsilon)\leq V(u,x) 
\end{align*}
for any $(u,x)\in ([0,\infty)\times [0,\infty))\cap E$. Therefore, we conclude that $V$ is upper semi-continuous, and from general results of optimal stopping (see e.g. Corollary 2.9 in \cite{peskir2006optimal}), we deduce that $\tau_D$ is an optimal stopping time for $V$. The proof is complete.

\end{proof}

Using the fact that $\tau_D$ is optimal, we can then give a representation of $V$ in terms of the measure $\P$ and the stopping times $\tau_b^{u,x}$ and $\tau_b^{g,x}$, defined in (\ref{eq:auxiliarystoppingtimetaugamma}) and (\ref{eq:auxiliarystoppingtimetaug}), respectively. For any $(u,x)\in E$ we can write
\begin{align}
V(u,x)&=\E_{u,x} \left(\int_{0}^{\tau_{D}} G(U_s,X_{s})\dd s \right)\nonumber\\
&=\E \left(\int_0^{\sigma_{-x}^- \wedge \tau_b^{u,x}} G(u+s,X_{s}+x)\dd s +\I_{\{	\sigma_{-x}^- \leq \tau_b^{u,x} \}} \int_{\sigma_{-x}^-}^{\tau_{b}^{g,x}} G(U_s^{(-x)},X_{s}+x)\dd s \right)\nonumber\\
\label{eq:Vwhenxpositive}
&=\E \left(\int_0^{\sigma_{-x}^- \wedge \tau_b^{u,x}} G(u+s,X_{s}+x)\dd s +\I_{\{	\sigma_{-x}^- \leq \tau_b^{u,x} \}} V(0,X_{\sigma_{-x}^-}+x) \right).
\end{align}

Note that in the last equality, we no longer have explicitly the process $\{ U_t^{(-x)}, t\geq 0\}$. This alternative representation of $V$ in terms of the original measure $\P$ will be useful to prove further properties of $b$ and $V$. \\	
The next lemma describes the limit behaviour of the function $b$.

\begin{lemma}
\label{lemma:behaviourofbatinfinity}
We have that 
\begin{align*}
\lim_{u \rightarrow \infty} b(u)=0.
\end{align*}
\end{lemma}

\begin{proof}
Note that, since the curve $b$ is non-increasing and it is bounded from below by $\lim_{u\rightarrow \infty} h(u)= 0$, the limit $b^*:=\lim_{u\rightarrow \infty} b(u)$ exists and $b^*\geq 0$. We prove by contradiction that $b^*=0$. Suppose $b^*>0$ and define the stopping time
\begin{align*}
\sigma_*=\inf\{ t\geq 0: X_{t}\notin (0,b^*) \}.
\end{align*}
Take $u> 0$ and $x\in (0,b^*)$. From the fact that $ b(u)\geq b^*>0$ we have that $\sigma_*\leq \tau_D \wedge \sigma_0^-$ under $\P_{u,x}$. Then we have that
\begin{align}
\label{eq:provinglimitofb}
V(u,x)&=\E_{u,x}\left( \int_0^{\tau_D} G(U_s,X_{s})\dd s \right)\nonumber\\
&=\E_{x}\left( \int_0^{\sigma_*} G(u+s,X_{s} )\dd s \right)+\E_{u,x}\left( V(U_{\sigma_*},X_{\sigma_*}) \right)\nonumber\\
&=\E_{x}\left( \int_0^{\sigma_*} G(u+s,X_{s} )\dd s \right)+\E_{x}\left( V(u+\sigma_*,X_{\sigma_*})\I_{\{X_{\sigma_*}>0 \}} \right)\nonumber\\
&\qquad +\E_{x}\left( V(0,X_{\sigma_*})\I_{\{X_{\sigma_*}\leq 0 \}} \right),
\end{align}
where in the last equality we used the Markov property of the two-dimensional process $\{(U_t,X_{t}), t\geq 0\}$. For a fixed $x\in \R$, the function $u \mapsto V(u,x)$ is non-decreasing and bounded from above by zero, thus we have that $\lim_{u\rightarrow \infty} V(u,x)$ exists and $-\infty < \lim_{u\rightarrow \infty} V(u,x)\leq 0$ for all $x\in \R$. By the dominated convergence theorem, we also conclude that $-\infty<\lim_{u\rightarrow \infty} \E_{x}\left( V(u+\sigma_*,X_{\sigma_*})\I_{\{X_{\sigma_*}> 0 \}} \right)\leq 0 $. Moreover, using Fatou's lemma (take $G(u+s,X_s)-G(0,0)\geq 0$ for all $s<\sigma_*$ and note that $\E(\sigma_*)<\E(\tau_D)<\infty$) and the fact that $\lim_{u\rightarrow \infty} G(u,x)=\infty$ we deduce that

\begin{align*}
\liminf_{u\rightarrow \infty } \E_{x}\left( \int_0^{\sigma_*} G(u+s,X_{s} )\dd s \right)\geq \E_{x}\left( \int_0^{\sigma_*} \liminf_{u\rightarrow \infty } G(u+s,X_{s} )\dd s \right) =\infty.
\end{align*}
Hence, taking $u\rightarrow \infty$ in (\ref{eq:provinglimitofb}) we get that

\begin{align*}
\lim_{u\rightarrow \infty } V(u,x)=\infty.
\end{align*}
Which yields the desired contradiction. Therefore, we conclude that $b^*=0$.
\end{proof}
In the following, we proceed to analyse the continuity properties of $b$ and $V$. Note that, by using standard arguments (from the fact that $D$ is a closed set), we can show that the function $b$ is right continuous. It turns out that $b$ is continuous, the proof of this fact makes use of a variational inequality and will be proved later.\\



Now we are ready to show the continuity of the value function $V$. The proof is rather long and technical so is included in Appendix \ref{sec:Appendix}.
\begin{lemma}
\label{lemma:continuityofV0x}
The function $V$ is continuous on $E$. Moreover, in the case that $X$ is of infinite variation we have that
\begin{align*}
\lim_{h\downarrow 0} V(u,h)=V(0,0)
\end{align*}
for all $u>0$.

\end{lemma}
We know that $D$ is a closed set, so $b$ is a right-continuous function. To show left continuity, we use a variational inequality that is satisfied by the value function $V$. We will dedicate the upcoming paragraphs to introducing that.\\

It is well known that for every optimal stopping problem, there is an associated free boundary problem, which is stated in terms of the infinitesimal generator (see e.g. \cite{peskir2006optimal}, Chapter III). In this case, provided that the value function is smooth enough, we have that $V$ solves the Dirichlet/Poisson problem. That is,

\begin{align*}
\mathcal{A}_{U,X}( V)=\frac{\partial}{\partial u} \widetilde{V}+\mathcal{A}_{X}( \widetilde{V}) =-G \qquad \text{in } E\setminus D,
\end{align*}
where $\mathcal{A}_{U,X}$ corresponds to the infinitesimal generator of the process $(U,X)$ given in (\ref{eq:infinitesimalgeneratorUX}) and $\mathcal{A}_X$ is the infinitesimal generator of $X$ given by
\begin{align*}
\mathcal{A}_X(\widetilde{V})&=  -\mu \frac{\partial }{\partial	x} \widetilde{V}(u,x) +\frac{1}{2} \sigma^2  \frac{\partial^2}{\partial x^2} \widetilde{F}(u,x)\\
&\qquad  +\int_{(-\infty,0)} \left( \widetilde{V}(u,x+y) -\widetilde{V}(u,x)-y\I_{\{y>-1\}}\frac{\partial }{\partial x} \widetilde{V}(u,x) \right)\Pi(\dd y),
\end{align*}
whilst $\widetilde{V}$ is the extension of $V$ to the set $\R_+ \times \R$ given by
\begin{align}
\label{eq:extensionofVtoR+R}
\widetilde{V}(u,x)=\left\{
\begin{array}{cc}
V(u,x), & \text{for } u>0  \text{ and } x>0,\\
V(0,x), & \text{for }u\geq 0 \text{ and } x\leq 0,\\
V(0,0), & \text{for } u=0 \text{ and } x>0.
\end{array}
\right.
\end{align}
However, in our setting, it turns out to be challenging to show the smoothness of the function $V$. Indeed, from \cite{garroni2002second} it could be checked that when the diffusion component is non-degenerate (that is $\sigma>0$), the function $V$ is $C^{1,2}$ on the set $C^+=\{ (u,x)\in E: 0<x<b(u)\}$ (cf. \cite{bayraktar2012regularity}). However, for the case $\sigma=0$ differentiability of $V$ might fail, even in the finite variation case (see e.g. \cite{cont2005integro}). It turns out that \cite{lamberton2008critical} showed that we could state an analogous (in)equality in the sense of distributions.\\

In Appendix \ref{appendix:variationaliniequalities} we recall some facts and notation from the theory of distributions (see also \cite{friedlander1998introduction} or \cite{rudin1991functional} for further details). Since $V$ is continuous on $E$ we have that $\widetilde{V}$ is a locally integrable function in $\R_+\times \R$ (note that $\widetilde{V}$ may be discontinuous at points of the form $(u,0)$ for $u>0$ when $X$ is of finite variation) so we can define $\widetilde{V}$ as a distribution in any open set $O \subset \R_+ \times \R$ via the functional
\begin{align*}
\varphi\mapsto \langle \widetilde{V}, \varphi \rangle= \int_{\R_+} \int_{\R} \widetilde{V}(u,x)\varphi(u,x)\dd x \dd u,
\end{align*}
where $\varphi$ is taken from the set of test functions with compact support in $O$. The derivatives of the distribution $\widetilde{V}$ are defined as
\begin{align*}
\varphi \mapsto \langle \frac{\partial^{i+j}}{\partial u^i \partial x^j}\widetilde{V}, \varphi \rangle= (-1)^{i+j}\int_{\R_+} \int_{\R} \widetilde{V}(u,x) \frac{\partial^{i+j}}{\partial u^i \partial x^j} \varphi(u,x)\dd x \dd u.
\end{align*}
Moreover, provided that the function $(u,x) \mapsto \int_{(-\infty,-1)}\widetilde{V}(u,x+y)\Pi(\dd y)$ is locally integrable in $\R_+ \times \R$, the functional $B_X(\widetilde{V})$ defined by 
\begin{align*}
\varphi \mapsto \langle B_X(\widetilde{V}), \varphi \rangle =\int_{\R_+} \int_{\R} \widetilde{V}(u,x)  B_X^*(\varphi)(u,x) \dd x \dd u,
\end{align*}
defines a distribution on $O$ (see Lemma \ref{lemma:distributionofthejumpspart}), where
\begin{align*}
 B_X^*(\varphi)(u,x)=\int_{(-\infty,0)} [\varphi(u,x-y)-\varphi(u,x)+y \frac{\partial}{\partial x}\varphi(u,x) \I_{\{ y>- 1\}}]\Pi(\dd y).
\end{align*}
We have the following Lemma that ensures that the integrability conditions for $\widetilde{V}$ are satisfied so then $B_X(\widetilde{V})$ is indeed a distribution.

\begin{lemma}
\label{lemma:intVislocallyintegrable}
The function 
\begin{align*}
(u,x) \mapsto \int_{(-\infty,-1)}\widetilde{V}(u,x+y)\Pi(\dd y)
\end{align*}
is locally integrable in $\R_+\times \R$.
\end{lemma}
\begin{proof}
First note that from equation (\ref{eq:lowerboundforV}) we have that for any $x\leq 0$, 
\begin{align*}
\int_{(-\infty,-1)} V(0,x+y) \Pi(\dd y)
&\geq -A'_{p-1}\Pi(-\infty,-1]-C'_{p-1} \int_{(-\infty,-1)}|x+y|^p \Pi(\dd y)\\
&\qquad+V(0,0) \Pi(-\infty,-1)\\
&>-\infty,
\end{align*}
where we used the fact that $\Pi(-\infty,-1)<\infty$ and Lemma \ref{lemma:finitenessofEgn}. Moreover, since $V$ is non-decreasing in each argument we have that for any $u>0$ and $x>0$ that
\begin{align*}
\int_{(-\infty,-1)}  \widetilde{V}(u,x+y) \Pi(\dd y)\geq \int_{(-\infty,-1)}  V(0,y) \Pi(\dd y)>-\infty.
\end{align*}
Hence we conclude that $\int_{(-\infty,-1)}  \widetilde{V}(u,x+y) \Pi(\dd y)>-\infty$ for any $(u,x)\in \R_+\times \R$. Since $V$ is continuous on $E$ and the definition of $\widetilde{V}$ we have that the mapping $(u,x) \mapsto \int_{(-\infty,-1)}\widetilde{V}(u,x+y)\Pi(\dd y)$ is locally integrable. 
\end{proof}
Hence, we can define the operator $\mathcal{A}_X$ in the sense of distributions by 
\begin{align*}
\mathcal{A}_X(\widetilde{V})=-\mu \frac{\partial }{\partial	x} \widetilde{V} +\frac{1}{2} \sigma^2  \frac{\partial^2}{\partial x^2} \widetilde{V}+ B_X(\widetilde{V}).
\end{align*}
The next lemma is an extension of Proposition 2.5 (see also Theorem 2.8) in \cite{lamberton2008critical}. 

\begin{lemma}
\label{lemma:variationalinequalities}
The distribution $\frac{\partial}{\partial u} \widetilde{V}+\mathcal{A}_{X}(\widetilde{V})+G$ is a non-negative distribution on $(0,\infty)\times (0,\infty)$. Moreover, we have $\frac{\partial}{\partial u} \widetilde{V}+\mathcal{A}_{X}(\widetilde{V})+G=0$ on the set $C^+:=\{(u,x) \in(0,\infty)\times (0,\infty): 0<x<b(u) \}$ and $\mathcal{A}_X(V(0,\cdot))+G(0,\cdot)=0$ on $(-\infty,0)$ in the sense of distributions. 
\end{lemma}

\begin{proof}

From the general theory of optimal stopping we have that (see \cite{peskir2006optimal}, Theorem 2.4) for every $(u,x)\in E$, the stochastic process $\{ Z_t, t\geq 0\}$ is a sub-martingale under the measure $\P_{u,x}$, where 
\begin{align*}
Z_t=V(U_t,X_t)+\int_0^{t}G(U_s,X_s)\dd s. 
\end{align*}
Moreover, we have that the stopped process $\{Z_{t\wedge \tau_D}, t\geq 0 \}$ is a martingale under $\P_{u,x}$ for all $(u,x)\in E$. Then, from Doob's stopping time theorem, we have that for every $(u,x)\in E$, the process $\{ Z_{t\wedge \sigma_0^-}, t\geq 0\}$ is a sub-martingale and $\{ Z_{t\wedge \tau_D \wedge \sigma_0^-}, t\geq 0 \}$ is a martingale under $\P_{u,x}$. From the fact that $U_t=0$ if and only if $X_t\leq 0$ we have that, under $\P_{u,x}$,
\begin{align*}
Z_{t\wedge \sigma_0^-}&=V(U_{t\wedge\sigma_0^-},X_{t\wedge\sigma_0^-})+\int_0^{t\wedge\sigma_0^-}G(U_s,X_s)\dd s\\
&=V(u+t,X_{t})\I_{\{t<\sigma_0^-\}}+V(0,X_{\sigma_0^-})\I_{\{\sigma_0^- \leq t\}}+\int_0^{t\wedge\sigma_0^-}G(u+s,X_s)\dd s\\
&=\widetilde{V}(u+t,X_{t})\I_{\{t<\sigma_0^-\}}+\widetilde{V}(u+\sigma_0^-,X_{\sigma_0^-})\I_{\{\sigma_0^- \leq t\}}+\int_0^{t\wedge\sigma_0^-}G(u+s,X_s)\dd s\\
&=\widetilde{V}(u+t\wedge\sigma_0^-,X_{t \wedge \sigma_0^-})+\int_0^{t\wedge\sigma_0^-}G(u+s,X_s)\dd s
\end{align*}
for every $u>0$ and $x>0$. Hence, using an analogous argument as in Proposition 2.5 in \cite{lamberton2008critical} (see also Proposition \ref{prop:variationalinequality}), we have that $\frac{\partial}{\partial u } \widetilde{V}+\mathcal{A}_X(\widetilde{V})+G$ is a non-negative distribution on $(0,\infty)\times (0,\infty)$. Similarly, we have that for any $u>0$ and $x>0$ such that $x<b(u)$, 
\begin{align*}
Z_{t\wedge \sigma_0^- \wedge \tau_D}
&=\widetilde{V}(u+t\wedge\sigma_0^-\wedge \tau_D,X_{t \wedge \sigma_0^-\wedge \tau_D})+\int_0^{t\wedge\sigma_0^-\wedge \tau_D}G(u+s,X_s)\dd s.
\end{align*}
Therefore, from Proposition \ref{prop:variationalinequality}, we have that $\frac{\partial}{\partial u } \widetilde{V}+\mathcal{A}_X(\widetilde{V})+G=0$ on $C^+$ in the sense of distributions. Lastly, since $b$ is non-negative, we have that $\tau_0^+\leq \tau_D$. Hence, under the measure $\P_{0,x}$, for any $x<0$, we have that $\{Z_{t\wedge \tau_0^+} , t\geq 0\}$ is a martingale. Moreover, since $X_t\leq 0$ for all $t<\tau_0^+$ we have that 
\begin{align*}
Z_{t\wedge  \tau_0^+}
&=V(0,X_{t\wedge  \tau_0^+})+\int_0^{t\wedge  \tau_0^+}G(0,X_s)\dd s.
\end{align*}
Then, as in Proposition 2.5 in \cite{lamberton2008critical}, from Proposition \ref{prop:variationalinequality} we deduce that $\mathcal{A}_X(V(0,\cdot))+G(0,\cdot)=0$ in the sense of distributions on the set $(-\infty,0)$. 
\end{proof}

\begin{rem}
\label{rem:variationalinequality}
\begin{itemize}
\item[i)]In \cite{lamberton2008critical} the definition of the infinitesimal generator in the sense of distributions assumes that the value function is a bounded Borel measurable function. In our setting such condition can be relaxed by the fact that $(u,x) \mapsto \int_{(-\infty,-1)} |\widetilde{V}(u,x+y)| \Pi(\dd y)$ is a locally integrable function on $\R_+ \times \R$. 
\item[ii)]  We note that similar as in (\ref{eq:infinitesimalgeneratorUX}) the infinitesimal generator of $(U,X)$ can be defined as $\mathcal{A}_{U,X}(V):= \partial/ \partial u \widetilde{V} +\mathcal{A}_{X}(\widetilde{V})$ in the sense of distributions, where $\mathcal{A}_{X}$ corresponds to the infinitesimal generator of $X$ (seen as a distribution).
\end{itemize}
\end{rem}
Let $\text{int}(D)$ be the interior of the set $D$. For $(u,x)\in \text{int}(D)$ we define the function 

\begin{align*}
\Lambda(u,x):=
\int_{(-\infty,0)} \widetilde{V}(u,x+y)\Pi(\dd y)+G(u,x).
\end{align*}

The next lemma states some basic properties of the function $\Lambda$.
\begin{lemma}
\label{lemma:auxiliaryfunctionLambda}
The function $\Lambda$ is such that $0<\Lambda (u,x)<\infty$ for all $(u,x)\in \text{int}(D)$. Moreover, it is strictly increasing in each argument and continuous in the interior of the set $D$. Furthermore, $\Lambda=\frac{\partial}{\partial u} \widetilde{V}+   \mathcal{A}_{X}(\widetilde{V})+G$ on $\text{int}(D)$ in the sense of distributions.

\end{lemma}
%

\begin{proof}
It follows from Lemma \ref{lemma:intVislocallyintegrable} and the fact that $V$ vanishes in $D$ that $|\Lambda(u,x)|<\infty$ for all $(u,x)\in E$. The fact that $\Lambda$ is continuous on $D$ follows from the continuity of $V$ and $G$, the dominated convergence theorem and the fact that $\Pi$ has no atoms. Moreover, $\Lambda$ is strictly increasing in each argument on $D$ since $V$ is non-decreasing in each argument and $G$ is strictly increasing in each argument on $D$. Then, we show that $\partial/ \partial u \widetilde{V}+ \mathcal{A}_{X}(\widetilde{V})+G=\Lambda$ on in the interior of $D$. Let $\varphi$ be a $C^{\infty}$ function with compact support on the interior of $D$. We have that
\begin{align*}
\langle  \frac{\partial}{\partial u}& \widetilde{V}+   \mathcal{A}_{X}(\widetilde{V})+G, \varphi \rangle\\
&=\int_{0}^{\infty}  \int_{-\infty}^{\infty} \widetilde{V}(u,x)[\mu \frac{\partial}{\partial x} \varphi(u,x)+\frac{1}{2}\sigma^2 \frac{\partial^2 }{\partial x^2} \varphi(u,x)+B_X^*(\varphi)(u,x)]\dd x  \dd u \\
&\qquad  +\int_0^{\infty} \int_{-\infty}^{\infty} G(u,x) \varphi(u,x)\dd x \dd u \\
&=\int_{0}^{\infty}  \int_{-\infty}^{b(u)} \widetilde{V}(u,x)[\mu \frac{\partial}{\partial x} \varphi(u,x)+\frac{1}{2}\sigma^2 \frac{\partial^2 }{\partial x^2} \varphi(u,x)+B_X^*(\varphi)(u,x)]\dd x  \dd u \\
&\qquad +\int_{0}^{\infty}  \int_{b(u)}^{\infty} \widetilde{V}(u,x)[\mu \frac{\partial}{\partial x} \varphi(u,x)+\frac{1}{2}\sigma^2 \frac{\partial^2 }{\partial x^2} \varphi(u,x)+B_X^*(\varphi)(u,x)]\dd x  \dd u \\
&\qquad  +\int_0^{\infty} \int_{-\infty}^{b(u)} G(u,x) \varphi(u,x)\dd x \dd u +\int_0^{\infty} \int_{b(u)}^{\infty} G(u,x) \varphi(u,x)\dd x \dd u\\
&= \int_{0}^{\infty}  \int_{-\infty}^{b(u)} \widetilde{V}(u,x) \int_{(-\infty,0)} \varphi(u,x-y)\Pi(\dd y) \dd x \dd u \\
&\qquad  +\int_0^{\infty} \int_{b(u)}^{\infty} G(u,x) \varphi(u,x)\dd x \dd u ,
\end{align*}
where the last equality follows since $V(u,x)=0$ for all $x\geq b(u)$, and since $\varphi$ has support on the interior of $D$, so that, for any $x\leq b(u)$ we have $\varphi(u,x)=\frac{\partial}{\partial x}\varphi(u,x)=\frac{\partial^2}{\partial x^2}\varphi(u,x)=0$ and
\begin{align*}
B_X^*(\varphi)(u,x)&= \int_{(-\infty,0)} [\varphi(u,x-y)-\varphi(u,x)+y \frac{\partial}{\partial x}\varphi(u,x) \I_{\{ y>- 1\}}]\Pi(\dd y)\\
&=\int_{(-\infty,0)} \varphi(u,x-y)\Pi(\dd y).
\end{align*}
Moreover, by the change of variable $z=x-y$, we see that for any $u>0$, 
\begin{align*}
\int_{-\infty}^{b(u)} \widetilde{V}(u,x) \int_{(-\infty,0)} \varphi(u,x-y)\Pi(\dd y) \dd x 
&=\int_{-\infty}^{\infty} \int_{(-\infty,0)} \widetilde{V}(u,x)  \varphi(u,x-y)\Pi(\dd y) \dd x\\
&=\int_{-\infty}^{\infty} \int_{(-\infty,0)} \widetilde{V}(u,z+y)  \varphi(u,z)\Pi(\dd y) \dd z\\
&=\int_{b(u)}^{\infty} \varphi(u,z) \int_{(-\infty,0)} \widetilde{V}(u,z+y)  \Pi(\dd y) \dd z,
\end{align*}
where we used again that $V$ vanishes on $D$ and $\varphi$ has support on the interior of $D$. Thus, we obtain that 
\begin{align*}
\langle  \frac{\partial}{\partial u} \widetilde{V}+   \mathcal{A}_{X}(\widetilde{V})+G, \varphi \rangle
&= \int_{0}^{\infty} \int_{b(u)}^{\infty}\varphi(u,z)  \int_{(-\infty,0)} \widetilde{V}(u,z+y) \Pi(\dd y) \dd z \dd u\\
&\qquad +\int_0^{\infty} \int_{b(u)}^{\infty} G(u,z) \varphi(u,x) \dd z \dd u \\
&=\int_0^{\infty} \int_{b(u)}^{\infty} \Lambda(u,x) \varphi (u,z)\dd z \dd u \\
&=\langle \Lambda,\varphi\rangle.
\end{align*}
Then we have that $\frac{\partial}{\partial u} \widetilde{V}+ \mathcal{A}_{X}(\widetilde{V})+G=\Lambda$ on in the interior of $D$. Moreover from Lemma \ref{lemma:variationalinequalities} and continuity of $\Lambda$ we conclude that $\Lambda(u,x)\geq 0$ for all $(u,x)\in \text{int}(D)$. In particular, is strictly positive in the interior of $D$ since it is strictly increasing in that set. 
\end{proof}

It turns out that the function $b$ is continuous, its proof is analogous to the one presented in \cite{lamberton2008critical} (see Theorem 4.2) in the American option setting.   
\begin{lemma}\label{lemma:continuous}
The function $b$ is continuous.
\end{lemma}

\begin{proof}
We already know, from the fact that $D$ is closed, that $b$ is right-continuous. We then show the left continuity of $b$. We proceed by contradiction. Suppose there is a point $u_*>0$ such that $b(u_*-):=\lim_{h\downarrow 0} b(u_*-h)>b(u_*)$. Then, since $b$ is non-decreasing, we have for all $(u,x) \in (0,u_*) \times (b(u_*),b(u_*-))$ that $V(u,x)<0$. Thus, $(0,u_*) \times (b(u_*),b(u_*-)) \subset C^+$. From Lemma \ref{lemma:variationalinequalities} we obtain that $\frac{\partial}{\partial u} \widetilde{V}+\mathcal{A}_X(\widetilde{V}) +G =0$ in $ (0,u_*) \times (b(u_*),b(u_*-))$. Hence, for any non-negative $C^{\infty}$ function $\varphi$ with compact support in $ (0,u_*) \times (b(u_*),b(u_*-))$, we have that
\begin{align*}
\langle \mathcal{A}_X(\widetilde{V}) +G ,\varphi \rangle &=-\langle \frac{\partial}{\partial u} \widetilde{V} ,\varphi \rangle\\
&=\int_{\R}\int_{(0,\infty)} \widetilde{V}(u,x) \frac{\partial}{\partial u} \varphi(u,x)\dd u \dd x \\
&=-\int_{\R}\int_{(0,\infty)} \widetilde{V}(\dd u,x)  \varphi(u,x) \dd x \\
&\leq 0,
\end{align*}
where we used the fact that for each $x>0$, $u\mapsto \widetilde{V}(u,x)=V(u,x)$ is non-decreasing. Hence, we conclude that $\mathcal{A}_X(\widetilde{V}) +G$ is a non-positive distribution on $ (0,u_*) \times (b(u_*),b(u_*-))$.
Thus, by continuity of $\widetilde{V}=V$ and $G$ on $(0,\infty)\times (0,\infty)$, we have for any $u\in (0,u_*)$ and any non-negative test function $\psi$ with compact support in $(b(u_*),b(u_*-))$ that
\begin{align*}
\int_{\R} \left\{ \widetilde{V}(u,x)\left[ -\mu \frac{\partial}{\partial x}\psi(x) +\frac{1}{2}\sigma^2 \frac{\partial^2}{\partial x^2} \psi(x)+ B_{X}^*(\psi)(x) \right]+G(u,x)\psi(x)\right\}\dd x \leq 0,
\end{align*}
where $B_X^*(\psi)(x)=\int_{(-\infty,0)} (\psi(x-y) -\psi(x)+y \frac{d}{dx} \psi(x) \I_{\{|y|\leq 1 \}} )\Pi(\dd y)$. Taking $u\uparrow u_*$ in the equation above, using the fact that $\widetilde{V}(u_*,x)=0$ for all $x \geq  b(u_*)$, and since $\psi$ has compact support in $(b(u_*),b(u_*-))$ we get that
\begin{align*}
0 & \geq \lim_{u \uparrow u_*} \int_{\R} \left\{ \widetilde{V}(u,x)\left[ -\mu \frac{\partial}{\partial x}\psi(x) +\frac{1}{2}\sigma^2 \frac{\partial^2}{\partial x^2} \psi(x)+ B_{X}^*(\psi)(x) \right]+G(u,x)\psi(x)\right\}\dd x\\
&= \int_{-\infty}^{b(u_*)}  \widetilde{V}(u_*,x)\int_{(-\infty,0)} \psi(x-y)\Pi(\dd y)\dd x +  \int_{b(u_*)}^{b(u_*-)}G(u_*,x)\psi(x)\dd x \\
&=\int_{b(u_*)}^{b(u_*-)} \psi(x) \int_{(-\infty,0)}\widetilde{V}(u_*,x+y) \Pi(\dd y)\dd x +  \int_{b(u_*)}^{b(u_*-)}G(u_*,x)\psi(x)\dd x\\
&= \int_{b(u_*)}^{b(u_*-)} \psi(x) \Lambda(u_*,x)\dd x\\
&>0,
\end{align*}
where the strict inequality follows from the fact that $\Lambda$ is strictly positive in each argument in $D$ (see Lemma \ref{lemma:auxiliaryfunctionLambda}). Hence we have got a contradiction and $b(u-)=b(u)$ for all $u>0$. Therefore $b$ is a continuous function.
\end{proof}

From Lemma \ref{lemma:behaviourofbatinfinity} we know that $b(u)$ converges to zero when $u$ tends to infinity. Moreover, from the discussion about $h$ after equation (\ref{eq:functionh}), we know that in case that $X$ is of finite variation, there exists a value $u_h^*<\infty$ for which $h(u)=0$ for all $u\geq u^*_h$. That suggests a similar behaviour for $b$. The next lemma addresses that conjecture.

\begin{lemma}
\label{cor:valueofub}
Define $u_b=\inf\{u>0: b(u)=0 \}$. If $X$ is of infinite variation or finite variation and infinite activity (that is $\Pi(-\infty,0)=\infty$) we have that $u_b=\infty$. Otherwise, $u_b=u^*$, where $u^*$ is the unique solution to
\begin{align}
\label{eq:definitionofub}
G(u,0)+\int_{(-\infty,0)} V(0,y)\Pi(\dd y)=0.
\end{align}
\end{lemma}

\begin{proof}
From the fact that $h(u)>0$ for all $u>0$ when $X$ is of infinite variation and the inequality $b(u)\geq h(u)$, we have that the assertion is true for this case. Suppose that $X$ has finite variation with infinite activity, that is, $\Pi(-\infty,0)=\infty$, and assume that $u_b<\infty$. Then, since $b$ is non-increasing, we have that $b(u)=0$ for all $u>u_b$ and thus $V(u,x)=0$ for all $x>0$ and $u>u_b$. From Lemma \ref{lemma:auxiliaryfunctionLambda} we deduce that
\begin{align*}
G(u,x) +\int_{(-\infty,-x)}V(0,x+y)\Pi(\dd y) \geq 0 
\end{align*}
for all $x>0$ and for all $u>u_b$. Taking $x\downarrow 0$ in the equation above and using the expression for $V(0,z)$ (when $z<0$) given in (\ref{eq:expressionforV0xnegative}) we have that for any $u>u_b$,

\begin{align*}
0&\leq G(u,0) -\lim_{x \downarrow 0} \int_{(-\infty,0)} \int_{0}^{-x+y} \int_{0}^{\infty}  \E_{-u-z}(g^{p-1})W'(u) \dd u  \dd z \Pi(\dd y)\\
&\qquad +\lim_{x\downarrow 0}V(0,0) \Pi(-\infty,-x)\\
& =-\infty
\end{align*}
which is a contradiction and then $u_b=\infty$. Next, assume that $X$ has finite variation with $\Pi(-\infty,0)<\infty$. Assume that $b(u^*)>0$, then $V(u^*,x)<0$ for $x\in (0,b(u^*))$. Moreover, since $V\leq 0$ and using the compensation formula for Poisson random measures (see \eqref{eq:compensationformulaPRM}) we have that for all $u>0$ and $x<b(u)$,

\begin{align*}
&\E_{u,x}(V(0,X_{\tau_0^-})\I_{\{\tau_0^-<\tau_D\}} )\\
&=\E_{u,x}\left(\int_{[0,\infty)} \int_{(-\infty,0)} V(0,X_{s-}+y)\I_{\{\underline{X}_{s-}>0,X_{s-}+y<0 \}}\I_{\{s\leq \tau_D \}} N(\dd s,\dd y)  \right)\\
&=\E_{u,x}\left(\int_{0}^{\infty} \int_{(-\infty,0)} V(0,X_{s}+y)\I_{\{\underline{X}_{s-}>0 \}}\I_{\{X_{s}+y<0 \}}\I_{\{s\leq \tau_D \}} \Pi(\dd y) \dd s  \right)\\
&=\E_{u,x}\left(\int_{0}^{\tau_0^-\wedge \tau_D} \int_{(-\infty,0)} V(0,X_{s}+y)\I_{\{X_{s}+y<0 \}} \Pi(\dd y) \dd s  \right).
\end{align*}
Then, from the Markov property we have that for any $0\leq x<b(u^*)$,
\begin{align*}
V(u^*,x)&=\E_{u^*,x}\left(\int_0^{\tau_D \wedge \tau_0^-} G(u^*+s,X_s)\dd s \right)+\E_{u^*,x}(V(0,X_{\tau_0^-})\I_{\{\tau_0^-<\tau_D\}} )\\
&=\E_{u^*,x}\left(\int_0^{\tau_D \wedge \tau_0^-} \left[G(u^*+s,X_s) + \int_{(-\infty,0)} V(0,X_{s}+y)\I_{\{X_{s}+y<0 \}} \Pi(\dd y)\right]\dd s  \right)\\
&>0,
\end{align*}
where the strict inequality follows from the fact that that $X$ is of finite variation and then $\tau_D \wedge \tau_0^->0$, the definition of $u^*$ and the fact that $G$ and $V$ are non-decreasing in each argument. Then, we are contradicting the fact that $V(u^*,x)<0$ and we conclude that $b(u^*)=0$ so that $u_b\leq u^*$. Moreover, from Lemma \ref{lemma:auxiliaryfunctionLambda} we know that for all $u>u_b$,
\begin{align*}
G(u,x) +\int_{(-\infty,-x)}V(0,x+y)\Pi(\dd y) \geq 0 \text{ for all } x>0.
\end{align*}
Taking $x\downarrow 0$ we get that for all $u\geq u_b$, $G(u,0) +\int_{(-\infty,0)}V(0,y)\Pi(\dd y) \geq 0$. The latter implies that $u^*\leq u_b$ (since $u\mapsto G(u,0)$ is strictly increasing). Therefore we conclude that $u^*=u_b$ and the proof is complete.
\end{proof}

As we mentioned before, proving the smoothness of $V$ is challenging. However, it is possible to show that the derivatives of $V$ at the boundary exist and are equal to zero. 
Recall from Lemma \ref{cor:valueofub} that when $X$ is of infinite variation or finite variation with infinite activity we have that $b(u)>0$ for all $u>0$. In the case that $X$ is of finite variation we have that $b(u)>0$ only if $u<u_b$ where $u_b$ is the solution to (\ref{eq:definitionofub}). In such cases, we can guarantee that the derivatives of $V$ exist at the boundary and are equal to zero, which is proven in the following Theorem. Since the proof is rather long and technical, it can be found in Appendix \ref{sec:Appendix}.

\begin{lemma}
\label{lemma:smoothfit}
Suppose that $u>0$ is such that $b(u)>0$. Then the first partial derivatives of $V(u,x)$ exist at the point $x=b(u)$ and
\begin{align*}
\frac{\partial }{\partial x} V(u, b(u))=0 \qquad \text{and} \qquad \frac{\partial }{\partial u} V(u, b(u))=0.
\end{align*}
\end{lemma}
Recall from equation (\ref{eq:expressionforV0xnegative}) that when $x<0$,
\begin{align*}
V(0,x)=-\int_{0}^{-x} \int_{0}^{\infty}  \E_{-u-z}(g^{p-1})W'(u) \dd u  \dd z +V(0,0).
\end{align*}
Note that the first term on the right-hand side of the equation above does not depend on the boundary $b$. Then, for $x<0$, the value function $V(0,x)$ is characterised by the value $V(0,0)$. Moreover, from Lemma \ref{cor:valueofub} we know that when $X$ is of finite variation with $\Pi(-\infty,0)<\infty$, the value $u_b$ is the unique solution to
\begin{align*}
G(u,0)-\int_{(-\infty,0)} \int_{0}^{-y} \int_{0}^{\infty}  \E_{-v-z}(g^{p-1})W'(v) \dd v  \dd z\Pi(\dd y)+ V(0,0)\Pi(-\infty,0)=0,
\end{align*}
otherwise, $u_b=\infty$. Then if $X$ is of finite variation with finite activity, $u_b$ is also characterised by the value $V(0,0)$, where we know from Remark \ref{rem:boundsforV} that 
\begin{align*}
-\frac{\E(g^{p-1})}{p} \leq V(0,0)< 0. 
\end{align*}

The next theorem gives a characterisation of the value function $V$ on the set $(0,\infty)\times (0,\infty)$, the boundary $b$ and the values $V(0,0)$ and $u_b$ as unique solutions of a system of non-linear integral equations within a class of functions. The method of proof is deeply inspired by the ideas of \cite{du2008predicting}. However, the presence of jumps adds an important level of difficulty. In particular, when $\Pi\neq 0$, the inequality (see Lemma \ref{lemma:auxiliaryfunctionLambda}) 

\begin{align*}
\Lambda(u,x)=G(u,x)+\int_{(-\infty,0)} \widetilde{V}(u,x+y)\Pi (\dd y)>0,
\end{align*}
for all $(u,x)\in \text{int}(D)$, is a necessary condition for the stochastic process $\{ V(U_t,X_t)+\int_0^t G(U_s,X_s) \dd s , t\geq 0\}$ to be a submartingale.

\begin{thm}
\label{thm:characterisationofbandV}
Let $p>1$ and $X$ be a spectrally negative L\'evy process drifting to infinity such that its L\'evy measure $\Pi$ has no atoms and $\int_{(-\infty,-1)} |x|^{p+1} \Pi(\dd x)<\infty$.  For all $u>0$ and $x>0$, the function $V$ can be written as
\begin{align}
&V(u,x)\nonumber\\
&=V(0,0)\frac{\sigma^2}{2}W'(x)\nonumber\\
&\qquad -\E_{x} \left(  \int_0^{ \tau_{0}^-} \int_{(-\infty,0)} V(u+s,X_s+y) \I_{\{0<X_s+y<b(u+s) \}} \Pi(\dd y)  \I_{\{ X_s>b(u+s)\}}\dd s \right) \nonumber\\
\label{eq:representationforVwithindicatorsfunc}
&\qquad +\E_{x}\left( \int_0^{ \tau_{0}^-} \left[ G(u+s,X_s) +\int_{(-\infty,-X_s)} V(0,X_s+y)  \Pi(\dd y)  \right] \I_{\{ X_s<b(u+s)\}} \dd s \right),
\end{align}
the value $V(0,0)$ satisfies 
\begin{align}
\label{eq:charactersationofV00}
V(0,&0)=-\frac{1}{\psi'(0+)}\int_0^{\infty} \E_{-z}(g^{p-1}) [1-\psi'(0+)W(z)]\dd z\nonumber\\
&\qquad +\frac{1}{\psi'(0+)}\E\left( \int_{0}^{\infty}   G(s,X_s)\frac{X_s}{s}  \I_{\{0<X_s<b(s) \}} \dd s \right) \nonumber\\
&\qquad-\frac{1}{\psi'(0+)} \E\left( \int_0^{\infty} \int_{(-\infty,0)} V(s,X_s+y)\I_{\{0<X_s+y<b(s) \}} \Pi(\dd y)\frac{X_s}{s} \I_{\{X_s>b(s) \}} \dd s \right)\nonumber\\
&\qquad-\frac{1}{\psi'(0+)} \E\left( \int_0^{\infty} \int_{(-\infty,0)} V(0,X_s+y)\I_{\{X_s+y\leq  0\}} \Pi(\dd y)\frac{X_s}{s} \I_{\{X_s>b(s) \}} \dd s \right),
\end{align}
whilst the curve $b$ satisfies the equation
\begin{align}
0&=V(0,0)\frac{\sigma^2}{2}W'(b(u)) \nonumber\\
&\qquad-\E_{b(u)} \left(  \int_0^{ \tau_{0}^-} \int_{(-\infty,0)} V(u+s,X_s+y) \I_{\{0<X_s+y<b(u+s) \}} \Pi(\dd y)  \I_{\{ X_s>b(u+s)\}}\dd s \right) \nonumber\\
\label{eq:characterisationofb}
&\qquad +\E_{b(u)}\left( \int_0^{ \tau_{0}^-} \left[ G(u+s,X_s) +\int_{(-\infty,-X_s)} V(0,X_s+y)  \Pi(\dd y)  \right] \I_{\{ X_s<b(u+s)\}} \dd s \right)
\end{align}
for all $u<u_b$, where for $x\leq 0$, the function $V(0,x)$ depends on $V(0,0)$ via (\ref{eq:expressionforV0xnegative}). For $u\geq u_b$, we have $b(u)=0$, where $u_b=\infty$ in the case $X$ is of infinite variation or finite variation with $\Pi(-\infty,0)=\infty$. Otherwise, $u_b$ is the unique solution to
\begin{align}
\label{eq:characterisationofub}
G(u,0)-\int_{(-\infty,0)} \int_{0}^{-y} \int_{0}^{\infty}  \E_{-u-z}(g^{p-1})W'(u) \dd u  \dd z\Pi(\dd y)+ V(0,0)\Pi(-\infty,0)=0.
\end{align}
Moreover, in the case that there is a Brownian motion component (i.e. $\sigma>0$) we have that (\ref{eq:charactersationofV00}) is equivalent to 
\begin{align}
\label{eq:derivativeofVat00}
\frac{\partial}{\partial x} V_+ (0,0) =\frac{\partial}{\partial x} V_- (0,0),
\end{align}
where $\frac{\partial}{\partial x} V_+ (u,0)$ and $\frac{\partial}{\partial x} V_- (0,0)$ are the right and left derivatives of $x\mapsto  V(u,x)$ and $x\mapsto V(0,x)$ at zero, respectively and $\frac{\partial}{\partial x} V_+ (0,0)=\lim_{u\downarrow 0} \frac{\partial}{\partial x} V_+ (u,0)$.\\

Furthermore, the quadruplet $(V,b,V(0,0),u_b)$ is uniquely characterised by the equations above, where $V$ is considered in the class of non-positive continuous functions such that 

\begin{align}
\label{eq:intV+Gispositive}
\int_{(-x-b(u),-x)} & V(u,b(u)+x+y)\Pi (\dd y)\nonumber\\
&\qquad \qquad +\int_{(-\infty,-x-b(u)]} V(0,b(u)+x+y)\Pi (\dd y)+G(u,x+b(u))\geq 0
\end{align}
for all $u<u_b$ and $x>0$ and $b$ is considered in the class of non-increasing functions with $b\geq h$ whereas $-\frac{1}{p}\E(g^p)\leq V(0,0)<0$.
\end{thm}

%

Since the proof of Theorem \ref{thm:characterisationofbandV} is rather long, we break it into a series of Lemmas. The next section is entirely dedicated to that purpose.

\section{Proof of Theorem \ref{thm:characterisationofbandV}}
\label{sec:proofofmainmainthm}
First, we show that the relevant quantities are integrable. The proof of the Lemma is long, and then it is included in Appendix \ref{sec:Appendix}.

\begin{lemma}
\label{lemma:integrativilityconditionsofGbandVb}
We have that for all $(u,x)\in E$, 
\begin{align}
\label{eq:integrabilityofGbelowb}
& \E_{u,x} \left( \int_0^{\infty}| G(U_s,X_s)| \I_{\{ X_s <b(U_s) \}}\dd s \right) <\infty, \\
\label{eq:integrabilityVPidy}
&\E_{u,x}\left(\int_0^{\infty} \int_{(-\infty,0)} \widetilde{V}(U_s,X_s+y)\Pi(\dd y) \I_{\{X_s >  b(U_s) \}}   \right)>-\infty.
\end{align}
Moreover, we have that 
\begin{align}
\label{eq:limitofEintV}
\lim_{u,x\rightarrow \infty} &\E_{u,x}\left(\int_0^{\infty} \int_{(-\infty,0)} \widetilde{V}(U_s,X_s+y)\Pi(\dd y) \I_{\{X_s >  b(U_s) \}}   \right)=0.
\end{align}
Furthermore, when $X$ is of finite variation with finite activity (that is, $\Pi(-\infty,0)<\infty))$ we have that
\begin{align}
\label{eq:limitEGbelowD}
\lim_{u,x \rightarrow \infty} & \E_{u,x} \left( \int_0^{\infty}G(U_s,X_s) \I_{\{ X_s <b(U_s) \}}\dd s \right) =0.
\end{align}
\end{lemma}

Next, we show that $V$ satisfies the alternative representation mentioned in the infinite variation case or finite variation case with infinite activity.
\begin{lemma}
\label{eq:alternativeequationinfinitevariationcase}
Suppose that $X$ is of infinite variation of finite variation with infinite activity. Then we have that $V$ and $b$ satisfy equations \eqref{eq:representationforVwithindicatorsfunc} and \eqref{eq:characterisationofb}. 
\end{lemma}
\begin{proof}
Recall that $V$ is continuous on $E$ and, when $X$ is of infinite variation, we have that for any $u>0$, $\lim_{x\downarrow 0} V(u,x)=V(0,0)$ (see Lemma \ref{lemma:continuityofV0x}), implying that $\widetilde{V}$ is continuous on $\R_+\times \R$.  We follow an analogous argument as \cite{Lamberton2013} (see Theorem 3.2). Let $\rho$ be a positive $C^{\infty}$ function with support in $[0,1]\times [0,1]$ and $\int_0^{\infty}\int_0^{\infty} \rho(v,y) \dd v \dd y=1$. For $n\geq 1$, define $\rho_n(v,y)=n^2\rho(nv,ny)$, then $\rho_n$ is $C^{\infty}$ and has compact support in $[0,1/n]\times[0,1/n]$. The function defined by $\widetilde{V}_n(u,x):=(\widetilde{V}* \rho_n)(u,x)=\int_0^{\infty}\int_0^{\infty} \widetilde{V}(u-v,x-y)\rho_n(v,y)\dd v \dd y$ is a $C^{1,2}(\R_+\times \R)$ function such that the derivatives of $\widetilde{V}_n$ and the function $(u,x)\mapsto \int_{(-\infty,-1)} \widetilde{V}_n(u,x+y) \Pi(\dd y)$ are bounded in the set $\R_+\times \R_+$. Moreover, we have that $ \widetilde{V}_n \uparrow \widetilde{V}$ on $E$ when $n\rightarrow \infty$. 

Similar as in \cite{lamberton2008critical} (see proof of Proposition 2.5), due to equation \eqref{eq:freeboundaryproblemconvolution}, we have that for all $(u,x)\in [(1/n,\infty	)\times (1/n,\infty	)] \cap C^+$,
\begin{align}
\label{eq:generatorUXappliedtoVrhon}
\frac{\partial}{\partial u}\widetilde{V}_n(u,x) +\mathcal{A}_{X}(\widetilde{V}_n )(u,x)=-(G*\rho_n)(u,x)  ,
\end{align}
where $\mathcal{A}_{X}$ is the infinitesimal generator of the process $X$. On the other hand, since $V$ vanishes on $D$, we have that $\widetilde{V}_n(u,x)=0$ for $(u,x)\in D$ and $n$ sufficiently large. Hence, since $\Pi$ has no atoms, we deduce that for any $(u,x)\in E$ and $n$ sufficiently large,
\begin{align*}
\frac{\partial}{\partial u}&\widetilde{V}_n(u,x) +\mathcal{A}_{X}(\widetilde{V}_n )(u,x)\\
&=\int_{(-\infty,0)}\widetilde{V}_n(u,x+y)\Pi(\dd y)\\
&=\int_{(-x,0)}\widetilde{V}_n(u,x+y)\Pi(\dd y)+\int_{(-\infty,-x)}\widetilde{V}_n(0,x+y)\Pi(\dd y).
\end{align*}
Therefore, by continuity of $V$ on $E$ and the dominated convergence theorem, we see that
\begin{align*}
\lim_{n \rightarrow \infty }\left[ \frac{\partial}{\partial u}\widetilde{V}_n(u,x) +\mathcal{A}_{X}(\widetilde{V}_n )(u,x)\right]=\int_{(-\infty,0)}\widetilde{V}(u,x+y)\Pi(\dd y)
\end{align*}
for any $(u,x)\in D$. \\

Next, let $u>0$ and $x>0$ fixed, and take $n>0$ and $k>0$ such that $u>1/n >0$ and $x>k\geq  1/n>0$. We apply It\^o formula to $\widetilde{V}_n(u+t\wedge \tau_{k-x}^-,X_{t\wedge \tau_{k-x}^-}+x)$ to get 

\begin{align*}
\widetilde{V}_n&(u+t\wedge  \tau_{k-x}^-,X_{t\wedge \tau_{k-x}^-}+x)\\
&= \widetilde{V}_n(u,x) + M_{t}+\int_0^{t \wedge \tau_{k-x}^-} \left[\frac{\partial}{\partial u}\widetilde{V}_n(u+s,X_s+x) +\mathcal{A}_{X}(\widetilde{V}_n )(u+s,X_s+x) \right]\dd s,
\end{align*}
where $\{ M_{t}, t\geq 0 \}$ is a zero mean martingale. Taking expectations and from \eqref{eq:generatorUXappliedtoVrhon} we get that
\begin{align*}
\E_{x}&\left(\widetilde{V}_n(u+t\wedge  \tau_{k}^-,X_{t\wedge \tau_{k}^-}) \right)\\
&= \widetilde{V}_n(u,x) + \E_{x}\left(\int_0^{t \wedge \tau_{k}^-} \left[\frac{\partial}{\partial u}\widetilde{V}_n(u+s,X_s) +\mathcal{A}_{X}(\widetilde{V}_n )(u+s,X_s) \right]\dd s\right) \\
&= \widetilde{V}_n(u,x) - \E_{x}\left(\int_0^{t \wedge \tau_{k}^-} (G*\rho_n )(u+s,X_s)\I_{\{X_s<b(u+s) \}}\dd s\right) \\
&\qquad +\E_{x}\left(\int_0^{t \wedge \tau_{k}^-} \left[\frac{\partial}{\partial u}\widetilde{V}_n(u+s,X_s) +\mathcal{A}_{X}(\widetilde{V}_n )(u+s,X_s) \right]\I_{\{X_s>b(u+s) \}}\dd s\right),
\end{align*}
where we used the fact that $b$ is finite for all $u>0$, and that $\P_x(X_s=b(u+s))=0$ for all $s>0$ and $x\in \R$ when $X$ is of infinite variation or finite variation with infinite activity (see \cite{sato1999levy}, Theorem 27.4). Since $X_{t} \geq \underline{X}_{\infty}$ for all $t>0$ and $V$ is non-decreasing in each argument we have that
\begin{align*}
0\geq \E_{x}\left(\widetilde{V}_n(u+t\wedge  \tau_{k}^-,X_{t\wedge \tau_{k}^-}) \right)
 \geq -A'_{p-1}-C'_{p-1} \E_{x-1}((-\underline{X}_{\infty})^p )+V(0,0)
&>-\infty,
\end{align*}
where the second inequality follows from equation (\ref{eq:lowerboundforV}) and the last quantity is finite by Lemma \ref{lemma:finitenessofEgn}. Therefore, by the dominated convergence theorem and letting $n,t \rightarrow \infty$ and $k\downarrow 0$, we deduce that
\begin{align}
\label{eq:alternativerepresentationofVformgeq0}
\E_{x}\left(\widetilde{V} (u+\tau_0^-, X_{ \tau_{0}^-}) \right)
&= V(u,x) - \E_{x}\left(\int_0^{ \tau_{0}^-} G(u+s,X_s)\I_{\{X_s<b(u+s) \}}\dd s\right) \nonumber\\
&\qquad +\E_{x}\left(\int_0^{\tau_{0}^-} \int_{(-\infty,0)} \widetilde{V}(u+s,X_s+y) \Pi(\dd y)\I_{\{X_s> b(u+s) \}}\dd s\right)
\end{align}
for all $u>0$ and $x>0$. Note that, since $b(u)<\infty$ for all $u>0$ and $\lim_{u\rightarrow \infty} b(u)=0$, we have that $\lim_{u,x\rightarrow \infty} \widetilde{V}(u,x)=\lim_{u,x\rightarrow \infty} V(u,x)=0$. Hence, since $\widetilde{V}(u,y)=V(0,y)$ for any $u\geq 0$ and $y\leq 0$, and $X$ drifts to infinity we get that
\begin{align*}
&\E_{x}\left(\widetilde{V} (u+\tau_0^-, X_{ \tau_{0}^-}) \right)\\
&=\E_{x}\left(V (0, X_{ \tau_{0}^-}) \I_{\{\tau_0^-<\infty \}}\right)\\
&=V (0,0) \P_x(X_{\tau_0^-}=0,\tau_0^-<\infty )+\E_{x}\left(V (0, X_{ \tau_{0}^-}) \I_{\{X_{ \tau_{0}^-}<0 \}}\right)\\
&=V (0,0) \frac{\sigma^2}{2} W'(x)+\E_{x}\left( \int_0^{\infty}\int_{(-\infty,0)}V (0, X_{s-}+y)\I_{\{X_{s-}+y <0 \}} \I_{\{\underline{X}_{s-}\geq 0	 \}} N(\dd s, \dd y)\right)\\
&=V (0,0) \frac{\sigma^2}{2} W'(x)+\E_{x}\left( \int_0^{\tau_0^-}\int_{(-\infty,0)}V (0, X_{s}+y)\I_{\{X_{s}+y <0 \}}  \dd s \Pi(\dd y)\right),
\end{align*}
where in the second last equality, we used the probability of creeping given in (\ref{eq:probabilityofcreepingdownwards}) (note that $\Phi(0)=0$ since $X$ drifts to infinity) and in the last equality, the compensation formula for Poisson random measures. Then, from above and equation (\ref{eq:alternativerepresentationofVformgeq0}) we see that for any $u>0$ and $x>0$, 
\begin{align*}
V&(u,x)\\
&=\E_{x}\left(\widetilde{V} (u+\tau_0^-, X_{ \tau_{0}^-}) \right)+\E_{x}\left(\int_0^{ \tau_{0}^-} G(u+s,X_s)\I_{\{X_s<b(u+s) \}}\dd s\right) \\
&\qquad -\E_{x}\left(\int_0^{\tau_{0}^-} \int_{(-\infty,0)} \widetilde{V}(u+x,X_s+y) \Pi(\dd y)\I_{\{X_s> b(u+s) \}}\dd s\right)\\
&=V(0,0)\frac{\sigma^2}{2}W'(x)\\
&\qquad -\E_{x} \left(  \int_0^{ \tau_{0}^-} \int_{(-\infty,0)} V(u+s,X_s+y) \I_{\{0<X_s+y<b(u+s) \}} \Pi(\dd y)  \I_{\{ X_s>b(u+s)\}}\dd s \right)\\
&\qquad +\E_{x}\left( \int_0^{ \tau_{0}^-} \left[ G(u+s,X_s) +\int_{(-\infty,-X_s)} V(0,X_s+y)  \Pi(\dd y)  \right] \I_{\{ X_s<b(u+s)\}} \dd s \right),
\end{align*}
where in the last equality we used that $V(u+s,X_s+y)=0$ when $X_s+y\geq b(u+s)$. Moreover, we have that (\ref{eq:characterisationofb}) follows directly from the equation above since $V(u,b(u))=0$ for all $u>0$.
\end{proof}

We define an auxiliary function. For all $(u,x)\in E$, let  
\begin{align}
R(u,x)&=\E_{u,x}\left(\int_0^{\infty} G(U_s,X_s)\I_{\{X_s<b(U_s) \}} \dd s \right)\nonumber\\
\label{eq:definitionofR}
&\qquad- \E_{u,x}\left(\int_0^{\infty} \int_{(-\infty,0)} \widetilde{V}(U_s,X_s+y) \Pi(\dd y)\I_{\{X_s>b(U_s) \}} \dd s \right).
\end{align}
Note from Lemma \ref{lemma:integrativilityconditionsofGbandVb} that $R$ is a well-defined function.
The following Lemma shows that $R$ coincides with $V$.
\begin{lemma}
For any $(u,x)\in E$ we have that 
\begin{align}
\label{eq:integralequationforVnonstopped}
V(u,x)
&=\E_{u,x}\left( \int_0^{\infty} G(U_s,X_s)\I_{\{X_s < b(U_s) \}} \dd s\right)\nonumber\\
&\qquad -\E_{u,x}\left( \int_0^{\infty} \int_{(-\infty,0)} \widetilde{V}(U_s,X_s+y)\Pi(\dd y)\I_{\{X_s > b(U_s) \}}  \dd s\right).
\end{align}
\end{lemma}
\begin{proof}
First, we assume that $X$ is of infinite variation or finite variation with $\Pi(-\infty,0)=\infty$. Let $(u,x)\in E$, from the Markov property applied to the stopping time $\tau_0^+$, the fact that $b$ is non-negative and equation (\ref{eq:expressionforV0xnegative}), we get that for all $x<0$,
\begin{align*}
R(0,x)=\E_x\left( \int_0^{\tau_0^+} G(0,X_s) \dd s \right)+R(0,0)=V(0,x)+R(0,0)-V(0,0).
\end{align*}
Similarly, using the Markov property at time $\tau_0^-$, we get that for any $u>0$ and $x>0$,
\begin{align}
R(u,x)&=\E_x(R(0,X_{\tau_0^-})\I_{\{\tau_0^-<\infty \}} ) +\E_x\left(\int_0^{\tau_0^-} G(u+s,X_s)\I_{\{X_s<b(u+s) \}} \dd s \right) \nonumber  \\
&\qquad- \E_x\left(\int_0^{\tau_0^-} \int_{(-\infty,0)} \widetilde{V}(u+s,X_s+y) \Pi(\dd y)\I_{\{X_s>b(u+s) \}} \dd s \right)\nonumber \\
&=V(u,x)+\E_x([R(0,X_{\tau_0^-})-V(0,X_{\tau_0^-})]\I_{\{\tau_0^-<\infty \}} ) \nonumber\\
\label{eq:expressionforRonC+}
&=V(u,x)+[R(0,0)-V(0,0)]\P_x(\tau_0^-<\infty),
\end{align}
where the second  equality follows from equation (\ref{eq:alternativerepresentationofVformgeq0}) and the last from the expression for $R(0,x)$ deduced above. Then, applying the strong Markov property at time $\tau_D$ to the definition of $R$ (see \eqref{eq:definitionofR}), and the fact that for any $s<\tau_D$ we have that $X_s<b(U_s)$, we get that for any $(u,x)\in E$ such that $x<b(u)$,
\begin{align*}
R(u,x)&=\E_{u,x}\left(\int_0^{\tau_D} G(U_s,X_s) \dd s \right)+\E_{u,x}(R(U_{\tau_D}, X_{\tau_D} ))\\
&=V(u,x)+\E_{u,x}(R(U_{\tau_D}, X_{\tau_D} )),
\end{align*}
where we used that $\tau_D$ is optimal for $V$. Since we are considering the infinite variation case or the finite variation case with infinite activity we have that $b(u)>0$ for all $u>0$ (see Lemma \ref{cor:valueofub}), so then $ X_{\tau_D}>0$ and $U_{\tau_D}>0$. Hence, from equation \eqref{eq:expressionforRonC+} and the equation above we deduce that for any $(u,x)\in E$ such that $x<b(u)$,

\begin{align*}
R(u,x)&=V(u,x)+[R(0,0)-V(0,0)]\E_{u,x}( \P_{X_{\tau_D}}(\tau_0^-<\infty)),
\end{align*}
where we used that $(U_{\tau_D}, X_{\tau_D})\in D$ and that $V$ vanishes on $D$. In particular, taking $u=0$ and $x=0$ in the equation above and rearranging the terms, we conclude that
\begin{align*}
0=[R(0,0)-V(0,0)]\E( \P_{X_{\tau_D}}(\tau_0^-=\infty)). 
\end{align*}
Since $b(u)>0$ for all $u>0$ and $\P_x(\tau_0^-=\infty)>0$ for all $x>0$, the equation above implies that $R(0,0)=V(0,0)$, and then $V(u,x)=R(u,x)$ in this case. For the finite variation with finite activity case, consider the sequence of stopping times, 
\begin{align*}
\tau_b^{(1)}=\inf\{t\geq 0: X_t \geq b(U_t)\},
\end{align*}
and for $k=1,2,\ldots$,
\begin{align*}
\sigma_b^{(k)}=\inf\{t\geq \tau_{b}^{(k)}: X_t < b(U_t) \},\\
\tau_b^{(k+1)}=\inf\{t\geq \sigma_b^{(k)}: X_t\geq b(U_t)  \}. 
\end{align*}
Since $X$ is of finite variation we have that $\tau_b^{(k)}<\sigma_b^{(k)}<\tau_b^{(k+1)}$ for all $k\geq 1$. Let $u>0$ and $x\geq b(u)$, by the Markov property applied to time $\tau_b^{(2)}$ we get that 	
\begin{align*}
R(u,x)&=-\E_{u,x}\left( \int_{0}^{\sigma_{b}^{(1)}} \int_{(-\infty,0)} \widetilde{V}(U_s,X_s+y)\Pi(\dd y)  \dd s\right)\\
&\qquad +\E_{u,x}\left(\I_{\{ \sigma_b^{(1)}<\infty \}} \int_{\sigma_b^{(1)}}^{\tau_b^{(2)}} G(U_s,X_s) \dd s\right)+ \E_{u,x}(R(U_{\tau_b^{(2)}},X_{\tau_b^{(2)}} ) \I_{\{\tau_b^{(2)}<\infty \}}) \\
&=-\E_{u,x}\left(  \int_{0}^{\sigma_{b}^{(1)}} \int_{(-\infty,0)} \widetilde{V}(U_s,X_s+y)\Pi(\dd y)  \dd s\right)\\
&\qquad+\E_{u,x}\left(   \I_{\{ \sigma_b^{(1)}<\infty \}} V(U_{\sigma_b^{(1)}},X_{\sigma_b^{(1)}})\right)+ \E_{u,x}(R(U_{\tau_b^{(2)}},X_{\tau_b^{(2)}} ) \I_{\{\tau_b^{(2)}<\infty \}}) \\
&=\E_{u,x}(R(U_{\tau_b^{(2)}},X_{\tau_b^{(2)}} ) \I_{\{\tau_b^{(2)}<\infty \}}),
\end{align*}
where in the second inequality, we used the Markov property at time $\sigma_b^{(1)}$, the definition of $V$ in terms of the stopping time $\tau_D$ and in the last equality, we used the compensation formula for Poisson random measures. Using an induction argument we can verify that for all $x\geq b(u)$ and $n\geq 1$,
\begin{align*}
R(u,x)&= \E_{u,x}(R(U_{\tau_b^{(n)}},X_{\tau_b^{(n)}} ) \I_{\{\tau_b^{(n)}<\infty \}}).
\end{align*}

It can be shown that for any $(u,x)\in E$, $\lim_{n \rightarrow \infty} \tau_b^{(n)}=\infty$ $\P_{u,x}$-a.s. Hence, by the dominated convergence theorem, the fact that $\lim_{u,x \rightarrow \infty }R(u,x)=0$ (see \eqref{eq:limitofEintV} and \eqref{eq:limitEGbelowD}), that $\lim_{t \rightarrow \infty} U_t=t-g_t\geq \lim_{t\rightarrow \infty } t-g=\infty$ and that $X$ drifts to infinity we get that 
\begin{align*}
R(u,x)=\lim_{n\rightarrow \infty } \E_{u,x}(R(U_{\tau_b^{(n)}},X_{\tau_b^{(n)}} ) \I_{\{\tau_b^{(n)}<\infty \}})=0
\end{align*}
for all $u>0$ and $x\geq b(u)$. Next, take $x<b(u)$, applying the strong Markov property and using that $\tau_b^{(1)}$ is optimal for $V$ we get that 
\begin{align*}
R(u,x)=\E_{u,x}\left(\int_0^{\tau_b^{(1)}} G(U_s,X_s)\dd s \right)+ \E_{u,x}(R(U_{\tau_b^{(1)}},X_{\tau_b^{(1)}}))=V(u,x).
\end{align*}
Hence, we conclude that for all $(u,x)\in E$,
\begin{align*}
V(u,x)=R(u,x).
\end{align*}
The proof is now complete.
\end{proof}
Now we are ready to show that in either case regarding the variation of $X$, the equations stated in Theorem \ref{thm:characterisationofbandV} hold.

\begin{lemma}
\label{lemma:alternativerepforV}
The quadruplet $(V,b,V(0,0),u_b)$ satisfy equations (\ref{eq:representationforVwithindicatorsfunc})-(\ref{eq:characterisationofub}) and equation (\ref{eq:intV+Gispositive}).
\end{lemma}
\begin{proof}
We know from Lemma \ref{eq:alternativeequationinfinitevariationcase} that equations (\ref{eq:representationforVwithindicatorsfunc}) and (\ref{eq:characterisationofb}) hold in the infinite variation case or in the finite variation case with finite activity. Then suppose that $X$ is of finite variation and $\Pi(-\infty,0)<\infty$. The strong Markov property applied at time $\tau_0^-$ in (\ref{eq:integralequationforVnonstopped}) implies that (\ref{eq:alternativerepresentationofVformgeq0}) also holds in this case. Then, proceeding as in Lemma \ref{eq:alternativeequationinfinitevariationcase}  (see argument below equation (\ref{eq:alternativerepresentationofVformgeq0})) we see that (\ref{eq:representationforVwithindicatorsfunc}) and (\ref{eq:characterisationofb}) also hold in the finite variation case with finite activity. Moreover, the assertions about $u_b$ and equation (\ref{eq:characterisationofub}) follow from Lemma \ref{cor:valueofub}, the lower bound for $V(0,0)$ follows from Remark \ref{rem:boundsforV} and (\ref{eq:intV+Gispositive}) holds due to Lemma \ref{lemma:auxiliaryfunctionLambda}. 
\\

We now proceed to show that (\ref{eq:charactersationofV00}) is satisfied for $V(0,0)$. Taking $u=x=0$ in (\ref{eq:integralequationforVnonstopped}) and using Fubini's theorem we have that 
\begin{align*}
V(0,0)
&=\E\left( \int_0^{\infty} G(U_s,X_s) \I_{\{X_s < b(U_s) \}} \dd s \right)\\
&\qquad-\E\left( \int_0^{\infty} \int_{(-\infty,0)} \widetilde{V}(U_s,X_s+y) \Pi(\dd y) \I_{\{X_s > b(U_s) \}} \dd s \right)\\
&=\E\left( \int_0^{\infty} G(0,X_s) \I_{\{X_s \leq  0 \}} \dd s \right)
+\E\left( \int_0^{\infty} G(U_s,X_s) \I_{\{0<X_s < b(U_s) \}} \dd s \right)\\
&\qquad-\E\left( \int_0^{\infty} \int_{(-\infty,0)} \widetilde{V}(U_s,X_s+y) \Pi(\dd y) \I_{\{X_s > b(U_s) \}} \dd s \right)\\
&=\int_{(-\infty,0]} G(0,z) \int_0^{\infty} \P(X_s \in \dd z) \dd s\\
&\qquad+\int_{(0,\infty)}\int_{(0,b(u))}   G(u,z)  \int_0^{\infty} \P(U_s \in \dd u, X_s \in \dd z) \dd s \\
&\qquad-  \int_{(0,\infty)} \int_{(b(u),\infty)} \int_{(-\infty,0)} \widetilde{V}(u,z+y) \Pi(\dd y) \int_0^{\infty} \P(U_s\in \dd u, X_s \in \dd z)\dd s ,
\end{align*}
where in the second equality we used the fact that $b$ is non-negative and that $U_s=0$ if and only if $X_s\leq 0$. From the fact that $G(0,z)=-\E_z(g^{p-1})$ for any $z< 0$ and the formulas for the $0$-potential density of $X$ and $(U,X)$ (see equations  (\ref{eq:qpotentialdensitywithoutkilling}) and (\ref{eq:potentialmeasureUXatzero})), respectively, we obtain that
\begin{align*}
V(0,0)
&=-\frac{1}{\psi'(0+)}\int_0^{\infty} \E_{-z}(g^{p-1}) [1-\psi'(0+)W(z)]\dd z\\
&\qquad	+\frac{1}{\psi'(0+)}\int_{0}^{\infty}\int_{(0,b(s))}   G(s,z)  \frac{z}{s}\P(X_s \in \dd z) \dd s \\
&\qquad-  \frac{1}{\psi'(0+)}\int_{0}^{\infty} \int_{(b(s),\infty)} \int_{(-\infty,0)} \widetilde{V}(s,z+y) \Pi(\dd y)\frac{z}{s}\P(X_s \in \dd z) \dd s \\
&=-\frac{1}{\psi'(0+)}\int_0^{\infty} \E_{-z}(g^{p-1}) [1-\psi'(0+)W(z)]\dd z \\
&\qquad+\frac{1}{\psi'(0+)}\E\left( \int_{0}^{\infty}   G(s,X_s)\frac{X_s}{s}  \I_{\{0<X_s<b(s) \}} \dd s \right) \\
&\qquad-\frac{1}{\psi'(0+)} \E\left( \int_0^{\infty} \int_{(-\infty,0)} \widetilde{V}(s,X_s+y) \Pi(\dd y)\frac{X_s}{s} \I_{\{X_s>b(s) \}} \dd s \right).
\end{align*}
Then, equation (\ref{eq:charactersationofV00}) holds by recalling that $\widetilde{V}(u,x)=V(u,x)$ when $u>0$ and $x>0$, and $\widetilde{V}(u,x)=V(0,x)$ when $x\leq 0$ for any $u\geq 0$.
\end{proof}
We finish the first part of the proof by showing that the derivative of $V$ at $(0,0)$ exists when there is a Brownian motion component.
\begin{lemma}
\label{lemma:proofofequationforV0}
The function $V$ satisfies equation (\ref{eq:derivativeofVat00}) when $\sigma>0$.
\end{lemma}

\begin{proof}
We first note that the strict inequalities in the indicator functions in Equation \eqref{eq:integralequationforVnonstopped} can be changed to inequalities when $x=u=0$. Indeed, due to \eqref{eq:potentialmeasureUXatzero} and Fubini's theorem we have   
\begin{align*}
\E&\left( \int_0^{\infty} G(U_s,X_s)\I_{\{ 0<X_s<b(U_s)\}}\dd s \right)\\
&=\int_{(0,\infty)}\int_{(0,\infty)}  G(v,y)\I_{\{ 0<y<b(v)\}} \int_0^{\infty} \P(U_s\in \dd v, X_s\in \dd y )\dd s\\
&=\int_{(0,\infty)}\int_{0}^{b(v)}  G(v,y)  \Phi'(0)\P(\tau_y^+ \in \dd v)\dd y\\  
&=\int_{(0,\infty)}\int_{(0,\infty)}  G(v,y)\I_{\{ 0<y\leq b(v)\}} \int_0^{\infty} \P(U_s\in \dd v, X_s\in \dd y )\dd s\\
&=\E\left( \int_0^{\infty} G(U_s,X_s)\I_{\{0< X_s\leq b(U_s)\}}\dd s \right).
\end{align*}
Similarly, we have that 
\begin{align*}
\E &\left(  \int_0^{\infty} \int_{(-\infty,0)} \widetilde{V}(U_s,X_s+y) \Pi(\dd y)  \I_{\{ X_s>b(U_s)\}}\dd s \right)\\
&=\E \left(  \int_0^{\infty} \int_{(-\infty,0)} \widetilde{V}(U_s,X_s+y) \Pi(\dd y)  \I_{\{ X_s\geq b(U_s)\}}\dd s \right).
\end{align*}
Hence, from equation (\ref{eq:integralequationforVnonstopped}) and the dominated convergence theorem (see Lemma \ref{lemma:integrativilityconditionsofGbandVb}) we obtain that 
\begin{align*}
V(0,0)
 &= \E\left( \int_0^{\infty} G(U_s,X_s)\I_{\{ X_s<b(U_s)\}}\dd s \right)\\
 &\qquad-\E \left(  \int_0^{\infty} \int_{(-\infty,0)} \widetilde{V}(U_s,X_s+y) \Pi(\dd y)  \I_{\{ X_s\geq b(U_s)\}}\dd s \right)\\
 &=\lim_{\delta \downarrow 0}\E\left( \int_0^{\infty}  G(U_s+\delta,X_s)\I_{\{ X_s<b(U_s+\delta)\}}\dd s \right)\\
 &\qquad-\lim_{\delta \downarrow 0}\E \left(  \int_0^{\infty} \int_{(-\infty,0)} \widetilde{V}(U_s+\delta,X_s+y) \Pi(\dd y)  \I_{\{ X_s\geq b(U_s+\delta)\}}\dd s \right),
\end{align*}
where we used that $G$ and $V$ are continuous (since $\sigma>0$ and then $W$ is continuous on $\R$), that the mappings $s \mapsto \I_{\{ x\geq s\}}$ and $s \mapsto \I_{\{ x<s\}}$ are left-continuous and that $b(U_s+\delta)\uparrow b(U_s)$ when $\delta \downarrow 0$, for any $s\geq 0$. Then, using once again \eqref{eq:potentialmeasureUXatzero}, Fubini's theorem and the equation above we can easily see that 
\begin{align*} 
V(0,0) &= \lim_{\delta \downarrow 0}\E\bigg( \int_0^{\infty} [K_1(U_s+\delta,X_s)+ K_2(U_s+\delta,X_s)]\dd s \bigg),
\end{align*}
where for any $(u,x)\in E$,  
\begin{align*}
K_1(u,x)&:=G(u,x)\I_{\{x\leq b(u) \}},\\
K_2(u,x)&:= -\int_{(-\infty,0)} \widetilde{V}(u,x+y)\Pi(\dd y)\I_{\{x>b(u) \}}.
\end{align*}
Therefore, from Lemma \ref{lemma:formulaforK1andK2} we deduce that
 \begin{align*}
V(0,0)= \lim_{\delta \downarrow 0} \lim_{\varepsilon \downarrow 0} \frac{ \E_{\varepsilon} \left(\I_{\{\tau_0^-<\infty \}}K^-(\delta, X_{\tau_0^-}-\varepsilon ) \right) +K^+(\delta,\varepsilon) }{\psi'(0+)W(\varepsilon)},
\end{align*}
where for all $\delta >  0$ and $x\leq 0$,
\begin{align*}
K^-(\delta, x)=\E_x\left(\int_0^{\tau_0^+} [K_1(\delta,X_r)+K_2(\delta,X_r)] \dd r \right)
\end{align*}
and for all $\delta,x>0$,
\begin{align*}
K^+(\delta,x)=\E_x\left(\int_0^{\tau_0^-} [K_1(\delta+s,X_r)+K_2(\delta+s,X_r)] \dd r \right).
\end{align*}
Using the fact that $b$ is non-negative and $W(x)=0$ for all $x<0$ (and then $G(\delta,x)=G(0,x)$ for all $x<0$), we have that for all $x<0$,
\begin{align*}
K^-(\delta, x)=\E_{x}\left(\int_0^{\tau_0^+} G(\delta,X_s) \dd s \right)=V(0,x)-V(0,0),
\end{align*}
where the last equality follows from the expression of $V$ in terms of the stopping time $\tau_D$. Moreover, from equation (\ref{eq:alternativerepresentationofVformgeq0}), we deduce that
\begin{align*}
K^+(\delta,\varepsilon)
&=V(\delta,\varepsilon)-\E_{\varepsilon} (V(0,X_{\tau_0^-}) \I_{\{\tau_0^-<\infty\}}),
\end{align*}
for all $\delta >0$ and $x>0$. Hence, by substituting the expressions for $K^+$ and $K^-$, rearranging the terms and by dominated convergence theorem we see that

\begin{align*}
V(0,0)&=\lim_{\delta \downarrow 0 }  \lim_{\varepsilon \downarrow 0 }  \frac{\E( V(0,X_{\tau_{-\varepsilon}^-})\I_{\{\tau_{-\varepsilon}^-<\infty\}})-\E(V(0,X_{\tau_{-\varepsilon}^-}+\varepsilon) \I_{\{\tau_{-\varepsilon}^-<\infty\}})}{ \psi'(0+)W(\varepsilon)}   \\
&\qquad+\lim_{\delta \downarrow 0 }  \lim_{\varepsilon \downarrow 0 } \frac{V(\delta,\varepsilon)-V(0,0)\P_{\varepsilon}(\tau_0^-<\infty) }{ \psi'(0+)W(\varepsilon)}\\
&= \lim_{\varepsilon \downarrow 0 }  \frac{\E( V(0,X_{\tau_{-\varepsilon}^-})\I_{\{\tau_{-\varepsilon}^-<\infty\}})-\E(V(0,X_{\tau_{-\varepsilon}^-}+\varepsilon) \I_{\{\tau_{-\varepsilon}^-<\infty\}})}{ \psi'(0+)W(\varepsilon)}\\
&\qquad   + \frac{\sigma^2}{2 \psi'(0+)}  \frac{\partial }{\partial x} V_+(0,0)+V(0,0),
\end{align*}
where in the last equality we used that $\P_{\varepsilon}(\tau_0^-<\infty)=1-\psi'(0+)W(\varepsilon)$ (see equation (\ref{eq:laplacetransoformoftau0minus})) and the fact that $W'(0)=2/\sigma^2$. 
Therefore, from Lemma \ref{lemma:convergencederivativeofVXtaue} we see that  
\begin{align*}
\lim_{\varepsilon \downarrow 0 }  \frac{\E( V(0,X_{\tau_{-\varepsilon}^-})\I_{\{\tau_{-\varepsilon}^-<\infty\}})-\E(V(0,X_{\tau_{-\varepsilon}^-}+\varepsilon) \I_{\{\tau_{-\varepsilon}^-<\infty\}})}{ \psi'(0+)W(\varepsilon)}&=-\frac{\sigma^2}{2 \psi'(0+)}  \frac{\partial }{\partial x} V_-(0,0).
\end{align*}
Then, we deduce that
\begin{align*}
V(0,0)= \frac{\sigma^2}{2 \psi'(0+)}\left[ -\frac{\partial }{\partial x} V_-(0,0)+ \frac{\partial }{\partial x} V_+(0,0)\right]+V(0,0),
\end{align*}
and we conclude that (\ref{eq:derivativeofVat00}) holds. The proof is now complete.
\end{proof}

Next, we show the uniqueness claim. Suppose that there exist continuous functions $H$ and $c$ on $E$ and $\R_+$, respectively, and real numbers $H_0<0$ and $u_H>0$ such that the conclusions of the theorem hold. Specifically, suppose that $H$ is a non-positive continuous real valued function on $E$, $c$ is a continuous real valued function on $(0,\infty)$ such that $c \geq h \geq 0$ and $H_0 \in (-\frac{1}{p}\E(g^p),0)$ such that equations (\ref{eq:representationforVwithindicatorsfunc})-(\ref{eq:characterisationofb}) hold. That is, we assume that $H$, $H_0$ and $c$ are solutions to the equations
\begin{align}
H&(u,x) \nonumber\\
&=H_0\frac{\sigma^2}{2}W'(x)\nonumber\\
&\qquad-\E_{x} \left(  \int_0^{ \tau_{0}^-} \int_{(-\infty,0)} H(u+s,X_s+y) \I_{\{0<X_s+y<c(u+s) \}} \Pi(\dd y)  \I_{\{ X_s>c(u+s)\}}\dd s \right) \nonumber\\
\label{eq:representationforHwithindicatorsfunc}
&\qquad +\E_{x}\left( \int_0^{ \tau_{0}^-} \left[ G(u+s,X_s) +\int_{(-\infty,-X_s)} H(0,X_s+y)  \Pi(\dd y)  \right] \I_{\{ X_s<c(u+s)\}} \dd s \right),
\end{align}
for $u>0$ and $x>0$, 
\begin{align}
\label{eq:Hat00}
H_0&=-\frac{1}{\psi'(0+)}\int_0^{\infty} \E_{-z}(g^{p-1}) [1-\psi'(0+)W(z)]\dd z\\
&\qquad +\frac{1}{\psi'(0+)}\E\left( \int_{0}^{\infty}   G(s,X_s)\frac{X_s}{s}  \I_{\{0<X_s<c(s) \}} \dd s \right) \nonumber\\
&\qquad-\frac{1}{\psi'(0+)} \E\left( \int_0^{\infty} \int_{(-\infty,0)} H(s,X_s+y)\I_{\{0<X_s+y<c(s) \}} \Pi(\dd y)\frac{X_s}{s} \I_{\{X_s>c(s) \}} \dd s \right)\nonumber\\
&\qquad-\frac{1}{\psi'(0+)} \E\left( \int_0^{\infty} \int_{(-\infty,0)} H(0,X_s+y)\I_{\{X_s+y\leq 0) \}} \Pi(\dd y)\frac{X_s}{s} \I_{\{X_s>c(s) \}} \dd s \right),
\end{align}
and
\begin{align}
0&=H_0\frac{\sigma^2}{2}W'(c(u))\nonumber\\
&\qquad-\E_{c(u)} \left(  \int_0^{ \tau_{0}^-} \int_{(-\infty,0)} H(u+s,X_s+y) \I_{\{0<X_s+y<c(u+s) \}} \Pi(\dd y)  \I_{\{ X_s>c(u+s)\}}\dd s \right) \nonumber\\
\label{eq:representationforHwithindicatorsfuncforc}
&\qquad +\E_{c(u)}\left( \int_0^{ \tau_{0}^-} \left[ G(u+s,X_s) +\int_{(-\infty,-X_s)} H(0,X_s+y)  \Pi(\dd y)  \right] \I_{\{ X_s<c(u+s)\}} \dd s \right),
\end{align}
for $u<u_H$, where for any $x\leq 0$, 
\begin{align}
\label{eq:Hwhenxnegative}
 H(0,x)= -\int_{0}^{-x} \int_{0}^{\infty}  \E_{-u-z}(g^{p-1})W'(u) \dd u  \dd z +H_0.
\end{align}
The value $u_H$ is such that $u_H=\infty$ when $X$ is of infinite variation or $X$ is of finite variation with infinite activity. Otherwise, let $u_H$ be the solution of
\begin{align}
\label{eq:definitionofuH}
G(u,0)-\int_{(-\infty,0)} \int_{0}^{-y}\int_{0}^{\infty} \E_{-u-z}(g^{p-1})W'(u) \dd u  \dd z\Pi(\dd y)+H_0\Pi(-\infty,0)=0.
\end{align}
\\
Moreover, assume that $c(u)>0$ for all $u<u_H$ and $c(u)=0$ for all $u\geq u_H$, and that 
\begin{align}
\label{eq:inequalityVPidyGgeq0}
\int_{(-\infty,-x)} \widetilde{H}(u,x+c(u)+y)\Pi(\dd y)+G(u,c(u)+x)\geq 0
\end{align}
for all $u<u_H$ and $x>0$, where $\widetilde{H}$ is the extension of $H$ to the set $\R_+\times \R$ as in (\ref{eq:extensionofftoR+R}). That is, 
\begin{align}
\label{eq:extensionofHtoR+R}
\widetilde{H}(u,x)=\left\{
\begin{array}{cc}
H(u,x), & \text{for } u>0  \text{ and } x>0,\\
H(0,x), & \text{for }u\geq 0 \text{ and } x\leq 0,\\
H(0,0), & \text{for }u=0 \text{ and } x>0.
\end{array}
\right.
\end{align}
Note that, using the same arguments as the ones used in Lemma \ref{eq:alternativeequationinfinitevariationcase} (see argument below Equation (\ref{eq:alternativerepresentationofVformgeq0})), we can see that (\ref{eq:representationforHwithindicatorsfunc}) and (\ref{eq:representationforHwithindicatorsfuncforc}) are equivalent to

\begin{align}
\label{eq:definitionofU}
H(u,x)&= \E_x(H(0, X_{\tau_0^-})\I_{\{\tau_0^-<\infty \}})+\E_{x}\left( \int_0^{ \tau_{0}^-} G(u+s,X_s)\I_{\{ X_s<c(u+s)\}}\dd s \right)\nonumber\\
&\qquad -\E_{x} \left(  \int_0^{ \tau_{0}^-} \int_{(-\infty,0)} \widetilde{H}(u+s,X_s+y)\I_{\{ X_s+y<c(u+s)\}} \Pi(\dd y)  \I_{\{ X_s>c(u+s)\}}\dd s \right)
\end{align}
for all $(u,x)\in E$, and
\begin{align}
\label{eq:Uandcattheboundary}
\E_{c(u)}&(H(0, X_{\tau_0^-})\I_{\{\tau_0^-<\infty \}})+\E_{c(u)}\left( \int_0^{ \tau_{0}^-} G(u+s,X_s)\I_{\{ X_s<c(u+s)\}}\dd s \right)\nonumber\\
&=\E_{c(u)} \left(  \int_0^{ \tau_{0}^-} \int_{(-\infty,0)} \widetilde{H}(u+s,X_s+y)\I_{\{ X_s+y<c(u+s)\}}\Pi(\dd y)  \I_{\{ X_s>c(u+s)\}}\dd s \right)
\end{align}
for any $u<u_H$. Following a similar proof than \cite{duToit2008} we show that $c=b$, which implies that $H=V$, $H_0=V(0,0)$ and $u_H=u_b$.\\

First, we show that $H$ has an alternative representation.

\begin{lemma}
\label{lemma:alternativerepresentationofH}
For all $(u,x)\in E$ we have that
\begin{align}
H(u,x)&=\E_{u,x}\left( \int_0^{\infty} G(U_s,X_s)\I_{\{ X_s<c(U_s)\}}\dd s \right)\nonumber\\
\label{eq:alternativerepresentationforH}
& \qquad -\E_{u,x} \left(  \int_0^{\infty} \int_{(-\infty,0)} \widetilde{H}(U_s,X_s+y)\I_{\{ X_s+y<c(U_s)\}} \Pi(\dd y)  \I_{\{ X_s>c(U_s)\}}\dd s \right).
\end{align}
Moreover, the same conclusion holds if, in the case that $\sigma>0$, instead of (\ref{eq:Hat00}) we assume that 
\begin{align}
\label{eq:derivativeofHwhensigmapos}
\frac{\partial}{\partial x} H_+ (0,0) =\frac{\partial}{\partial x} H _- (0,0),
\end{align}
where $\frac{\partial}{\partial x} H_+ (u,0)$ and $\frac{\partial}{\partial x} H_- (0,0)$ are the right and left derivatives of $x\mapsto  H(u,x)$ and $x\mapsto H(0,x)$ at zero, respectively and $\frac{\partial}{\partial x} H_+ (0,0)=\lim_{u\downarrow 0} \frac{\partial}{\partial x} H_+ (u,0)$.\\
\end{lemma}
\begin{proof}
Define for all $(u,x)\in E$ the function

\begin{align*}
K(u,x)&= \E_{u,x}\left( \int_0^{\infty} G(U_s,X_s)\I_{\{ X_s<c(U_s)\}}\dd s \right)\\
&\qquad-\E_{u,x} \left(  \int_0^{\infty} \int_{(-\infty,0)} \widetilde{H}(U_s,X_s+y) \I_{\{ X_s+y<c(U_s)\}}\Pi(\dd y)  \I_{\{ X_s>c(U_s)\}}\dd s \right)  .
\end{align*}
In a analogous way as in Lemma \ref{lemma:alternativerepforV}, from (\ref{eq:qpotentialdensitywithoutkilling}) and (\ref{eq:potentialmeasureUXatzero}), we have that for any spectrally negative L\'evy process $X$,

\begin{align*}
K(0,0)&= \E\left( \int_0^{\infty} G(U_s,X_s)\I_{\{ X_s<c(U_s)\}}\dd s \right)\\
&\qquad-\E_{u,x} \left(  \int_0^{\infty} \int_{(-\infty,0)} \widetilde{H}(U_s,X_s+y) \I_{\{ X_s+y<c(U_s)\}}\Pi(\dd y)  \I_{\{ X_s>c(U_s)\}}\dd s \right)  \\
&=\frac{1}{\psi'(0+)}\int_0^{\infty} \E_{-z}(g^{p-1}) [1-\psi'(0+)W(z)]\dd z\\
&\qquad+\frac{1}{\psi'(0+)}\E\left( \int_{0}^{\infty}   G(s,X_s)\frac{X_s}{s}  \I_{\{0<X_s<c(s) \}} \dd s \right) \nonumber\\
&\qquad-\frac{1}{\psi'(0+)} \E\left( \int_0^{\infty} \int_{(-\infty,0)} \widetilde{H}(s,X_s+y) \I_{\{X_s+y<c(s)\}} \Pi(\dd y)\frac{X_s}{s} \I_{\{X_s>c(s) \}} \dd s \right)\\
&=H_0\\
&=H(0,0).
\end{align*}
Moreover, by the Markov property, the fact that $X$ creeps upwards, $c$ is a non-negative curve and the definition of $H(0,x)$, for $x<0$ (see (\ref{eq:Hwhenxnegative})), we have that for $u=0$ and $x<0$,
\begin{align}
\label{eq:functionKforxnegative}
K(0,x)&= \E_x\left( \int_0^{\tau_0^+} G(U_s,X_s)\dd s\right)+K(0,0)=H(0,x).
\end{align}
Then, taking $u>0$ and $x>0$, by the strong Markov property at time $\tau_0^-$ and equation (\ref{eq:definitionofU}),
\begin{align*}
K(u,x)&=\E_{x}(K(0,X_{\tau_0^-})\I_{\{\tau_0^-<\infty \}} )+\E_{x}\left( \int_0^{\tau_0^-} G(u+s,X_s)\I_{\{ X_s<c(u+s)\}}\dd s \right) \nonumber \\
&\qquad -\E_{x} \left(  \int_0^{\tau_0^-} \int_{(-\infty,0)} \widetilde{H}(u+s,X_s+y) \I_{\{ X_s+y<c(u+s)\}}\Pi(\dd y)  \I_{\{ X_s>c(u+s)\}}\dd s \right) \nonumber \\
&=H(u,x).
\end{align*}
If, in the case that $\sigma>0$, we assume that $H$ and $c$ satisfy equations (\ref{eq:Hwhenxnegative}), (\ref{eq:definitionofU}), (\ref{eq:Uandcattheboundary})  and (\ref{eq:derivativeofHwhensigmapos}). From formula (\ref{eq:calculationofUsXsintegral}) (in a similar way as in the proof of Lemma \ref{lemma:alternativerepforV}) we obtain that 
\begin{align*}
K(0,0)&= \frac{\sigma^2}{2 \psi'(0+)}\left[  -\frac{\partial }{\partial x} H_-(0,0)+ \frac{\partial }{\partial x} H_+(0,0)\right]+H(0,0)\\
&=H(0,0).
\end{align*}
The rest of the proof remains unchanged.
\end{proof}
Define the set $D_c=\{ (u,x) \in E: x \geq c(u)\}$. We show in the following lemma that $H$ vanishes in $D_c$ so that $D_c$ corresponds to the ``stopping set'' of $H$.
\begin{lemma}
\label{lemma:HvanishesinDc}
We have that $H(u,x)=0$ for all $(u,t)\in D_c$.
\end{lemma}

\begin{proof}
Note that, from equations (\ref{eq:definitionofU}) and (\ref{eq:Uandcattheboundary}), we know that $H(u,c(u))=0$ for all $u\in (0,u_H)$. Let $(u,x)\in D_c$ such that $x>c(u)$ and define $\sigma_c$ as the first time that $(U,X)$ exits $D_c$, that is,
\begin{align*}
\sigma_c=\inf\{s\geq 0: X_{s} < c(U_s) \}.
\end{align*}
From the fact that $X_{r}\geq c(U_r)$ for all $r< \sigma_c$, the Markov property and representation (\ref{eq:alternativerepresentationforH}) of $H$ we can see that 
\begin{align*}
H(u,x)&=\E_{u,x}(H(U_{\sigma_c}, X_{\sigma_c})\I_{\{ \sigma_c<\infty \}} )+\E_{u,x}\left( \int_0^{ \sigma_c} G(U_s,X_s)\I_{\{ X_s<c(U_s)\}}\dd s \right)\\
&\qquad -\E_{u,x} \left(  \int_0^{ \sigma_c} \int_{(-\infty,0)} \widetilde{H}(U_s,X_s+y) \I_{\{ X_s+y<c(U_s)\}}\Pi(\dd y)  \I_{\{ X_s>c(U_s)\}}\dd s \right)\\
&= \E_{u,x}(H(U_{\sigma_c}, X_{\sigma_c})\I_{\{ \sigma_c<\infty , X_{\sigma_c}< c(U_s) \}} )\\
&\qquad-\E_{u,x} \left(  \int_0^{ \sigma_c} \int_{(-\infty,0)} \widetilde{H}(U_s,X_s+y) \I_{\{ X_s+y<c(U_s)\}}\Pi(\dd y) \dd s \right),
\end{align*}
where the last equality follows from the fact that $\P_{x}(X_{\sigma_c}=c(u+\sigma_c))>0$ only when $\sigma>0$ and then $H(u,c(u))=0$ for all $u>0$ (since $u_H=\infty$ in this case). Then, since $H\leq 0$, applying the compensation formula for Poisson random measures and the fact that $\sigma_c\leq \tau_0^-$ (since $c(u)\geq 0$ for all $u>0$), we get
\begin{align*}
&\E_{u,x}(H(U_{\sigma_c}, X_{\sigma_c})\I_{\{ \sigma_c<\infty \}} \I_{\{ X_{\sigma_c}< c(U_s) \}} )\\
&=\E_{x}\left( \int_0^{\infty} \int_{(-\infty,0)} \I_{\{X_{r} \geq  c(u+r) \text{ for all } r<s \}}\I_{\{X_{s-}+y <  c(u+s) \}} \widetilde{H}(u+s,X_{s-}+y) N(\dd s,\dd y)\right)\\
&=\E_{x}\left( \int_0^{\infty} \int_{(-\infty,0)} \I_{\{X_{r} \geq  c(u+r) \text{ for all } r<s \}}\I_{\{X_{s-}+y <  c(u+s) \}} \widetilde{H}(u+s,X_{s-}+y)\Pi(\dd y)\dd s\right)\\
&=\E_{u,x}\left( \int_0^{\sigma_c} \int_{(-\infty,0)} \widetilde{H}(U_s,X_{s}+y)\I_{\{X_s+y<c(U_s) \}}\Pi(\dd y)\dd s\right).
\end{align*}
Hence, we have that $H(u,x)=0$ for all $(u,x)\in D_c$ as claimed.
\end{proof}

The following Lemma states that $H$ dominates the value function $V$.
\begin{lemma}
\label{lemma:HgeqV}
We have that $H(u,x)\geq V(u,x)$ for all $(u,x)\in E$.

\end{lemma}

\begin{proof}
If $(u,x)\in D_c$ we have the inequality
\begin{align*}
H(u,x)=0\geq V(u,x).
\end{align*}
Now we show that the inequality also holds in $E\setminus D_c$. Consider the stopping time
\begin{align*}
\tau_c=\inf\{ s \geq 0 : X_{s} \geq  c(U_s) \}.
\end{align*}
Then, using the Markov property and equation (\ref{eq:alternativerepresentationforH}), we get that for all $(u,x)\in E$ with $x<c(u)$, 

\begin{align*}
H(u,x)&=\E_{u,x}\left( H(U_{\tau_c} , X_{\tau_c})\right)+ \E_{u,x}\left( \int_0^{\tau_c} G(U_s,X_s)\I_{\{ X_s<c(U_s)\}}\dd s \right)\\
&\qquad -\E_{u,x} \left(  \int_0^{\tau_c} \int_{(-\infty,0)} \widetilde{H}(U_s,X_s+y) \Pi(\dd y)  \I_{\{ X_s>c(U_s)\}}\dd s \right)  \\
&=\E_{u,x}\left( H(U_{\tau_c} , c(U_{\tau_c}))\right)+ \E_{u,x}\left( \int_0^{\tau_c} G(U_s,X_s)\dd s \right),
\end{align*}
where in the second equality we used the fact $X$ creeps upwards and $\tau_c<\infty$. Note that, since for any $t>0$, $X_t >0$ if and only if $U_t>0$ and $c(u)>0$ for all $u<u_H$, we have that $U_{\tau_c}<u_H$ so then  $c(U_{\tau_c})>0$, and hence $ H(U_{\tau_c} , c(U_{\tau_c}))=0$. Therefore

\begin{align*}
H(u,x)=\E_{u,x}\left( \int_0^{\tau_c} G(U_s,X_s)\dd s \right) \geq V(u,x),
\end{align*}
where the inequality follows from the definition of $V$ as per (\ref{eq:optimalstoppingproblem}).
\end{proof}
It turns out that the fact that $H$ dominates $V$ implies that $b$ dominates the curve $c$. This fact is shown in the following Lemma.
\begin{lemma}
\label{lemma:blesthanc}
We have that $b(u) \geq c(u)$ for all $u>0$.
\end{lemma}

\begin{proof}
Note that in the case that $X$ is of finite variation with $\Pi(-\infty,0)<\infty$, we have that $c(u)=0\leq b(u)$ for all $u>u_H$. We proceed by contradiction. Suppose that there exists $u_0>0$ such that $b(u_0)<c(u_0)$. Then, in the case that $X$ is of finite variation with $\Pi(-\infty,0)<\infty$, it holds that $u_0<u_H$. Take $x>c(u_0)$ and consider the stopping time
\begin{align*}
\sigma_b&=\inf\{s>0: X_s<b(U_s) \}.
\end{align*}
Then, from the Markov property and the representation of $H$ given in (\ref{eq:alternativerepresentationforH}), we have that
\begin{align*}
H(u_0,x)&=\E_{u_0,x}\left(H(U_{\sigma_b^-}, X_{\sigma_b^-} ) \right)+ \E_{u_0,x}\left(\int_0^{\sigma_{b}^-} G(U_s,X_s)\I_{\{X_s<c(U_s) \}}\dd s \right) \nonumber\\
&\qquad- \E_{u_0,x}\left(\int_0^{\sigma_{b}^-} \int_{(-\infty,0)} \widetilde{H}(U_s,X_s+y)\Pi(\dd y)\I_{\{X_s>c(U_s) \}}\dd s \right).
\end{align*}
Moreover, since $V(u,x)=0$ for $(u,x)\in D$ and $0\geq H\geq V$, we have that

\begin{align*}
\E_{u_0,x}\left(H(U_{\sigma_b^-}, X_{\sigma_b^-} ) \right)&=
\E_{u_0,x}\left(H(U_{\sigma_b^-}, X_{\sigma_b^-} )\I_{\{ X_{\sigma_b^-} <b(U_{\sigma_b^-}) \}} \right)\\
&=\E_{u_0,x}\left(\int_0^{\sigma_{b}^-} \int_{(-\infty,0)} \widetilde{H}(U_s,X_s+y)\I_{\{X_s+y \leq b(U_s) \}} \Pi(\dd y) \dd s \right)\\
&= \E_{u_0,x}\left(\int_0^{\sigma_{b}^-} \int_{(-\infty,0)} \widetilde{H}(U_s,X_s+y)\Pi(\dd y)\dd s \right),
\end{align*}
where the second equality follows from the compensation formula for Poisson random measures. Hence, combining the two equations above and from the fact that $x>c(u_0)$, then $H(u_0,x)=0$, we get

\begin{align*}
0= \E_{u_0,x}\left(\int_0^{\sigma_{b}^-} \left[ G(U_s,X_s) +\int_{(-\infty,0)} \widetilde{H}(U_s,X_s+y)\Pi(\dd y)\right]\I_{\{X_s<c(U_s) \}}\dd s \right).
\end{align*}
Due to the continuity of $b$ and $c$, there exists a value $u_1$ sufficiently small such that $c(v)>b(v)$ for all $v\in [u_0,u_0+u_1)$. Thus, from Lemma \ref{lemma:auxiliaryfunctionLambda}, the fact that $u \mapsto G(u,x)$ is strictly increasing when $x>0$ and the inequality $H\geq V$ (see Lemma \ref{lemma:HgeqV}), we have that for all $u>0$ and $x> b(u)$,
\begin{align*}
G(u,x)+\int_{(-\infty,0)} \widetilde{H}(u,x+y)\Pi(\dd y)\geq G(u,x)+\int_{(-\infty,0)} \widetilde{V}(u,x+y)\Pi(\dd y)>0,
\end{align*}
where the strict inequality follows from Lemma \ref{lemma:auxiliaryfunctionLambda}. Note that taking $x$ sufficiently large we have that, under the measure $\P_{u_0,x}$, $X$ spends a positive amount of time between the curves $b(u)$ and $c(u)$ for $u\in [u_0,u_0+u_1]$ with positive probability. Thus, since $\sigma_c <\tau_0^-$ the above facts imply that

\begin{align*}
0= \E_{u_0,x}\left(\int_0^{\sigma_{b}^-} \left[ G(U_s,X_s) +\int_{(-\infty,0)} \widetilde{H}(U_s,X_s+y)\Pi(\dd y)\right]\I_{\{X_s<c(U_s) \}}\dd s \right)>0.
\end{align*}
We have got a contradiction and then we have that $c(u)\leq b(u)$ for all $u>0$.
\end{proof}
Note that (\ref{eq:inequalityVPidyGgeq0}) and the definition of $u_H$ given in (\ref{eq:definitionofuH}) imply the inequality $G(u,x)+\int_{(-\infty,0)} \widetilde{H}(u,x+y)\Pi(\dd y)\geq 0$ for all $u>0$ and $x>c(u)$. It can be shown that such inequality guarantees that the process $\{ H(U_t,X_t)+\int_0^t G(U_s,X_s) \dd s, t\geq 0 \}$ is a $\P_{u,x}$-submartingale, for all $(u,x)\in E$. We finish the proof by showing that indeed $c$ corresponds to $b$.

\begin{lemma}
\label{lemma:uisb}
We have that $c(u)=b(u)$ for all $u>0$ and $V(u,x)=H(u,x)$ for all $(u,x)\in E$.
\end{lemma}

\begin{proof}
Suppose that there exists $u>0$ such that $c(u)<b(u)$ and take $x\in (c(u),b(u))$. Then, by the Markov property and representation (\ref{eq:alternativerepresentationforH}) we have that

\begin{align}
H(u,x)&= \E_{u,x}\left(H(U_{\tau_D}, X_{\tau_D} )  \right)+ \E_{u,x}\left(\int_0^{\tau_D } G(U_s,X_s)\I_{\{X_s<c(U_s) \}}\dd s \right)\nonumber\\
\label{eq:HintermsoftauD}
&\qquad- \E_{u,x}\left(\int_0^{\tau_D} \int_{(-\infty,0)} \widetilde{H}(U_s,X_s+y)\Pi(\dd y)\I_{\{X_s>c(U_s) \}}\dd s \right),
\end{align}
where $\tau_D=\inf\{t>0: X_t\geq b(U_t) \}$. On the other hand, we have that

\begin{align*}
V(u,x)= \E_{u,x}\left(\int_0^{\tau_D } G(U_s,X_s)\dd s \right).
\end{align*}
Since $X_{\tau_D}=b(U_{\tau_D})\geq  c(U_{\tau_D})$ under $\P_{u,x}$, we deduce from Lemma \ref{lemma:HvanishesinDc} that $H(U_{\tau_D},X_{\tau_D})=0$. Hence, using the inequality $H\geq V$ (see Lemma \ref{lemma:HgeqV}) in \eqref{eq:HintermsoftauD}, we obtain that
\begin{align*}
0\geq  \E_{u,x}\left(\int_0^{\tau_D } \left[ G(U_s,X_s)+\int_{(-\infty,0)} \widetilde{H}(U_s,X_s+y)\Pi(\dd y) \right] \I_{\{X_s>c(U_s) \}}\dd s \right)>0,
\end{align*}
where the strict inequality follows by the inequality (\ref{eq:inequalityVPidyGgeq0}) and the continuity of $b$ and $c$. This contradiction allows us to conclude that $c(u)=b(u)$ for all $u>0$. Lastly, recall from the proof of Lemma \ref{lemma:HgeqV} that $H(u,x)=0=V(u,x)$ for all $x>c(u)=b(u)$, and 

\begin{align*}
H(u,x)=\E_{u,x}\left( \int_0^{\tau_c} G(U_s,X_s)\dd s \right)=\E_{u,x}\left( \int_0^{\tau_D} G(U_s,X_s)\dd s \right) = V(u,x)
\end{align*}
for all $(u,x)\in E$ such that $x<c(u)=b(u)$. The proof is now complete. 
\end{proof}

\begin{rem}
A close inspection of the proof tells us that the assumption $H\leq 0$ can be dropped when $\Pi\equiv 0$.
\end{rem}

\section{Examples}
\label{sec:Examples}
\subsection{Brownian Motion with drift example}
\label{sbsec:BMwdriftexampleLp}
Suppose that $X_t$ is given by
\begin{align*}
X_t=\mu t+\sigma B_t,
\end{align*}
where $\mu>0$, $\sigma>0$ and $B=\{ B_t ,t\geq 0\}$ is a standard Brownian motion. Here, we consider the case $p=2$. Then, for any $(u,x)\in E$ we have,
\begin{align*}
G(u,x)=u \psi'(0+) W(x)- \E_x( g).
\end{align*}
It is well known that for $\beta\geq 0$ and $q\geq 0$,
\begin{align*}
\psi(\beta)=\frac{\sigma^2}{2}\beta^2+\mu \beta \qquad \text{and} \qquad  \Phi(q)=\frac{1}{\sigma^2}\left[ \sqrt{\mu^2 +2\sigma^2 q} -\mu\right].
\end{align*}
Thus, $\psi'(0+)=\mu$, $\Phi'(0)=\frac{1}{\mu}$, $\Phi''(0)=-\frac{\sigma^2}{\mu^3}$ and $\Phi'''(0)=3\sigma^4/ \mu^5$. The scale function is (see e.g. \cite{kyprianou2011theory}, on p. 102) given by
%
%
\begin{align*}
W(x)=\frac{1}{\mu}(1-\exp(-2\mu x/\sigma^2)), \qquad x\geq 0.
\end{align*}
An easy calculation shows that $W^{*(2)}$ is given by
\begin{align*}
W^{*(2)}(x)
=\frac{1}{\mu^2} x[1+ \exp(-2\mu/\sigma^2 x)]-\frac{\sigma^2}{\mu^2}\frac{1}{\mu}(1-\exp(-2\mu/\sigma^2 x)), \qquad x\geq 0.
\end{align*}
For all $x\in \R$, the value $\E_x(g)$ can be calculated from (\ref{eq:laplacetransformofg}) via differentiation to have
\begin{align*}
\E_x(g)&=  -\psi'(0+)[\Phi''(0+)+x\Phi'(0)^2]+\psi'(0+)W^{*2}(x)\\
&=\left\{
\begin{array}{lc}
\frac{\sigma^2}{\mu^2}-\frac{x}{\mu},& x<0,\\
\frac{\sigma^2}{\mu^2} \exp(-2\mu/\sigma^2 x) +\frac{x}{\mu} \exp(-2\mu/\sigma^2 x), & x\geq 0.
\end{array}
\right.
\end{align*}
and $\E(g^2)=\Phi'''(0)\psi'(0+)=3(\sigma/\mu)^4$. Moreover, it is well known that $X_r \sim N(\mu r, \sigma^2 r)$ for any $r\geq 0$ and
\begin{align*}
\P_x( X_r \in \dd z, \underline{X}_r \leq   0 )
&=e^{-\frac{2\mu}{\sigma^2 } x} \frac{1}{\sqrt{\sigma^2 r}} \phi\left(\frac{z+x-\mu r}{\sqrt{\sigma^2 r}} \right) \dd z,
\end{align*}
for any $r\geq 0$, $x\geq 0$ and $z\geq 0$, where $\phi$ is the density of a standard normal distribution. Hence we have that for any $x\geq 0$ and $z\geq 0$,
\begin{align*}
\P_x( X_r \in \dd z, \underline{X}_r \geq  0 )=\frac{1}{\sqrt{\sigma^2 r}}\left[  \phi\left(\frac{z-x-\mu r}{\sqrt{\sigma^2 r}} \right)-e^{-\frac{2\mu}{\sigma^2 } x}  \phi\left(\frac{z+x-\mu r}{\sqrt{\sigma^2 r}} \right)\right] \dd z.
\end{align*}
Then, we calculate for any $u>0$ and $x>0$,
\begin{align*}
\E_{x}&\left(\int_0^{\tau_0^-} G(u+r,X_r)\I_{\{X_{r} <b(r+u) \}} \dd r \right)\\
&=\int_0^{\infty}  \int_0^{b(r+u)} [(r+u) \psi'(0+)W(z)- \E_{z}( g)]\P_x(X_r\in \dd z,\underline{X}_r \geq 0)  \dd r\\
&=\int_0^{\infty} \left\{ H(r,u,x,b(r+u)) -e^{-2\mu/\sigma^2 x}H(r,u,-x,b(r+u)) \right\} \dd r,
\end{align*}
where a lengthy but straightforward calculation gives
\begin{align*}
H(r,t,x,b)&=\int_0^{b} [(r+t) \psi'(0+)W(z)- \E_{z}( g)] \frac{1}{\sqrt{\sigma^2 r}} \phi\left(\frac{z-x-\mu r}{\sqrt{\sigma^2 r}} \right)dz\\
&=(r+t)\left[\Psi\left( \frac{b-x-\mu r}{\sigma \sqrt{r}}\right)-\Psi\left( \frac{-x-\mu r}{\sigma \sqrt{r}}\right) \right]\\
&\qquad  -\left[\frac{x}{\mu}+t+\frac{\sigma^2 }{\mu^2}  \right]  e^{-2\mu/\sigma^2 x}\left[\Psi\left( \frac{b-x+\mu r}{\sigma \sqrt{r}}\right)-\Psi\left( \frac{-x+\mu r}{\sigma \sqrt{r}}\right) \right]\\
&\qquad +\frac{\sigma \sqrt{r}}{\mu} e^{-2\mu/\sigma^2 x} \left[ \phi\left(\frac{b-x+\mu r}{\sigma \sqrt{r}} \right)-\phi\left(\frac{-x+\mu r}{\sigma \sqrt{r}} \right) \right].
\end{align*}
From formula (\ref{eq:expressionforV0xnegative}) we can deduce that
\begin{align*}
V(0,x)&=-\int_{0}^{-x} \int_{0}^{\infty}  \E_{-u-z}(g)W'(u) \dd u  \dd z +V(0,0)\\
&= \frac{3 \sigma^2}{2  \mu^3} x-\frac{1}{ 2\mu^2} x^2+V(0,0).
\end{align*}
Then,
\begin{align*}
\frac{\partial }{\partial x} V_-(0,0)= \frac{3 \sigma^2}{2  \mu^3}.
\end{align*}
From Theorem \ref{thm:characterisationofbandV} we have that for $u>0$ and $x>0$,
\begin{align*}
V(u,x)
&=
V(0,0)\exp(-2\mu x/\sigma^2) \\
&\qquad+\int_0^{\infty} \left\{ H(r,u,x,b(r+u)) -e^{-2\mu/\sigma^2 x}H(r,u,-x,b(r+u)) \right\} \dd r .
\end{align*}
Therefore, the curve $b(u)$ and $V(0,0)$ satisfy the equations
\begin{align*}
0&=V(0,0)\exp(-2\mu b(u)/\sigma^2) \\
&\qquad +\int_0^{\infty} \left\{ H(r,u,x,b(r+u)) -e^{-2\mu/\sigma^2 x}H(r,u,-x,b(r+u)) \right\} \dd r,\\
0&=\frac{3 \sigma^2}{2  \mu^3}-\frac{\partial}{\partial x}V_+(0,0),
\end{align*}
for all $u>0$, where
\begin{align*}
-\frac{3}{2}\frac{\sigma^4}{\mu^4}\leq V(0,0)<0.
\end{align*}
Note that $\frac{\partial}{\partial x}V_+(0,0)$ can be estimated via $[V(h_0,h_0)-V(0,0)]/h_0$ for $h_0$ sufficiently small. \\

We can approximate the integrals above by Riemann sums to implement a numerical approximation of $b$ and $V$. Take $u>0$ and $x>0$, by the dominated convergence theorem (see equations \eqref{eq:integrabilityofGbelowb} and \eqref{eq:integrabilityVPidy}) we have that 
\begin{align*}
V(u,x)
&=
V(0,0)\exp(-2\mu x/\sigma^2)\\
&\qquad+\lim_{T\rightarrow \infty}\int_0^{T-u} \left\{ H(r,u,x,b(r+u)) -e^{-2\mu/\sigma^2 x}H(r,u,-x,b(r+u)) \right\} \dd r \\
&=
V(0,0)\exp(-2\mu x/\sigma^2)\\
&\qquad+\lim_{T\rightarrow \infty}\int_u^{T} \left\{ H(r-u,u,x,b(r)) -e^{-2\mu/\sigma^2 x}H(r-u,u,-x,b(r)) \right\} \dd r \\
&=
V(0,0)\exp(-2\mu x/\sigma^2)\\
&\qquad+\lim_{T\rightarrow \infty} \lim_{n \rightarrow \infty} \sum_{i=k}^{n-1} \Big\{ H((i-k+1)h_n,u,x,b(ih_n))\\
&\qquad\qquad\qquad\qquad\qquad -e^{-2\mu/\sigma^2 x}H((i-k+1)h_n,u,-x,b(ih_n)) \Big\} h_n,
\end{align*}
where for $T>0$ and $n\in \mathbb{Z}_+$ fixed, we define $h_n:=T/n$, the value $k$ is an integer such that $u\in [kh_n,(k+1)h_n)$, and we used that for each $T-u>0$, the functions $r\mapsto  H(r,u,	x,b(r+u))$ and $r\mapsto  H(r,u,-x,b(r+u))$ are Riemann integrable on $[0,T-u]$.\\

 Hence, take $n \in \mathbb{Z}_+$ and $T>0$ sufficiently large such that $h_n=T/n$ is small.  For each $k\in \{0,1,2,\ldots,n\}$, we define $u_k=kh_n$. Then, from the equation above and continuity of $V$, we can approximate $V(u,x)$, for any $x>0$ and $u\in [u_k,u_{k+1})$, by 
\begin{align}
V_h(u_k,x)&=V(0,0)\exp(-2\mu x/\sigma^2) \nonumber\\
\label{eq:approximationofVBMcase}
&\qquad+\sum_{i=k}^{n-1} [H(u_{i-k+1},u_k,x,b(u_i))-e^{-2\mu/\sigma^2 x}H(u_{i-k+1},u_k,-x,b(u_i))]h_n.
\end{align}
In particular, taking $x=b(u_k)$ we have that 
\begin{align*}
0&\approx V(0,0)\exp(-2\mu b(u_k)/\sigma^2)\\
&\qquad+\sum_{i=k}^{n-1} [H(u_{i-k+1},u_k,b(u_k),b(u_i))-e^{-2\mu/\sigma^2 b(u_k)}H(u_{i-k+1},u_k,-b(u_k),b(u_i))]h_n.
\end{align*}
Moreover, from the definition of the right-derivative, we have 
\begin{align}
\label{eq:approximationforV0BMcase}
0\approx \frac{3 \sigma^2}{2  \mu^3} - \frac{V(h_n,h_n)-V(0,0)}{h_n} \approx \frac{3 \sigma^2}{2  \mu^3} - \frac{V_h(h_n,h_n)-V(0,0)}{h_n},
\end{align}
where $V_h$ is given in \eqref{eq:approximationofVBMcase}. Note that $V_h$ also depends on $V(0,0)$ and on the values $\{b(t_i), i=k,\ldots,n-1 \}$. Then, for each $k\in \{1,2,\ldots,n-1 \}$, equation \eqref{eq:approximationforV0BMcase} can be interpreted as a non-linear equation that depends on $V(0,0)$ and $\{b(u_i), i=k,\ldots,n-1 \}$. We then propose the following algorithm:

\begin{enumerate}
\item Take $n \in \mathbb{Z}_+$ and $T>0$ sufficiently large such that $h_n=T/n$ is small. For each $k\in \{0,1,2,\ldots,n\}$, define $u_k=kh_n$.
\item Take a value $V_0 \in [-\frac{3}{2}\frac{\sigma^4}{\mu^4},0)$.
\item Let $b_{n-1}>0$ be the solution to the equation 
\begin{align*}
0&= V_0\exp(-2\mu b_{n-1}/\sigma^2)\\
&\qquad+ [H(u_{1},u_{n-1},b_{n-1},b_{n-1})-e^{-2\mu/\sigma^2 b_{n-1}}H(u_{1},u_{n-1},-b_{n-1},b_{n-1})]h_n.
\end{align*}
\item For $1\leq k\leq n-2$, let $b_k>0$ be the solution to the equation 
\begin{align*}
0&= V_0\exp(-2\mu b_k/\sigma^2)\\
&\qquad+\sum_{i=k}^{n-1} [H(u_{i-k+1},u_k,b_k,b_i)-e^{-2\mu/\sigma^2 b_k}H(u_{i-k+1},u_k,-b_k,b_i)]h_n.
\end{align*}
Note that in this step we calculate backwards the values $\{ b_k, k=1,2,\ldots, n-2\}$. 
\item Calculate the quantity 
\begin{align*}
R^{V_0}&= \frac{3 \sigma^2}{2  \mu^3}+\frac{V_0}{h_n} - \frac{1}{h_n} \exp(-2\mu h_n/\sigma^2)\\
&\qquad- \frac{1}{h_n} \sum_{i=1}^{n-1} [H(u_{i},u_1,h_n,b_i)-e^{-2\mu/\sigma^2 h_n}H(u_{i},u_1,-h_n,b_i)]h_n. 
\end{align*}
If the value of $R^{V_0}\approx 0$ then stop the algorithm, otherwise, go back to step 2 with a different choice of $V_0$. 

\end{enumerate}

The sequence $\{ b_k, k=1,\ldots,n-1 \}$ is a numerical approximation of the sequence $\{b(t_k), k=1,\ldots,n-1 \}$, whereas $V_0$ is an approximation of $V(0,0)$. Note that our algorithm is a simple method to approximate the curve b and to illustrate how Theorem \ref{thm:characterisationofbandV} can be used for that purpose. Note that better methods are needed to achieve higher precision and shorter computation time. We show in Figure \ref{pic:BMexample} a numerical calculation of the optimal boundary and the value function using the method above. The case considered is when $\mu=1/2$ and $\sigma=1$.

\begin{figure}
\centering
\includegraphics[scale=0.28]{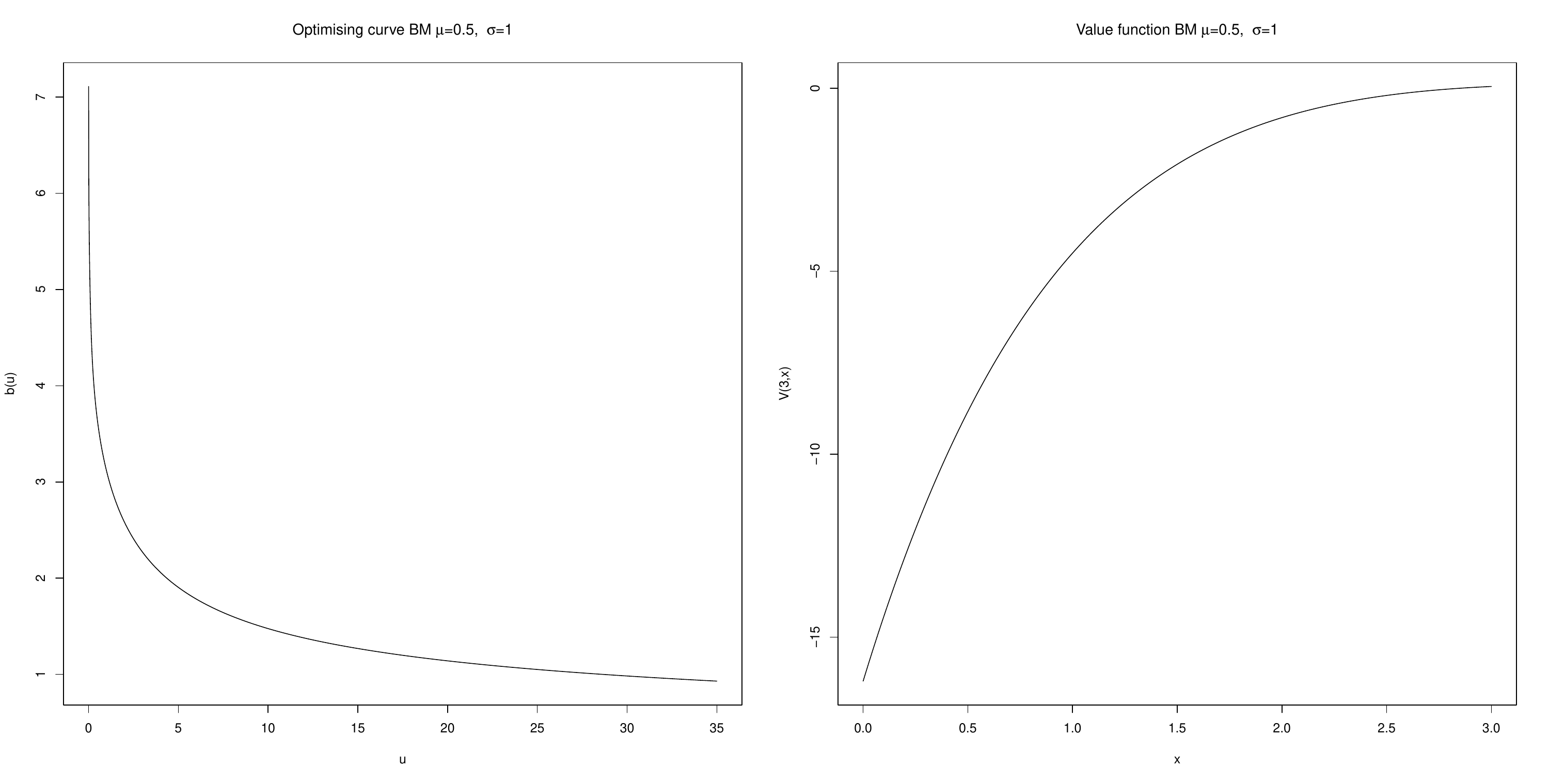}
\caption{Numeric calculation of the optimal boundary and value function $V$ for the Brownian motion with drift case.}
\label{pic:BMexample}
\end{figure}
\newpage

\subsection{Brownian motion with exponential jumps example}
Consider the case $p=2$ and when $X$ is a Brownian motion with drift and exponential jumps, i.e., $X=\{X_t,t\geq 0 \}$ with
\begin{align*}
X_t=\mu t+\sigma B_t-\sum_{i=1}^{N_t} Y_i, \qquad t\geq 0,
\end{align*}
where $\sigma> 0$, $\mu>0$, $B=\{ B_t,t\geq 0\}$ is a standard Brownian motion, $N=\{N_t,t\geq 0 \}$ is an independent Poisson process with rate $\lambda>0$ and $\{Y_i, i\geq 1\}$ is a sequence of independent exponential distributed random variables with parameter $\rho>0$ independent of $B$ and $N$. We further assume that $\mu \rho >\lambda$ so $X$ drifts to infinity. The Laplace exponent is given by for $\beta\geq 0$ by
%
\begin{align*}
\psi(\beta)=\mu \beta + \frac{\sigma^2 }{2}\beta^2-\frac{\lambda \beta}{\rho +\beta}.
\end{align*}
 In this case the L\'evy measure is given by $\Pi(\dd x)=\lambda \rho e^{\rho x} \dd x$, for all $x<0$. An easy calculation leads to $\psi'(0+)=\mu-\lambda/\rho$. Using the identity $\psi(\Phi(q))=q$, for any $q\geq 0$, we deduce that	
\begin{align*}
\Phi'(0+)&=\frac{\rho}{\mu \rho -\lambda }, \\
 \Phi''(0+)&=-\frac{\sigma^2\rho^3+2\lambda \rho}{\left[\mu\rho-\lambda\right]^3},\\
 \Phi'''(0+)&=3\frac{\rho[\sigma^2 \rho^2+2\lambda]^2}{[\mu \rho -\lambda]^5}+\frac{6\lambda \rho}{[\mu \rho -\lambda]^4}.
\end{align*}
%
%

It is know that (see e.g. \cite{kyprianou2011theory}, on p. 101) the scale function $W$ is given by
\begin{align*}
W(x)=\frac{1}{\psi'(0+)}+ \frac{e^{\zeta_1 x}}{\psi'(\zeta_1) }+\frac{e^{\zeta_2 x}}{\psi'(\zeta_2) }
\end{align*}
for $x\geq 0$, where
\begin{align*}
\zeta_1&=\frac{-\left(\frac{\sigma^2 }{2} \rho+\mu\right)+\sqrt{\left(\frac{\sigma^2 }{2} \rho-\mu\right)^2+2\sigma^2 \lambda }}{\sigma^2},\\
\zeta_2&=\frac{-\left(\frac{\sigma^2 }{2} \rho+\mu\right)-\sqrt{\left(\frac{\sigma^2 }{2} \rho-\mu\right)^2+2\sigma^2 \lambda }}{\sigma^2}.
\end{align*}
By differentiating \eqref{eq:laplacetransformofg} with respect to $q$ we obtain that
\begin{align*}
\E_x(g)&=  -\psi'(0+)[\Phi''(0)+x\Phi'(0)^2]+\psi'(0+)W^{*2}(x)\\
&=\left\{
\begin{array}{lr}
 \frac{\sigma^2\rho^2+2\lambda }{\left[\mu\rho-\lambda\right]^2}- \frac{\rho}{\mu \rho -\lambda }x,   &\qquad x<0,\\
\frac{\sigma^2\rho^2+2\lambda }{\left[\mu\rho-\lambda\right]^2}- \frac{\rho}{\mu \rho -\lambda }x  +(\mu-\lambda/\rho) W^{*2}(x),   &  \qquad  x\geq 0.
\end{array}
\right. 
\end{align*}
Moreover, taking $x=0$ and differentiating twice \eqref{eq:laplacetransformofg} with respect to the variable $q\geq 0$, we see that 
\begin{align*}
\E(g^2)=\psi'(0+)\Phi'''(0)=3\frac{[\sigma^2 \rho^2+2\lambda]^2}{[\mu \rho -\lambda]^4}+\frac{6\lambda }{[\mu \rho -\lambda]^3}.
\end{align*}
For $x<0$, the value function is then given by
\begin{align*}
V(0,x)&=-\int_{0}^{-x} \int_{0}^{\infty}  \E_{-u-z}(g)W'(u) \dd u  \dd z +V(0,0)\\
&=\int_{0}^{-x} \int_{0}^{\infty} \left[\Phi''(0+)+\Phi'(0)^2(-u-z) \right] \psi'(0+) W'(u) \dd u  \dd z +V(0,0)\\
&=\left[ \Phi''(0)(-x)+\Phi'(0)^2 \E(\underline{X}_{\infty})(-x)-\Phi'(0)^2 x^2/2\right]+V(0,0),
\end{align*}
where in the last equality we used that $\psi(0+)W(x)=\P_x(\tau_0^-=\infty)=\P(-\underline{X}_{\infty}\leq x )$ and hence $\psi'(0+) W'(u)$ is the density function of the random variable $\underline{X}_{\infty}$. From (\ref{eq:laplacetransformofrunninginfimumexptime}) we know that for any $\beta\geq 0$,
\begin{align*}
\E(e^{\beta \underline{X}_{\infty}}) =\psi'(0+) \frac{\beta}{\psi(\beta)}.
\end{align*}
Hence, by differentiating and using the fact that $\Phi'(q)=1/\psi'(\Phi(q))$, we can see that 
\begin{align*}
\E(\underline{X}_{\infty})=\frac{\Phi''(0)}{2\Phi'(0)^2}.
\end{align*}
Hence,
\begin{align*}
V(0,x)
&=-\left[ \frac{3}{2} \Phi''(0)x+\Phi'(0)^2 x^2/2\right]+V(0,0)
\end{align*}
for any $x<0$. %
Next, we calculate for any $x>0$,
\begin{align*}
\int_{(-\infty,-x)} &V(0,x+y)\Pi(\dd y)\\
 &=\int_{-\infty}^{-x}\left[-\frac{3}{2} \Phi''(0)(x+y)-\Phi'(0)^2 (x+y)^2/2+V(0,0) \right]\lambda \rho e^{\rho y}\dd y\\
&=\lambda e^{-\rho x} \int_{-\infty}^{0}\left[-\frac{3}{2}  \Phi''(0)y- \Phi'(0)^2 y^2/2+V(0,0) \right] \rho e^{\rho y}\dd y\\
&=\lambda e^{-\rho x} \left[\frac{3  \Phi''(0)}{2 \rho} -\frac{  \Phi'(0)^2}{\rho^2} +V(0,0) \right].
\end{align*}
Similarly, we have that for all $u>0$ and $x>b(u)$, 
\begin{align*}
\int_{(-\infty,0)} V(u,x+y)\I_{\{0<x+y<b(u) \}}\Pi(\dd y)&=e^{-\rho x}\int_{0}^{b(u)} V(u,y)\lambda \rho e^{\rho y} \dd  y \\
&=e^{-\rho (x-b(u))}\int_{(-b(u),0)} V(u,y+b(u))\Pi(\dd y).
\end{align*}
Then, since $G$ is non decreasing in each argument, condition \eqref{eq:intV+Gispositive} is satisfied if and only if 
\begin{align*}
\int_{(-b(u),0)} &V(u,y+b(u))\Pi(\dd y)\\
&\qquad+\lambda e^{-\rho b(u)} \left[\frac{3  \Phi''(0)}{2 \rho} -\frac{  \Phi'(0)^2}{\rho^2} +V(0,0) \right]+G(u,b(u))\geq 0
\end{align*}
for all $u>0$. On the other hand, for any $u,x>0$, equation (\ref{eq:representationforVwithindicatorsfunc}) reads as
\begin{align*}
&V(u,x)\\
&=V(0,0)\frac{\sigma^2}{2}W'(x)\\
&\qquad -\E_{x} \bigg( \int_0^{ \tau_{0}^-}  e^{-\rho (X_s-b(u+s))}\I_{\{X_s >b(u+s) \}} \\
&\qquad \qquad \qquad \qquad \times \int_{(-b(u+s),0)}V(u+s,y+b(u+s)) \Pi(\dd y)  \dd s \bigg)\\
&\qquad +\E_{x}\left( \int_0^{ \tau_{0}^-} \left[ G(u+s,X_s) +\int_{(-\infty,-X_s)} V(0,X_s+y)  \Pi(\dd y)  \right] \I_{\{ X_s<b(u+s)\}} \dd s \right)\\
&=V(0,0)\frac{\sigma^2}{2}W'(x) \\
&\qquad-\int_0^{ \infty} \int_{(-b(u+s),0)}V(u+s,y+b(u+s)) \Pi(\dd y) \\
&\qquad \qquad \qquad \qquad  \times \E_x\left( e^{-\rho (X_s-b(u+s))}\I_{\{X_s >b(u+s) , \underline{X}_s \geq 0 \}} \right)\dd s\\
&\qquad +\int_0^{\infty}  \E_{x}\left( \left[ G(u+s,X_s) +\int_{(-\infty,-X_s)} V(0,X_s+y)  \Pi(\dd y)  \right] \I_{\{ X_s<b(u+s) , \underline{X}_s \geq 0\}} \dd s \right)\\ 
&=V(0,0)\frac{\sigma^2}{2}W'(x)-\int_0^{ \infty}\mathcal{V}(u+s,b(u+s))  F_2(s,x,b(u+s))\dd s \\
&\qquad+\int_0^{\infty}  F_1(s,u,x,b(u+s),V(0,0)) \dd s,
\end{align*}
where for any $s,u,x,b>0$ we define
\begin{align*}
\mathcal{V}(u,b)&:=\int_{(-b,0)} V(u,y+b)\Pi( \dd y),\\
F_1(s,u,x,b,V_0)&:=\E \left( G(u+s,X_s+x)\I_{\{X_s+x <b, \underline{X}_s+x\geq 0 \}}\right) \\
&\qquad +\E\left(\lambda e^{-\rho (X_s+x))} \left[\frac{3  \Phi''(0)}{2 \rho} -\frac{  \Phi'(0)^2}{\rho^2} +V_0 \right]   \I_{\{X_s+x <b, \underline{X}_s+x\geq 0 \}}\right),\\
F_2(s,x,b)&:=\E\left(e^{-\rho (X_s+x-b)} \I_{\{ X_s+x>b, \underline{X}_s+x\geq 0\}} \right).
\end{align*}
In summary, we have that $V$, $b$ and $V(0,0)$ satisfy the equations
\begin{align*}
V(u,x)
&=V(0,0)\frac{\sigma^2}{2}W'(x)+\int_0^{\infty}  F_1(s,u,x,b(u+s),V(0,0)) \dd s\\
&\qquad-\int_0^{ \infty}\mathcal{V}(u+s,b(u+s))  F_2(s,x,b(u+s))\dd s,\\
0
&=V(0,0)\frac{\sigma^2}{2}W'(b(u))+\int_0^{\infty}  F_1(s,u,b(u),b(u+s),V(0,0)) \dd s\\
&\qquad -\int_0^{ \infty}\mathcal{V}(u+s,b(u+s))  F_2(s,b(u),b(u+s))\dd s,\\
0&=\frac{3}{2} \Phi''(0) +\frac{\partial}{\partial x}V_+(0,0),
\end{align*}
for all $u,x>0$.\\

In a similar way as in the previous section, we can approximate the integrals above by Riemann sums to implement a numerical approximation of $b$ and $V$. Take $u>0$ and $x>0$, by the dominated convergence theorem (see equations \eqref{eq:integrabilityofGbelowb} and \eqref{eq:integrabilityVPidy}) we have that  
\begin{align*}
V(u,x)
&=
V(0,0)\frac{\sigma^2}{2}W'(x)+\lim_{T\rightarrow \infty} \lim_{n \rightarrow \infty} \sum_{i=k}^{n-1} F_1((i-k+1)h_n,u,x,b(ih_n),V(0,0))h_n\\
&\qquad -\lim_{T\rightarrow \infty} \lim_{n \rightarrow \infty} \sum_{i=k}^{n-1} \mathcal{V}((i+1)h_n,b((i+1)h_n) )\\
&\qquad \qquad\qquad\qquad \qquad \times F_2((i-k+1)h_n,x,b(ih_n)) h_n,
\end{align*}
where $h_n:=T/n$, for $T>0$ and $n\in \mathbb{Z}_+$, the value $k$ is an integer such that $u\in [kh_n,(k+1)h_n)$, and we used that for each $T-u>0$, the functions $r\mapsto  F_1(r,u,x,b(r+u),V(0,0))$ and $r\mapsto \mathcal{V}(r+u,b(r+u))F_2(r,x,b(r+u))$ are Riemann integrable on $[0,T-u]$.\\

 Let $n \in \mathbb{Z}_+$ and $T>0$ sufficiently large such that $h_n=T/n$ is small, for each $k\in \{0,1,2,\ldots,n\}$, we define $u_k=k h_n$. Then, from the equation above and continuity of $V$, we can approximate $V(u,x)$, for any $x>0$ and $u\in [u_k,u_{k+1})$, by 
\begin{align}
V_h(u_k,x)&=V(0,0)\frac{\sigma^2}{2}W'(x)+\sum_{i=k}^{n-1} F_1(u_{i-k+1},u_k,x,b(u_i),V(0,0))h_n \nonumber\\
\label{eq:approximationofVBMwithexpjumpscase}
&\qquad -\sum_{i=k}^{n-1}\mathcal{V}(u_{i+1},b(u_{i+1})) F_2(u_{i-k+1},x,b(u_i)) h_n.
\end{align}
In particular, taking $x=b(u_k)$ we have that 
\begin{align*}
0&\approx V(0,0)\frac{\sigma^2}{2}W'(b(u_k))+\sum_{i=k}^{n-1} F_1(u_{i-k+1},u_k,b(u_k),b(u_i),V(0,0))h_n \\
&\qquad -\sum_{i=k}^{n-1}\mathcal{V}(u_{i+1},b(u_{i+1})) F_2(u_{i-k+1},b(u_k),b(u_i)) h_n.
\end{align*}
Moreover, from the definition of the right derivative we have that
\begin{align}
\label{eq:approximationforV0BMwithexpjumpscase}
0\approx \frac{3}{2} \Phi''(0) - \frac{V(h_n,h_n)-V(0,0)}{h_n} \approx \frac{3}{2} \Phi''(0) - \frac{V_h(h_n,h_n)-V(0,0)}{h_n},
\end{align}
where $V_h$ is given in \eqref{eq:approximationofVBMwithexpjumpscase}. Note that, for each $k\in \{1,2,\ldots,n-1\}$, $V_h(u_k,x)$ depends on $V(0,0)$ and on the values $\{b(t_i), i=k,\ldots,n \}$. Then, equation \eqref{eq:approximationforV0BMwithexpjumpscase} can be interpreted as a non-linear equation that depends on $V(0,0)$ and $\{b(u_i), i=k,\ldots,n \}$. The functions $F_1$ and $F_2$ can be estimated by simulating the process $\{ (X_t,\underline{X}_t), t \geq  0)\}$ (see e.g. \cite{kuznetsov2011wiener}, Theorem 4 and Remark 3). For $V_0<0$, $x\geq 0$, $\mathbf{a}_k=(a_i, i=k,\ldots,n)$ with $a_i\geq 0$ and $k\in \{1,2,\ldots,n\}$, we define the following auxiliary functions,
\begin{align*}
V_h^{0}(u_k,x,\mathbf{a}_k,V_0)&=V_0\frac{\sigma^2}{2}W'(x)+\sum_{i=k}^{n-1} F_1(u_{i-k+1},u_k,x,a_i,V_0)h_n \nonumber\\
&\qquad -\sum_{i=k}^{n-1}\mathcal{V}_h^0(u_{i+1},\mathbf{a}_{i+1},V_0) F_2(u_{i-k+1},x,a_i) h_n,\\
\mathcal{V}^{0}_h(u_n,\mathbf{a}_n,V_0)&=0,\\
\mathcal{V}^{0}_h(u_{k},\mathbf{a}_k,V_0)&=\sum_{j=1}^{\lfloor a_k/h_n \rfloor} V_h^{0}(u_{k},jh_n,\mathbf{a}_k,V_0 )\lambda \rho e^{\rho jh_n}h_n,
\end{align*}
where $\lfloor \cdot \rfloor$ is the floor function. We then propose the following algorithm:

\begin{enumerate}
\item Take $n \in \mathbb{Z}_+$ and $T>0$ sufficiently large such that $h_n=T/n$ is small. For each $k\in \{0,1,2,\ldots,n\}$, define $u_k=kh_n$.
\item Take a value $V_0 \in [-\frac{1}{2}\E(g^2),0)$.
\item Define $b_n=0$ and $\mathcal{V}(u_n,b_n)=0$. Let $b_{n-1}>0$ be the solution to the equation 
\begin{align*}
0&=V_0\frac{\sigma^2}{2}W'(b_{n-1})+F_1(u_{1},u_{n-1},b_{n-1},b_{n-1})h_n,
\end{align*}
subject to $V_h^{0}(u_{n-1},jh_n,\mathbf{b}_{n-1},V_0)\leq 0$ for $j=1,2,\ldots,\lfloor b_{n-1}/h_n \rfloor$ and that 
\begin{align*}
\mathcal{V}_h^0(u_{n-1},\mathbf{b}_{n-1},V_0)
+\lambda e^{-\rho b_{n-1}} \left[\frac{3  \Phi''(0)}{2 \rho} -\frac{  \Phi'(0)^2}{\rho^2} +V_0 \right]+G(u_{n-1},b_{n-1})\geq 0,
\end{align*}
where $\mathbf{b}_{n-1}=(b_{n-1},b_n)$.
\item For $1\leq k\leq n-2$, let $b_k>0$ be the solution to the equation 
\begin{align*}
0&= V_0\frac{\sigma^2}{2}W'(b_k)+\sum_{i=k}^{n-1} F_1(u_{i-k+1},u_k,b_k,b_i,V_0)h_n \nonumber\\
&\qquad -\sum_{i=k}^{n-1}\mathcal{V}(u_{i+1},b(u_{i+1})) F_2(u_{i-k+1},b_k,b_i) h_n,
\end{align*}
subject to $V_h^{0}(u_k,jh_n,\mathbf{b}_k,V_0)\leq 0$ for $j=1,2,\ldots,\lfloor b_k/h_n \rfloor$ and
\begin{align*}
\mathcal{V}_h(u_{k}, \mathbf{b}_{k},V_0)+\lambda e^{-\rho b_{k}} \left[\frac{3  \Phi''(0)}{2 \rho} -\frac{  \Phi'(0)^2}{\rho^2} +V_0 \right]+G(u_{k},b_{k})\geq 0, 
\end{align*}
where $\mathbf{b}_k=(b_i,i=k,\ldots,n)$. Note that in this step we calculate the values $b_{n-2}$, $ \mathcal{V}_h(u_{n-2},\textbf{b}_{n-2},V_0)$, $b_{n-3}$, $ \mathcal{V}_h(u_{n-3}$, $\textbf{b}_{n-3},V_0),\ldots, b_{1}$,$ \mathcal{V}_h(u_{1},\textbf{b}_{1},V_0)$.
\item Calculate the quantity 
\begin{align*}
R^{V_0}&= \frac{3}{2} \Phi''(0)+\frac{V_0}{h_n} -\frac{1}{h_n} V_h^0(u_1,h_n,\mathbf{b}_1,V_0) . 
\end{align*}
If the value of $R^{V_0}\approx 0$ stop the algorithm, otherwise, go back to step 2 with a different choice of $V_0$. 

\end{enumerate}

The sequence $\{ b_k, k=1,\ldots,n-1 \}$ is a numerical approximation of the sequence $\{b(t_k), k=1,\ldots,n-1 \}$, whereas $V_0$ is an approximation of $V(0,0)$. Note that our algorithm is a simple method to approximate the curve b and to illustrate how Theorem \ref{thm:characterisationofbandV} can be used for that purpose. Note that better methods are needed to achieve higher precision and shorter computation time. 
We show in Figure \ref{pic:BMwithexpjumpsexample} a numerical calculation of the optimal boundary and the value function using the parametrisation $\mu=3$, $\sigma=1$, $\lambda=1$ and $\rho=1$. The functions $F_1$ and $F_2$ above were estimated using Monte Carlo simulations accordingly to the algorithm given in \cite{kuznetsov2011wiener}.

\begin{figure}
\centering
\includegraphics[scale=0.28]{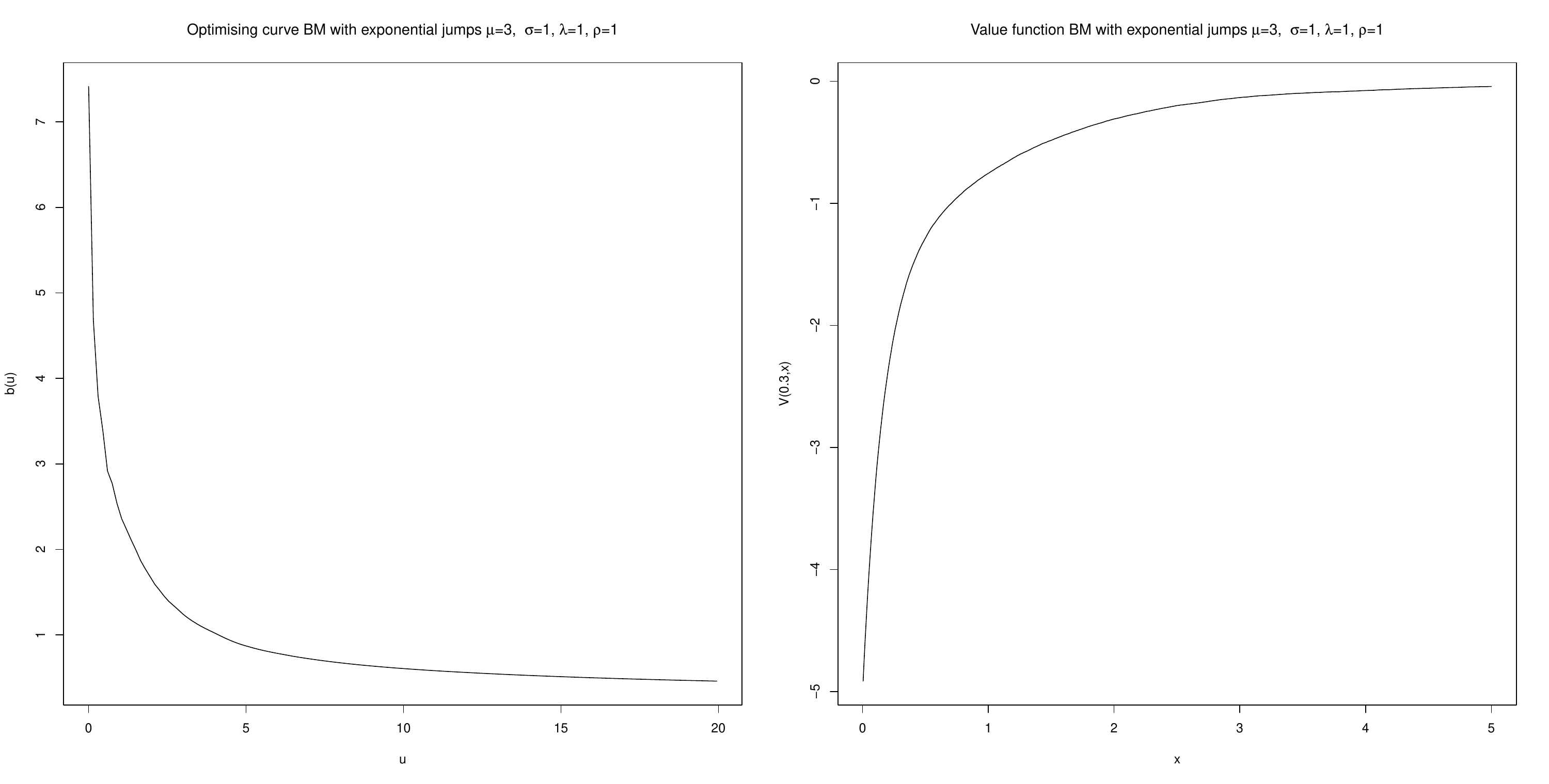}
\caption{Numeric calculation of the optimal boundary and value function $V$ for the Brownian motion with exponential jumps case.}
\label{pic:BMwithexpjumpsexample}
\end{figure}

\begin{appendix}
\section{Technical proofs}
\label{sec:Appendix}
In this Appendix, we include the most technical proofs of the results presented in Sections \ref{sec:Preliminaries}, \ref{sec:optimalpredictionproblem}, \ref{sec:Solutionoptimalstopping} and \ref{sec:proofofmainmainthm}.
\begin{proof}[Proof of Lemma \ref{lemma:boundaryformomentsoftau0+p}]
From equation (\ref{eq:laplacetransformtau0}) we know that
\begin{align*}
F(\theta,x):=\E_x(e^{-\theta \tau_0^+})=e^{\Phi(\theta)x} ,\qquad x\leq 0.
\end{align*}
Then using the formula of Fa\`a di Bruno (see for example \cite{spindler2005short}) we have that for any $n\geq 1$,

\begin{align}
&\frac{\partial^n }{\partial \theta ^n} F(\theta,x)\nonumber\\
\label{eq:derivativesofLaplacetransformoftau+}
& =\sum_{k=1}^n e^{\Phi(\theta)x } x^k \sum_{\substack{k_1+\cdots +k_n=k,\\k_1+\cdots +nk_n=n}} \frac{n!}{k_1! k_2!\cdots k_n!}\left(\frac{\Phi'(\theta)}{1!} \right)^{k_1}\left(\frac{\Phi''(\theta)}{2!} \right)^{k_2} \cdots \left(\frac{\Phi^{(n)}(\theta)}{n!} \right)^{k_n}.
\end{align}
Then evaluating at zero the above equation, using $\Phi(0)=0$ and the fact that $\Phi^{(i)}(0)<\infty$ for $i=1,\ldots,\lfloor p \rfloor +1$, we can find constants $A_r, C_r \geq 0$ such that $\E_x( (\tau_0^+)^r)\leq A_r+C_r|x|^{r}$ for any $r\in \{ 1,\ldots,\lfloor p \rfloor +1\}$. For any non integer $r< \lfloor p \rfloor+1 $ we can use H\"older's inequality to obtain
\begin{align*}
\E_x((\tau_0^+)^r) \leq [\E_x((\tau_0^+)^{\lfloor r \rfloor +1} )]^{\frac{r}{\lfloor r \rfloor +1}} \leq  (A_{\lfloor r \rfloor +1} +C_{\lfloor r \rfloor +1} |x|^{\lfloor r \rfloor +1} )^{\frac{r}{\lfloor r \rfloor +1}}.
\end{align*}
The result follows from the inequality in Lemma \ref{lemma:basicinequality}. Now we show that the second inequality holds. From the strong Markov property, we get that for any $x<0$,

\begin{align*}
\E_x (g^{r}) \leq 2^{r} \E(g^{r})+2^{r} \E_x((\tau_0^+)^{r}) \leq 2^{r} [\E(g^{r})+A_{r} ] +2^{r} C_{r} |x|^{r}.
\end{align*}
The proof is now complete.
\end{proof}

\begin{proof}[Proof of Lemma \ref{lemma:propertiesofExgn}]
It follows from the definition of $g$ that $x\mapsto \E_x(g^p)=\E(g^{(-x)})$ is non-negative and non-increasing. In order to check continuity notice that by integration by parts, we get
\begin{align*}
\E_x(g^p)&=p\int_0^{\infty} s^{p-1} \P_x(g>s)\dd s\\
&=p\int_0^{\infty} s^{p-1} \E_x(1-\psi'(0+)W(X_s))\dd s,
\end{align*}
where the last equality follows from (\ref{distributiong}). Take $x\in \R$ and $\delta
\in \R$. Then using the equation above we have that

\begin{align}
\label{eq:provingcontinuityofEgp}
|\E_x(g^p)-\E_{x+\delta}(g^p)|\leq p\psi'(0+) \E\left( \int_0^{\infty} s^{p-1} |W(X_s+x+\delta)-W(X_s+x)|\dd s \right).
\end{align}
First, suppose that $X$ is of infinite variation and thus $W$ is continuous on $\R$. From the fact that $X$ drifts to $\infty$ we know that $W(\infty)=1/\psi'(0+)$ and therefore it follows that $s^{p-1}(1-\psi'(0+)W(X_s))$ is integrable with respect to the product measure $\P_x\times\lambda([0,\infty))$, where $\lambda$ denotes Lebesgue measure.
We can now invoke the dominated convergence theorem to deduce that $x \mapsto \E_x(g^p)$ is continuous.

Next, in the case that $X$ is of finite variation we have that $W$ has a discontinuity at zero. However, the set $\{s\geq 0: X_s=x \}$ is almost surely countable and thus has Lebesgue measure zero. We can again use the dominated convergence theorem to conclude that  continuity also holds in this case.

We prove now the asymptotic behaviour of $\E_x(g^p)$. Note that when $x$ tends to $-\infty$ the random variable $g^{(-x)} \rightarrow \infty$. Then using Fatou's lemma

\begin{align*}
\liminf_{ x\rightarrow -\infty} \E_x(g^p) \geq  \E( \liminf_{ x\rightarrow -\infty} (g^{(-x)})^p)= \infty.
\end{align*}
Therefore, $\lim_{x\rightarrow -\infty} \E_x(g^p )=\infty$. On the other hand, note that for $x >0$,

\begin{align}
\label{eq:convergenceofgninprobability}
\P_x(g^p=0)=\P_x(g=0)=\P_x(\tau_0^-=\infty)=\psi'(0+)W(x) \xrightarrow{x \rightarrow \infty} 1.
\end{align}
Hence we deduce that the sequence $\{ (g^{(-n)})^p \}_{n\geq 1}$ converges in probability to $0$ (under the measure $\P$) when $n$ tends to infinity. Moreover, since the sequence $\{\E((g^{-n})^p)\}_{n\geq 1}$ is a non-increasing sequence of positive numbers we get that

\begin{align*}
\sup_{n\geq 1}\E((g^{-n})^p) \leq \E(g^p)<\infty,
\end{align*}
where the last inequality holds due to Lemma \ref{lemma:finitenessofEgn} and by assumption. Then $\{ (g^{(-n)})^p \}_{n\geq 1}$ is an uniformly integrable family of random variables. Then, together with the convergence in probability, allows us to conclude that $\E_x(g^p) \rightarrow 0$ when $x\rightarrow \infty$ as claimed.
\end{proof}

\begin{proof}[Proof of Lemma \ref{lemma:finitenessofintExgn}]
First, notice that due to the spatial homogeneity of L\'evy processes and that $x\mapsto \E_x(g^{p-1})$ is non-increasing, it suffices to prove the assertion for $x\leq 0$. Using Fubini's theorem we have that for all $x\leq 0$,

\begin{align*}
\E_x\left( \int_0^{\infty} \E_{X_s} (g^{p-1}) \dd s\right)&=\int_{(-\infty,\infty)} \E_z(g^{p-1}) \int_0^{\infty} \P_x(X_s\in \dd z).
\end{align*}
Since $X$ drifts to infinity we can use the density for the $0$-potential measure of $X$ without killing (see equation (\ref{eq:qpotentialdensitywithoutkilling})) to obtain

\begin{align}
\E_x\left( \int_0^{\infty} \E_{X_s} (g^{p-1}) \dd s\right)
&=\int_{-\infty}^{\infty} \E_z(g^{p-1}) \left[ \frac{1}{\psi'(0+)}-W(x-z) \right]\dd z \nonumber\\
&=\frac{1}{\psi'(0+)} \int_{-\infty}^x \E_z(g^{p-1}) \left[ 1-\psi'(0+)W(x-z) \right]\dd z \nonumber\\
\label{eq:finitenessintegralExgn}
&\qquad+\frac{1}{\psi'(0+)} \int_{x}^{\infty} \E_z(g^{p-1}) \dd z.
\end{align}
Then, we prove that the above two integrals are finite for all $x\leq 0$. From the fact that $z\mapsto \E_z(g^{p-1})$ is continuous on $\R$ and $W$ is continuous on $(0,\infty)$ we can assume without of loss of generality that $x=0$. \\

First, we show that the first integral on the right-hand side of (\ref{eq:finitenessintegralExgn}) is finite. From Lemma \ref{lemma:boundaryformomentsoftau0+p} we have that
%

\begin{align*}
 \int_0^{\infty} &\E_{-z}(g^{p-1}) \left[ 1-\psi'(0+)W(z) \right]\dd z\\
 &\leq 2^{p-1} \E(-\underline{X}_{\infty})[\E(g^{p-1})+A_{p-1}] + \frac{2^{p-1}}{p}C_{p-1} \E((-\underline{X}_{\infty})^p),
\end{align*}
where $A_{p-1}$ and $C_{p-1}$ are non-negative constants. In the equality above we relied on the fact that $z\mapsto \psi(0+)W(z)$ corresponds to the distribution function of the random variable $-\underline{X}_{\infty}$. We conclude from Lemma \ref{lemma:finitenessofEgn} that

\begin{align*}
\int_0^{\infty} \E_{-z}(g^{p-1}) \left[ 1-\psi'(0+)W(z) \right]\dd z <\infty.
\end{align*}
Now we proceed to check the finiteness of the second integral in (\ref{eq:finitenessintegralExgn}) when $x=0$. Using the strong Markov property we have that

\begin{align*}
\int_0^{\infty}& \E_z(g^{p-1})\dd z\\
&= \int_0^{\infty} \E_z(g^{p-1}\I_{\{\tau_0^-<\infty \}})\dd z\\
&\leq 2^{p-1}\int_0^{\infty} \E_z((\tau_0^-)^{p-1}\I_{\{\tau_0^-<\infty \}})\dd z+2^{p-1} \int_0^{\infty}\E_z( \E_{X_{\tau_0^-}}(g^{p-1})\I_{\{\tau_0^-<\infty \}} )\dd z\\
&\leq 2^{p-1}\int_0^{\infty} \E_z((\tau_0^-)^{p-1}\I_{\{\tau_0^-<\infty \}})\dd z+2^{p-1} \int_0^{\infty}\E_z( \E_{\underline{X}_{\infty}}(g^{p-1})\I_{\{\underline{X}_{\infty}<0 \}} ) \dd z,
\end{align*}
where in the last inequality we used the fact that $\underline{X}_{\infty} \leq X_{\tau_0^{-}}$ and that $x\mapsto \E_x(g^{p-1})$ is a non-increasing function. Using Fubini's theorem we have that

\begin{align*}
\int_0^{\infty}\E_z( \E_{\underline{X}_{\infty}}(g^{p-1})\I_{\{\underline{X}_{\infty}<0 \}} )\dd z&=\int_0^{\infty} \int_{(-\infty,0)} \E_{y}(g^{p-1}) \P_z(\underline{X}_{\infty}\in \dd y)\dd z\\
&=\int_{(-\infty,0)}\E_y(g^{p-1}) \int_0^{\infty}\P_z(\underline{X}_{\infty}\in \dd y) \dd z \\
&=\int_{0}^{\infty}\E_{-y}(g^{p-1}) [1-\psi'(0+)W(y)] \dd y\\
&<\infty.
\end{align*}
It thus only remains to show that
\begin{align*}
\int_0^{\infty} \E_z((\tau_0^-)^{p-1}\I_{\{\tau_0^-<\infty \}})\dd z<\infty.
\end{align*}
For this, define the function $F_1(q):=\int_0^{\infty} \E_z(e^{-q \tau_0^-}\I_{\{\tau_0^-<\infty \}})\dd z$. Differentiating equation (\ref{eq:laplacetransformofrunninginfimumexptime}) with respect to $\beta$ and evaluating at zero we obtain that

\begin{align*}
F_1(q)=\int_0^{\infty} \P(-\underline{X}_{\e_q}>z)\dd z=\E(-\underline{X}_{\e_q})=\frac{1}{\Phi(q)}-\frac{\psi'(0+)}{q},
\end{align*}
where $\e_q$ is an independent exponential random variable with parameter $q>0$. On the other hand, define the function $F_2(q)=\int_0^{\infty} \E_{-z}(e^{-q \tau_0^+})[1-\psi'(0+)W(z)]\dd z$. Using the expression for the Laplace transform of $\tau_0^+$ in \eqref{eq:laplacetransformtau0} and the definition of $W$, we have that

\begin{align*}
F_2(q)=\int_0^{\infty} e^{-\Phi(q) z}[1-\psi'(0+)W(z)]\dd z=\frac{1}{\Phi(q)}-\frac{\psi'(0+)}{q}=F_1(q).
\end{align*}
The fact that $F_2=F_1$ implies that, when $\alpha$ is a natural number, we can take derivatives of order $\alpha$ (with the help of the dominated convergence theorem) at $q=0$ and conclude that

\begin{align*}
\int_0^{\infty} \E_z((\tau_0^-)^{\alpha}\I_{\{\tau_0^-<\infty \}})\dd z<\infty
\end{align*}
if and only if  
\begin{align*}
\int_0^{\infty} \E_{-z}( (\tau_0^+)^{\alpha})[1-\psi'(0+)W(z)]\dd z<\infty.
\end{align*}
Furthermore, if $\alpha=k+\lambda$, with $k$ a positive integer and $0< \lambda <1$, we can draw the same conclusion using the Marchaud derivative (see e.g. \cite{laue1980remarks}). Using Lemma \ref{lemma:boundaryformomentsoftau0+p} we have that
\begin{align*}
\int_0^{\infty} \E_{-z}( (\tau_0^+)^{p-1})[1-\psi'(0+)W(z)]\dd z<\infty.
\end{align*}
and the proof is complete.
\end{proof}

\begin{proof}[Proof of Lemma \ref{lemma:finitenesofEtauD}]
Let $x\leq 0$ and take $\delta>0$. Then
\begin{align*}
\E_{0,x}((\tau_{D})^p)=\E_x((\tau_b^{g,0})^{p})& \leq \E_x((\tau_b^{g,0})^{p}\I_{\{g +\delta<\tau_b^{g,0} \}})+\E_x((g+\delta)^{p}\I_{\{g +\delta>\tau_b^{g,0} \}}).
\end{align*}
Note that on the event $\{g+\delta<\tau_b^{g,0} \}$ we have

\begin{align*}
\tau_b^{g,0}&=\inf\{t>g+\delta: X_t \geq b(U_t) \}\\
&=\inf\{t>0: X_{t+g+\delta} \geq b(t+\delta) \}+g+\delta\\
&\leq \inf\{t>0: X_{t+g+\delta} \geq b(\delta) \}+g+\delta,
\end{align*}
where the second equality follows from the fact that after $g$, the process $X$ never goes back below zero, and the last inequality holds since $b$ is non-increasing. Then, we have that 

\begin{align*}
\E_x((\tau_b^{g,0})^{p})&\leq \E_x( (\inf\{t>0: X_{t+g+\delta} \geq b(\delta) \}+g+\delta)^p \I_{\{g +\delta<\tau_b^{g,0} \}})\\
&\qquad+\E_x((g+\delta)^{p}\I_{\{g +\delta>\tau_b^{g,0} \}})  \\
&\leq 2^p \E_x( (\inf\{t>0: X_{t+g+\delta} \geq b(\delta) \})^{p})+(2^p+1)	\E_x  ((g+\delta)^p   )\\
&=2^p \E( (\inf\{t>0: X_{t+g^{(-x)}+\delta}+x \geq b(\delta) \})^{p})+(2^p+1)	\E_x( (g+\delta)^p   ),
\end{align*}
where $g^{(-x)}=\sup\{t\geq 0: X_t\leq -x \}$. From \cite{bertoin1998levy} (see Corollary VII.4.19) we know that  the law of the process $\{X_{t+g^{(-x)}}-(-x) , t\geq 0  \}$ is the same as that of $\P^{\uparrow}$, where $\P^{\uparrow}=\P^{\uparrow}_0$ and $\P^{\uparrow}_x$, for $x\geq 0$, corresponds to the law of $X$ starting at $x$ conditioned to stay positive. Moreover, from Proposition VII.3.14 in \cite{bertoin1998levy}, we know that the canonical process $X$ is a Feller process for the family $\{\P^{\uparrow}_x, x\geq 0 \}$. Hence, using the Markov property of $X$ under $\{ \P^{\uparrow}_x, x \geq 0 \}$ and equation VII.3.(6) in \cite{bertoin1998levy} we get

\begin{align*}
\E_x((\tau_b^{g,0})^{p})
&\leq 2^p \E( (\inf\{t>0: X_{t+g^{(-x)}+\delta}+x \geq b(\delta) \})^{p})+(2^p+1)	\E_x ( (g+\delta)^p   )\\
&=2^p \E^{\uparrow}( (\inf\{t>0: X_{t+\delta} \geq b(\delta) \})^{p})+(2^p+1)	\E_x(  (g+\delta)^p   )\\
&\leq 2^p \E^{\uparrow}( \E^{\uparrow}_{X_{\delta}} [(\tau_{b(\delta)}^+)^{p}]   )+(2^p+1)\E_x((g+\delta)^{p})\\
&=2^p \E^{\uparrow}\left( \frac{W(b(\delta))}{W(X_{\delta}) }  \E_{X_{\delta}} [(\tau_{b(\delta)}^+)^{p} \I_{\{ \tau_0^->\tau_{b(\delta)}^+ \}} ] \right)+(2^p+1)\E_x((g+\delta)^{p})\\
& \leq 2^p \E [(\tau_{b(\delta)}^+)^{p}]   \E^{\uparrow}\left( \frac{W(b(\delta))}{W(X_{\delta}) }   \right)+(2^p+1)\E_x((g+\delta)^{p})\\
& = 2^p \E [(\tau_{b(\delta)}^+)^{p}]  \int_{(0,\infty)}  \frac{W(b(\delta))}{W(z)} \P^{\uparrow}(X_{\delta} \in \dd z)   +(2^p+1)\E_x((g+\delta)^{p}),
\end{align*}
where the third inequality follows from the fact that $\E_x[(\tau_a^+)^p]\leq \E[(\tau_a^+)^p]$ for all $0\leq x\leq a$ and $X_{\delta}>0$ under $\P^{\uparrow}$. Thus, using that $\P^{\uparrow}(X_{\delta} \in \dd z)=[zW(z)/ \delta] \P(X_{\delta}\in \dd z)$ (see e.g. Corollary VII.3.16 in \cite{bertoin1998levy}) we have that

\begin{align}
\E_x((\tau_b^{g,0})^{p}) &\leq 2^p \E [(\tau_{b(\delta)}^+)^{p}]  \int_{(0,\infty)}  \frac{W(b(\delta))}{W(z)} \P^{\uparrow}(X_{\delta} \in \dd z)   +(2^p+1)\E_x((g+\delta)^{p}) \nonumber \\
\label{eq:finitenessoftaubg}
&=  2^p \E [(\tau_{b(\delta)}^+)^{p}]    \frac{W(b(\delta))}{ \delta} \E(X_{\delta}^+) +2^p(2^p+1)\delta^p+2^p(2^p+1)\E_x((g)^{p}),
\end{align}
where $X_{\delta}^+$ is the positive part of $X_{\delta}$. Thus from Lemma \ref{lemma:finitenessofEgn} we have that $\E_{0,x}((\tau_{D})^p)=\E_x((\tau_b^{g,0})^{p}) $ is finite for $x\leq 0$.\\

Next, we show that $\E_{u,x}((\tau_{D})^p)<\infty$ when $u,x>0$. From the Markov property of L\'evy processes, we have that
\begin{align*}
\E_{u,x}((\tau_{D})^p) &=\E_{x}((\tau_{b}^{u,0} )^p  \I_{\{\tau_b^{u,0}<\sigma_{0}^- \}})+\E_{x}((\tau_{b}^{g,0} )^p  \I_{\{\tau_b^{u,0}>\sigma_{0}^- \}})\\
&\leq \E_{x}((\tau_{b(u)}^+  )^p)+ 2^p \E_x( (\sigma_0^-)^p \I_{\{\sigma_0^- <\infty \}}) + 2^p  \E_{x}(\I_{\{\sigma_{0}^- <\infty\}}   \E_{X_{\sigma_0^-}}[(\tau_{b}^{g,0} )^p]  ).
\end{align*}
Using (\ref{eq:finitenessoftaubg}), the inequality $|X_{\sigma_0^-}|\leq |\underline{X}_{\infty}|$ under the event $\{ \sigma_0^-<\infty\}$ and Lemmas \ref{lemma:finitenessofEgn} and \ref{lemma:boundaryformomentsoftau0+p} we deduce that $\E_{u,x}((\tau_{D})^p) <\infty$ and the proof is complete.
\end{proof}

Using that $b$ is a right-continuous and a non-decreasing function, and that $X$ creeps upwards, it can be shown that for any $u\geq 0$ and $x \in \R$,
\begin{align*}
\lim_{h\rightarrow 0} \tau_b^{u,x+h}=\tau_b^{u,x} \text{ } \text{a.s.}\qquad \text{ and } \qquad\lim_{(h_1,h_2) \rightarrow (0,0)+} \tau_b^{u+h_1,x+h_2} =\tau_b^{u,x} \text{ } \text{a.s.}
\end{align*}
These facts will be useful in the proof of the continuity of the function $V$.

\begin{proof}[Proof of Lemma \ref{lemma:continuityofV0x}]
First, we show that the function $u \mapsto V(u,x)$ is continuous, for all $x>0$ fixed. Take $u_1, u_2 > 0$  and $x> 0$, then since the stopping time $\tau^*_{(u_1,x)}:=\tau_b^{u_1,x} \I_{\{ \tau_b^{u_1,x} <\sigma_{-x}^-\}}+\tau_b^{g,x} \I_{\{ \tau_b^{u_1,x}  \geq \sigma_{-x}^-\}}$ is optimal for $V(u_1,x)$ (under $\P$) we have that
\begin{align*}
V(u_1,x)&=\E \left(\int_0^{\sigma_{-x}^- \wedge \tau_b^{u_1,x}} G(u_1+s,X_{s}+x)\dd s +\I_{\{	 \tau_b^{u_1,x} \geq \sigma_{-x}^- \}} V(0,X_{\sigma_{-x}^-}+x) \right)\\
&=\E_x \left(\int_0^{\sigma_{0}^- \wedge \tau_b^{u_1,0}} G(u_1+s,X_{s})\dd s +\I_{\{	\tau_b^{u_1,0} \geq \sigma_{0}^- \}} V(0,X_{\sigma_{0}^-}) \right).
\end{align*}
On the other hand, from (\ref{eq:VintermsofPx}) we get 
\begin{align*}
V(u_2,x)&\leq \E\left( \int_0^{\tau^*_{(u_1,x)}} \left\{ G(u_2+s,X_s+x)\I_{\{\sigma_{-x}^->s\}}+G(U_s^{(-x)},X_s+x)\I_{\{\sigma_{-x}^-\leq s\}} \right\} \dd s\right)\\
&=\E_x\left( \I_{\{ \tau_b^{u_1,0} <\sigma_{0}^-\}}\int_0^{\tau_b^{u_1,0} } \left\{ G(u_2+s,X_s)\I_{\{\sigma_{0}^->s\}}+G(U_s,X_s)\I_{\{\sigma_{0}^-\leq s\}} \right\} \dd s\right) \\
&\qquad +\E_x\left( \I_{\{ \tau_b^{u_1,0} \geq \sigma_{0}^-\}}\int_0^{\tau_b^{g,0}  } \left\{ G(u_2+s,X_s)\I_{\{\sigma_{0}^->s\}}+G(U_s,X_s)\I_{\{\sigma_{0}^-\leq s\}} \right\} \dd s\right)\\
&=\E_x\left( \int_0^{\tau_b^{u_1,0} \wedge \sigma_0^- }  G(u_2+s,X_s) \dd s\right)+\E_x\left( \I_{\{ \tau_b^{u_1,0} \geq \sigma_{0}^-\}}\int_{\sigma_0^-}^{\tau_b^{g,0}  } G(U_s,X_s)  \dd s\right)\\
&=\E_x\left( \int_0^{\tau_b^{u_1,0} \wedge \sigma_0^- }  G(u_2+s,X_s) \dd s\right)+\E_x\left( \I_{\{ \tau_b^{u_1,0} \geq \sigma_{0}^-\}}V(0,X_{\sigma_0^-})\right),
\end{align*}
where in the first equality we used the definition of $\tau^*_{(u_1,x)}$ given above, in the second equality that $\tau_b^{u_1,0}\leq \tau_b^{g,0}$ and the last equality follows from the strong Markov property applied at time $\sigma_0^-$. Hence, we have that for any $x>0$ fixed and $u_1, u_2>0$, 
\begin{align*}
| V(u_2,x)-V(u_1,x)| &\leq  \E_x\left( \int_0^{\sigma_{0}^-\wedge \tau_b^{u_1,0}} |G(u_2+s,X_s)-G(u_1+s,X_s)|\dd s \right) \\
&\leq \E_x\left( \int_0^{\tau_{b(u_1)}^+} |G(u_2+s,X_s)-G(u_1+s,X_s)|\dd s \right)\\
&\leq \E\left( \int_0^{\tau_{b(u_1)}^+} |(u_2+s)^{p-1}-(u_1+s)^{p-1}|\dd s \right)\\
&=\frac{1}{p} |\E((\tau_{b(u_1)}^++u_2)^{p})-\E((\tau_{b(u_1)}^++u_1)^{p})-[u_2^p-u_1^p]|,
\end{align*}
where $\tau_{b(u_1)}^+=\inf\{t\geq 0: X_t >b(u_1)\}$. Thus, letting $u_2 \mapsto u_1$, together with the dominated convergence theorem and the fact that $\E((\tau_a^++u)^p)<\infty$, for all $u,a \geq 0$, we get that $u \mapsto V(u,x)$ is continuous uniformly over all $x>0$. \\

Next, we show that $x\mapsto V(u,x)$ is continuous. From equation (\ref{eq:expressionforV0xnegative}) we easily deduce that $x\mapsto V(0,x)$ is a continuous function on $(-\infty,0]$. Then, suppose that $u>0$ and $x>0$. Recall from equation (\ref{eq:Vwhenxpositive}) that we can write

\begin{align*}
V(u,x)&=\E \left(\int_{0}^{\sigma_{-x}^- \wedge \tau_b^{u,x}} G(u+s,X_{s}+x)\dd s\right) +\E(V(0,X_{\sigma_{-x}^-}+x)\I_{\{	\sigma_{-x}^- \leq \tau_b^{u,x} \}} ).
\end{align*}
Note that for all $s< \tau_b^{u,x} \wedge \sigma_{-x}^-$, it holds that $ 0< X_s+x\leq b(u+s)\leq b(u)$, and for all $x\in \R$ (see equation (\ref{eq:lowerboundforV})),  $ V(0,X_{\sigma_{-x}^-}+x)\I_{\{	\sigma_{-x}^- \leq \tau_b^{u,x} \}}\geq V(0,\underline{X}_{\infty}+x) \geq -A'_{p-1}-C'_{p-1} |\underline{X}_{\infty}+x|^p+V(0,0)$, where the last expression is integrable from Lemma \ref{lemma:finitenessofEgn}. Moreover, it can be shown that $\lim_{h \rightarrow 0} \sigma_{x+h}^-=\sigma_x^-$ a.s. and that  $\lim_{h \rightarrow 0} \tau_b^{u,x+h}=\tau_b^{u,x}$ a.s., for any $x\in \R$. Then, by the dominated convergence theorem, the fact that $V$ is continuous on $(-\infty,0]$ and $x\mapsto G(u,x)$ is continuous on $(0,\infty)$ we conclude that, for each $u>0$, the mapping $x\mapsto V(u,x)$ is continuous on $(0,\infty)$. Note that when $X$ is of infinite variation, $\lim_{h\downarrow 0} \sigma_{-h}^-=\tau_0^-=0$ a.s. and the previous argument also tells us that for all $u>0$,
\begin{align*}
\lim_{h \downarrow 0} V(u,h)=V(0,0).
\end{align*}
Note that the limit above implies that $\lim_{(u,x)\rightarrow (0,0)^+} V(u,x)=V(0,0)$ in the infinite variation case. Then we proceed to prove that this also holds when $X$ is of finite variation. In this case we know that $\tau_0^->0$ and then, due to the strong Markov property,
\begin{align*}
V(0,0)&=\E\left(\int_0^{\tau_b^{0,0} \wedge \tau_0^-} G(s,X_s)\dd s \right)+\E(\I_{\{ \tau_0^-<\tau_b^{0,0} \}}V(0,X_{\tau_0^-} )),
\end{align*}
where $\tau_b^{0,0}=\inf\{t>0: X_t \geq b(s)\}$. Taking limits in (\ref{eq:Vwhenxpositive}), we have from the dominated convergence theorem,
\begin{align*}
\lim_{(u,x)\rightarrow (0,0)^+} V(u,x)&=\lim_{(u,x)\rightarrow (0,0)^+} \E \left(\int_{0}^{\sigma_{-x}^- \wedge \tau_b^{u,x}} G(u+s,X_{s}+x)\dd s\right)\\
&\qquad +\lim_{(u,x)\rightarrow (0,0)^+}  \E(V(0,X_{\sigma_{-	x}^-}+x)\I_{\{	\sigma_{-x}^- \leq \tau_b^{u,x} \}} )\\
&=\E\left(\int_0^{\tau_0^- \wedge \tau_b^{0,0} } G(s,X_s)\dd s \right)+\E(\I_{\{ \tau_0^-<\tau_b^{0,0} \}}V(0,X_{\tau_0^-} ))\\
&=V(0,0),
\end{align*}
where we used that $\lim_{x \downarrow 0} \sigma_{-x}^-=\tau_0^-$ and $\lim_{(u,x)\rightarrow (0,0)+} \tau_b^{u,x}=\tau_b^{0,0}$ a.s. Therefore, $V$ is continuous on the set $E$.
\end{proof}

Before proving Lemma \ref{lemma:smoothfit} we first consider a technical lemma involving the derivative of the potential measure killed on exiting $[0,a]$. More specifically, for fixed $a> 0$, $x\in (0,a)$ and $r\in \mathbb{N}\cup\{0\}$ define the measure
\begin{align*}
U_r(a,x,\dd y)=\int_0^{\infty} t^r \P_x(X_t\in \dd y, t<\sigma_0^-\wedge \tau_a^+)\dd t.
\end{align*}

\begin{lemma}
\label{lemma:derivativeofpotentialmeasure}
Let $q\in \mathbb{N}\cup\{0\}$ such that $\int_{(-\infty,-1)} |x|^{q}\Pi(\dd x)<\infty$. Fix $a>0$ and $0\leq x\leq a$. We have that for all $r\in \{0,1,\ldots,q\}$, the measure $U_r(a,x,\dd y)$ is absolutely continuous with respect to the Lebesgue measure. It has a density $u_r(a,x,y)$ given by
\begin{align*}
u_r(a,x,y)=\lim_{q\downarrow 0} (-1)^r \frac{\partial^r }{\partial q^r} \left[ \frac{W^{(q)}(x) W^{(q)}(a-y) }{W^{(a)}(a)}-W^{(q)}(x-y) \right],
\end{align*}
for $y\in (0,a]$. Moreover, for a fixed $a>0$, the functions $x\mapsto \E_x((\tau_a^+)^r \I_{\{\sigma_0^-<\tau_a^+ \}})$ and $x\mapsto  u_r(a,x,y)$ have finite left derivatives on $(0,a]$ for all $y\in (0,a)$ and $r\in \{ 0,1,\ldots,q\}$.
\end{lemma}

\begin{proof}
Let $a>0$ and $x\in (0,a)$. First, we show that for all $r\in \{0,1,\ldots ,q\}$, the measure $U_r(a,x,\dd y)$ is absolutely continuous with respect to the Lebesgue measure. Take any measurable set $A\subset (0,a)$, thus by Fubini's theorem,

\begin{align*}
\int_A U_r(a,x,\dd y) &=\int_0^{\infty } t^r \P_x(X_t \in A, t<\sigma_0^- \wedge \tau_a^+) \dd t\\
&=\E_x\left(\int_0^{\tau_a^+\wedge \sigma_0^-} t^r \I_{\{X_t\in A \}} \dd t \right).
\end{align*}
From Lemma \ref{lemma:finitenessofEgn} we know that $\E_x((\tau_a^+)^r )<\infty$ for all $r\in \{0,1,\ldots,q \}$. Then, by dominated convergence theorem we have that
\begin{align*}
\int_A U_r(a,x,\dd y) &=\lim_{q\downarrow 0} \E_x\left(\int_0^{\tau_a^+\wedge \sigma_0^-}  t^r e^{-q t}   \I_{\{X_t\in A \}} \dd t \right)\\
&= \int_A  \lim_{q\downarrow 0} (-1)^r \frac{\partial^r }{\partial q^r} \left[ \frac{W^{(q)}(x) W^{(q)}(a-y) }{W^{(a)}(a)}-W^{(q)}(x-y) \right]\dd y,
\end{align*}
where the last equality follows from (\ref{eq:qpotentialdensitytkillingonexiting0,a}). From the convolution representation of $W^{(q)}$ (see equation (\ref{eq:convolutionrepresentationofWq})), the derivatives in the last equation above exist and indeed $u_r(a,x,y)$ is a density of $U_r(a,x,\dd y)$ for all $y\in (0,a)$.
Now we proceed to show the differentiation statements. Note that from equations (\ref{eq:laplacetransformtau0}) and (\ref{eq:laplacetransformoftaua+beforecrossingthelevelzero}) we have that
\begin{align*}
f_x(q):=\E_x(e^{-q\tau_a^+}\I_{\{\sigma_0^-<\tau_a^+ \}})=e^{\Phi(q)(x-a) } -\frac{W^{(q)}(x)}{W^{(q)}(a)},
\end{align*}
for any $x\in (0,a)$. Hence, we have that 
\begin{align*}
\E_x((\tau_a^+)^r \I_{\{\sigma_0^-<\tau_a^+ \}})=\lim_{q\downarrow 0} (-1)^r \frac{\partial^r }{\partial q^r}\left[e^{\Phi(q)(x-a) } -\frac{W^{(q)}(x)}{W^{(q)}(a)}\right]
\end{align*}
for any $x\in (0,a)$. From \eqref{eq:derivativesofLaplacetransformoftau+} we know that the first term above is differentiable with respect to the variable $x$ (note that is also possible to calculate directly the derivatives of $\Phi(q)$ by using the identity $\psi(\Phi(q))=q$ and the chain rule). Moreover, from \eqref{eq:convolutionrepresentationofWq} we can see that for any $q\geq 0$, $x\geq 0$ and $r\in \{0,1,2,\ldots\}$, 
\begin{align*}
\lim_{q\downarrow 0} \frac{\partial^r }{\partial q^r}W^{(q)}(x) = r!W^{*(r+1)}(x).
\end{align*}
Since $W$ is $C^1((0,\infty))$ and has left and right derivatives at zero (see discussion above equation \ref{eq:convolutionrepresentationofWq}), we conclude that $x\mapsto \E_x((\tau_a^+)^r \I_{\{\sigma_0^-<\tau_a^+ \}})$ has finite left derivatives on $(0,a]$, for any $r\in \{0,1,\ldots,q\}$. A similar argument works for the function $x\mapsto u_r(a,x,y)$.
\end{proof}
Before we prove Lemma \ref{lemma:smoothfit} we show that $V(0,x)$ is smooth on $(-\infty,0)$. This auxiliary result will also be useful later for the proof of Lemma \ref{lemma:proofofequationforV0}.

\begin{lemma}
\label{lemma:diferrentiabilityofV(0,x)}
We have that the function $V(0,x)$ is continuously differentiable on $(-\infty,0)$ and has left derivative at $0$. Moreover, we have that $\frac{\partial}{\partial x} V(0,x)$ is non-increasing and non-negative on $(-\infty,0)$, and for any $x<0$, 
\begin{align}
\label{eq:boundaryforderivativeofVbelowzero}
\frac{\partial}{\partial x} V(0,x)=\int_{0}^{\infty}  \E_{x-u}(g^{p-1}) W'(u) \dd u \leq \alpha_{p-1}+\gamma_{p-1}|x|^{p-1},
\end{align}
where $\alpha_{p-1}$ and $\gamma_{p-1}$ are non-negative constants.
\end{lemma}
\begin{proof}
We start showing that the inequality in \eqref{eq:boundaryforderivativeofVbelowzero} holds. Indeed, using Lemma \ref{lemma:boundaryformomentsoftau0+p} and the fact that $\P(-\underline{X}_{\infty} \in \dd u)=\psi'(0+)W'(u) \dd u$, we get that for all $x<0$ and $u\geq 0$,
\begin{align*}
\int_{0}^{\infty} & \E_{x-u}(g^{p-1}) W'(u) \dd u \\
&\leq \int_0^{\infty} (2^{p-1}[\E(g^{p-1})+A_{p-1}] +2^{p-1} C_{p-1} |x-u|^{p-1}) W'(u) \dd u  \\
&\leq \frac{2^{p-1}[ \E(g^{p-1})+A_{p-1}]+4^{p-1} C_{p-1} \E((-\underline{X}_{\infty})^{p-1}) }{\psi'(0+)}  +\frac{4^{p-1}  C_{p-1}}{\psi'(0+)}  |x|^{p-1},
\end{align*}
where the second inequality follows by using Lemma \ref{lemma:basicinequality}. Therefore, we conclude that the inequality in \eqref{eq:boundaryforderivativeofVbelowzero} holds by noticing that $\E((-\underline{X}_{\infty})^{p-1})<\infty$ (see Lemma \ref{lemma:finitenessofEgn}). Next, we can easily deduce from the continuity of $x\mapsto \E_x(g^{p-1})$ (see Lemma \ref{lemma:propertiesofExgn}) and the dominated convergence theorem that the mapping $x\mapsto \int_{0}^{\infty}  \E_{x-u}(g^{p-1}) W'(u) \dd u$ is continuous on $(-\infty,0]$. Thus, from (\ref{eq:expressionforV0xnegative}) we deduce that for any $x< 0$,
\begin{align*}
\frac{\partial}{ \partial x} V(0,x)=  \int_{0}^{\infty}  \E_{x-u}(g^{p-1}) W'(u) \dd u.
\end{align*}
Therefore, we conclude that $x\mapsto V(0,x)$ is continuously differentiable on $(-\infty,0)$ and has left derivative at zero as claimed. Lastly, since $x\mapsto \E_x(g^{p-1})$ is non-increasing and non-negative (see Lemma \ref{lemma:propertiesofExgn}) and $W'(u)> 0$ for all $u>0$, we deduce that $\frac{\partial}{\partial x} V(0,x)$ is a non-increasing and non-negative function. The proof is now complete.
\end{proof}

We are now ready to prove that the partial derivatives of $V$ at $(u,b(u))$ exist and are equal to zero.  
\begin{proof}[Proof of Lemma \ref{lemma:smoothfit}]
We first show that for all $u>0$ such that $b(u)>0$,
\begin{align*}
\frac{\partial}{\partial u} V(u,b(u))=0.
\end{align*}
From the proof of Lemma \ref{lemma:continuityofV0x} we know that for any $h>0$,
\begin{align*}
0\leq \frac{ V(u,b(u))-V(u-h,b(u))}{h}
\leq \E_{b(u)}\left( \int_0^{\tau_{b(u-h)}^+} \frac{[(u+s)^{p-1}-(u-h+s)^{p-1}]}{h}\dd s \right).
\end{align*}
The result then follows by taking $h\downarrow  0$, from the fact that the function $u\mapsto u^{p}$ is differentiable on $[0,\infty)$, the dominated convergence theorem and since $b$ is continuous.\\

Next, we proceed to show that the smooth fit condition on the spatial argument holds for $u<u_b$, that is, 
\begin{align*}
\frac{\partial}{\partial x} V(u,b(u))=0.
\end{align*}
Let $x>0$, $u>0$ and $0<\varepsilon<1$ such that $x-\varepsilon>0$ and $b(u)>0$. From equation (\ref{eq:Vwhenxpositive}) we know that 
\begin{align*}
V(u,x-\varepsilon)
&=\E\left( \int_0^{\tau_{b}^{u,x-\varepsilon} \wedge \sigma_{\varepsilon-x}^-} G(u+s,X_s+x-\varepsilon)\dd s\right)\\
&\qquad +\E\left( \I_{\{\sigma_{\varepsilon-x}^- <\tau_b^{u,x-\varepsilon} \}} \int_{\sigma_{\varepsilon-x}^-}^{\tau_b^{g,x-\varepsilon}} G(U_s^{(\varepsilon-x)}, X_s+x-\varepsilon)\dd s \right)\\
&=\E_x\left( \int_0^{\tau_{b}^{u,-\varepsilon} \wedge \sigma_{\varepsilon}^-} G(u+s,X_s-\varepsilon)\dd s\right)\\
&\qquad+\E_x\left( \I_{\{\sigma_{\varepsilon}^- <\tau_b^{u,-\varepsilon} \}} \int_{\sigma_{\varepsilon}^-}^{\tau_b^{g,-\varepsilon}} G(U_s^{(\varepsilon)}, X_s-\varepsilon)\dd s \right)\\
&=\E_x\left( \int_0^{\tau_{b}^{u,-\varepsilon} \wedge \sigma_{\varepsilon}^-} G(u+s,X_s-\varepsilon)\dd s\right)\\
&\qquad+\E_x\left( \I_{\{\sigma_{\varepsilon}^- <\tau_b^{u,-\varepsilon} \}} \int_{\sigma_{\varepsilon}^-}^{\tau_b^{g,-\varepsilon}\wedge \sigma_0^-} G(U_s^{(\varepsilon)}, X_s-\varepsilon)\dd s \right)\\
&\qquad +\E_x\left(\I_{\{\sigma_{0}^- <\tau_b^{g,-\varepsilon} \}}\I_{\{\sigma_{\varepsilon}^-< \tau_{b}^{u,-\varepsilon} \}}  \int_{\sigma_{0}^-}^{\tau_b^{g,-\varepsilon}} G(U_s^{(\varepsilon)}, X_s-\varepsilon)\dd s \right),
\end{align*}
where in the last inequality we used that $\sigma_{\varepsilon}<\sigma_0^-$ under the measure $\P_x$. On the other hand, define the stopping time $\tau_*:=   \tau_b^{u,-\varepsilon} \I_{\{ \sigma_{\varepsilon}^- >\tau_b^{u,-\varepsilon} \}} + \tau_b^{g,-\varepsilon}\I_{\{ \sigma_{\varepsilon}^- <\tau_b^{u,-\varepsilon} \}}$. From equation (\ref{eq:VintermsofPx}) we have that
%
\begin{align*}
V(u,x) &\leq \E_x\left( \int_0^{\tau_* \wedge \sigma_0^-} G(u+s, X_s)\dd s\right)+\E_x \left( \I_{\{\sigma_0^-< \tau_* \}} \int_{\sigma_0^-}^{\tau_*}G(U_s,X_s)\dd s \right)\\
&=\E_x\left( \int_0^{\tau_b^{u,-\varepsilon} \wedge \sigma_{\varepsilon}^- } G(u+s, X_s)\dd s \right)\\
&\qquad+\E_x\left( \I_{\{\sigma_{\varepsilon}^- <\tau_b^{u,-\varepsilon} \}} \int_{\sigma_{\varepsilon}^-}^{\tau_b^{g,-\varepsilon} \wedge \sigma_0^-} G(u+s, X_s)\dd s \right)\\
&\qquad +\E_x \left( \I_{\{\sigma_0^-< \tau_{b}^{g,-\varepsilon} \}}  \I_{\{\sigma_{\varepsilon}^-< \tau_{b}^{u,-\varepsilon} \}}\int_{\sigma_0^-}^{\tau_b^{g,-\varepsilon}}G(U_s,X_s)\dd s \right),
\end{align*}
where we again used that $\sigma_{\varepsilon}^-\leq \sigma_0^-$. Hence, for any $u>0$, $0<x\leq b(u)$ and $0<\varepsilon<1$ such that $x-\varepsilon>0$,
\begin{align*}
0\leq \frac{V(u,x)-V(u,x-\varepsilon)}{\varepsilon} \leq R_1^{(\varepsilon)}(u,x)+R_2^{(\varepsilon)}(u,x)+R_3^{(\varepsilon)}(u,x),
\end{align*}
where

\begin{align*}
R_1^{(\varepsilon)}(u,x)&:=\frac{1}{\varepsilon}\E_x\left( \int_0^{\tau_b^{u,-\varepsilon} \wedge \sigma_{\varepsilon}^- } [G(u+s, X_s)- G(u+s,X_s-\varepsilon)]\dd s \right),\\
R_2^{(\varepsilon)}(u,x) &:=\frac{1}{\varepsilon} \E_x\left( \I_{\{\sigma_{\varepsilon}^- <\tau_b^{u,-\varepsilon} \}} \int_{\sigma_{\varepsilon}^-}^{\tau_b^{g,-\varepsilon} \wedge \sigma_0^-} [G(u+s, X_s)-G(U_s^{(\varepsilon)}, X_s-\varepsilon)   ]\dd s \right),\\
R_3^{(\varepsilon)}(u,x)&:= \frac{1}{\varepsilon} \E_x \left( \I_{\{\sigma_0^-< \tau_{b}^{g,-\varepsilon} \}}  \I_{\{\sigma_{\varepsilon}^-< \tau_{b}^{u,-\varepsilon} \}}\int_{\sigma_0^-}^{\tau_b^{g,-\varepsilon}}[G(U_s,X_s)-G(U_s^{(\varepsilon)}, X_s-\varepsilon)]\dd s \right).
\end{align*}
We show that $\lim_{\varepsilon \downarrow 0} R_i^{(\varepsilon)}(u,b(u))=0$ for $i=1,2,3$. From the fact that $b$ is non-increasing we have that $ \tau_b^{u,-\varepsilon}\leq \tau_{b(u)+\varepsilon}^+$. Then, for all $u\in (0,u_b)$ we have
\begin{align*}
0&\leq R_1^{(\varepsilon)}(u,b(u))\\
&\leq \frac{1}{\varepsilon}\E_{b(u)}\left( \int_0^{\tau_{b(u)+\varepsilon}^+ \wedge \sigma_{\varepsilon}^- } (u+s)^{p-1}\psi'(0+)[W(X_s)-W(X_s-\varepsilon)] \dd s\right)\\
&\qquad -\frac{1}{\varepsilon}\E_{b(u)-\varepsilon}\left( \int_0^{\tau_{b(u)}^+ \wedge \sigma_{0}^- } [\E_{X_s+\varepsilon}(g^{p-1})-\E_{X_s}(g^{p-1})] \dd s\right) \\
&=\frac{1}{\varepsilon}\E_{b(u)}\left( \int_0^{\tau_{b(u)+\varepsilon}^+ \wedge \sigma_{\varepsilon}^- } (u+s)^{p-1}\psi'(0+)[W(X_s)-W(X_s-\varepsilon)] \dd s\right)\\
&\qquad -\frac{1}{\varepsilon} \int_{(0,b(u))} [\E_{z+\varepsilon}(g^{p-1})-\E_{z}(g^{p-1})] \int_0^{\infty} \P_{b(u)-\varepsilon} (X_s \in \dd z, t< \tau_{b(u)}^+ \wedge \sigma_{0}^- ) \dd s.
\end{align*}
Using the density of the $0$-potential measure of $X$ exiting the interval $[0,b(u)]$ given in equation (\ref{eq:qpotentialdensitytkillingonexiting0,a}), we obtain that
\begin{align*}
0&\leq R_1^{(\varepsilon)}(u,b(u))\\
&\leq \E_{b(u)}\left( \int_0^{\tau_{b(u)+\varepsilon}^+ \wedge \sigma_{\varepsilon}^- } (u+s)^{p-1}\psi'(0+)\frac{W(X_s)-W(X_s-\varepsilon)}{\varepsilon} \dd s\right)\\
&\qquad - \int_{0}^{b(u)-\varepsilon} [\E_{z+\varepsilon}(g^{p-1})-\E_{z}(g^{p-1})] \\
&\qquad\qquad \times \frac{1}{\varepsilon} \left[ \frac{W(b(u)-\varepsilon)W(b(u)-z)}{W(b(u))}-W(b(u)-\varepsilon-z)\right]\dd z \\
&\qquad -\frac{1}{\varepsilon} \int_{b(u)-\varepsilon}^{b(u)} [\E_{z+\varepsilon}(g^{p-1})-\E_{z}(g^{p-1})] \left[ \frac{W(b(u)-\varepsilon)W(b(u)-z)}{W(b(u))}\right]\dd z.
\end{align*}
Note that for all $s<\tau_{b(u)+\varepsilon}^+ \wedge \sigma_{\varepsilon}^-$, we have $X_s \in (\varepsilon, b(u)+\varepsilon)$. Then, using the fact that $W \in C^1 ((0,\infty)$), the function $z\mapsto \E_z (g^{p-1})$ is continuous, $ \lim_{ \varepsilon \downarrow 0} \tau_{b(u)+\varepsilon}^+ \wedge \sigma_{\varepsilon}^-  =\tau_{b(u)}^+ \wedge \sigma_{0}^-=0$ a.s., under $\P_{b(u)}$, and the dominated convergence theorem, we conclude that

\begin{align*}
\lim_{\varepsilon \downarrow 0}R_1^{(\varepsilon)}(u,b(u))=0.
\end{align*}
Now we show that $\lim_{\varepsilon \downarrow 0}R_2^{(\varepsilon)}(u,b(u))=0$. Take $0<x\leq b(u)$. Then, using the inequality $G(u,x)\leq u^{p-1}$, the fact that for $s<\sigma_0^-$, $X_s > 0$ (then $-\E_{-1}(g^{p-1})=G(0,-1)\leq G(U_s^{(\varepsilon)}, X_s-\varepsilon)$) and the strong Markov property at time $\sigma_{\varepsilon}^-$, we get that
\begin{align*}
0&\leq R_2^{(\varepsilon)}(u,x)\\
 &\leq  \frac{1}{\varepsilon}\E_x\left(\I_{\{\sigma_{\varepsilon}^- <\tau_b^{u,-\varepsilon} \}}[\tau_b^{g,-\varepsilon} \wedge \sigma_0^--\sigma_{\varepsilon}^-] [(u+\tau_b^{g,-\varepsilon} \wedge \sigma_0^-)^{p-1}+\E_{-1}(g^{p-1})   ]\right)\\
&\leq \frac{1}{\varepsilon}\E_x\left(\I_{\{\sigma_{\varepsilon}^- <\tau_{b(u)+\varepsilon}^+ \}}f(\sigma_{\varepsilon}^-,X_{\sigma_{\varepsilon}^-})  \right),
\end{align*}
where $f$ is given for all $t\geq 0$ and $x\in \R$ by
\begin{align*}
f(t,x):= [2^{p-1}(u+t)^{p-1}+\E_{-1}(g^{p-1})]\E_x(\tau_b^{g,-\varepsilon} \wedge \sigma_0^-)+2^{p-1}\E_x((\tau_b^{g,-\varepsilon} \wedge \sigma_0^-)^p)<\infty,
\end{align*}
where the last inequality follows due to Lemma \ref{lemma:finitenesofEtauD}. Note that $\E_x(\tau_b^{g,-\varepsilon} \wedge \sigma_0^-)=\E_x((\tau_b^{g,-\varepsilon} \wedge \sigma_0^-)^p)=0$ for all $x\leq 0$. Thus, from (\ref{eq:finitenessoftaubg}) there exists $M>0$ such that
\begin{align*}
\max\{ \E_x(\tau_b^{g,-\varepsilon} \wedge \sigma_0^-), \E_x((\tau_b^{g,-\varepsilon} \wedge \sigma_0^-)^p) \} \leq M
\end{align*}
for all $x\leq \varepsilon$. Hence, from the compensation formula for Poisson random measures, we get that

\begin{align*}
0\leq  &R_2^{(\varepsilon)}(u,x) \\
 &\leq \max\{ \E_{\varepsilon}(\tau_b^{g,-\varepsilon} \wedge \sigma_0^-), \E_{\varepsilon}((\tau_b^{g,-\varepsilon} \wedge \sigma_0^-)^p)\}  \\
 &\qquad \qquad \times \frac{1}{\varepsilon}\E_x\left(\I_{\{\sigma_{\varepsilon}^- <\tau_{b(u)+\varepsilon}^+ \}} [2^{p-1}(u+\tau_{b(u)+\varepsilon}^+)^{p-1}+\E_{-1}(g^{p-1})+2^{p-1}] \right)\\
 &\qquad + M\frac{1}{\varepsilon} \E_x\left(\I_{\{\sigma_{\varepsilon}^- <\tau_{b(u)+\varepsilon}^+ \}}[2^{p-1}(u+\sigma_{\varepsilon}^-)^{p-1}+\E_{-1}(g^{p-1})+2^{p-1}] \I_{\{0<X_{ \sigma_{\varepsilon}^-} <\varepsilon \}} \right)\\
 &=
 \max\{ \E_{\varepsilon}(\tau_b^{g,-\varepsilon} \wedge \sigma_0^-), \E_{\varepsilon}((\tau_b^{g,-\varepsilon} \wedge \sigma_0^-)^p)\} \\
 &\qquad \qquad \times \frac{1}{\varepsilon}\E_x\left(\I_{\{\sigma_{\varepsilon}^- <\tau_{b(u)+\varepsilon}^+ \}} [2^{p-1}(u+\tau_{b(u)+\varepsilon}^+)^{p-1}+\E_{-1}(g^{p-1})+2^{p-1}] \right)\\
 &\qquad + M\frac{1}{\varepsilon} \E_{x-\varepsilon} \bigg( \int_0^{\tau_{b(u)}^+ \wedge \sigma_0^-} \int_{(-\infty,0)} [2^{p-1}(u+t)^{p-1}-G(0,-1)+2^{p-1}] \\
 &\qquad \qquad \times \I_{\{-\varepsilon <X_t +y <0 \}}\Pi(\dd y)\dd t \bigg)\\
 &=
 \max\{ \E_{\varepsilon}(\tau_b^{g,-\varepsilon} \wedge \sigma_0^-), \E_{\varepsilon}((\tau_b^{g,-\varepsilon} \wedge \sigma_0^-)^p)\} \\
 &\qquad \qquad \times \frac{1}{\varepsilon}\E_x\left(\I_{\{\sigma_{\varepsilon}^- <\tau_{b(u)+\varepsilon}^+ \}} [2^{p-1}(u+\tau_{b(u)+\varepsilon}^+)^{p-1}+\E_{-1}(g^{p-1})+2^{p-1}] \right)\\
 &\qquad +   \int_0^{b(u)} \int_{(-\varepsilon-z,-z)}   \frac{M}{\varepsilon} \int_0^{\infty} [2^{p-1}(u+t)^{p-1}+\E_{-1}(g^{p-1})+2^{p-1}] \\
 &\qquad\qquad \times \P_{x-\varepsilon}(X_t\in \dd z, t<\tau_{b(u)}^+\wedge \sigma_0^-)\dd t  \Pi(\dd y).
\end{align*}
Letting $x=b(u)$ and tending $\varepsilon \downarrow 0$ we get from Lemma \ref{lemma:derivativeofpotentialmeasure} that

\begin{align*}
\lim_{\varepsilon \downarrow 0} R_2^{(\varepsilon)}(u,b(u))=0.
\end{align*}

Lastly, using the Markov property at time $\sigma_0^-$ and the fact that $\tau_{b}^{g,0} \leq \tau_b^{g,-\varepsilon}$, we get that
\begin{align*}
0 &\leq R_3^{(\varepsilon)}(u,x) \\
&=\frac{1}{\varepsilon} \E_{x} \left(  \I_{\{\sigma_0^-< \tau_{b}^{g,-\varepsilon} \}}  \I_{\{\sigma_{\varepsilon}^-< \tau_{b}^{u,-\varepsilon} \}} \E_{X_{\sigma_0^-}} \left[\int_{0}^{\tau_b^{g,-\varepsilon}}[G(U_s,X_s)-G(U_s^{(\varepsilon)}, X_s-\varepsilon)]\dd s  \right] \right)\\
&= \frac{1}{\varepsilon} \E_x \left( \I_{\{\sigma_0^-< \tau_{b}^{g,-\varepsilon} \}}  \I_{\{\sigma_{\varepsilon}^-< \tau_{b}^{u,-\varepsilon} \}}[V(0,X_{\sigma_0^-})-V(0,X_{\sigma_0^-}-\varepsilon)] \right)\\
&\qquad +\frac{1}{\varepsilon} \E_{x} \left(  \I_{\{\sigma_0^-< \tau_{b}^{g,-\varepsilon} \}}  \I_{\{\sigma_{\varepsilon}^-< \tau_{b}^{u,-\varepsilon} \}} \E_{X_{\sigma_0^-}} \left[\int_{\tau_b^{g,0}}^{\tau_b^{g,-\varepsilon}}G(U_s,X_s)\dd s  \right] \right)\\
&\leq  \frac{1}{\varepsilon} \E_x \left(   \I_{\{\sigma_{\varepsilon}^-< \tau_{b}^{u,-\varepsilon} \}}[V(0,X_{\sigma_0^-})-V(0,X_{\sigma_0^-}-\varepsilon)] \right)\\
& \qquad + \frac{1}{\varepsilon} \E_x \left( \I_{\{\sigma_0^-< \tau_{b}^{g,-\varepsilon} \}}  \I_{\{\sigma_{\varepsilon}^-< \tau_{b}^{u,-\varepsilon} \}} \E_{X_{\sigma_0^-}} \left( [\tau_b^{g,-\varepsilon}-\tau_b^{g,0}] (\tau_b^{g,-\varepsilon})^{p-1}\right)\right),
\end{align*}
where we used the fact that $G(U_s,X_s)\leq s^{p-1} \leq (\tau_b^{g,-\varepsilon})^{p-1}$ for all $s\in [\tau_b^{g,0}, \tau_b^{g,-\varepsilon}]$. 
Thus, since $|X_{\sigma_0^-}|\leq |\underline{X}_{\infty}|$, under the event $ \{\sigma_0^-<\infty\}$, and $\E_x((-\underline{X}_{\infty})^{p-1})<\infty$ for all $x\in \R$, we deduce from Lemma \ref{lemma:diferrentiabilityofV(0,x)} that the mapping $x \mapsto \E_x(\frac{\partial}{\partial x} V(0,X_{\sigma_{0}^-}) \I_{\{\sigma_0^-<\infty\}})$ is locally bounded. Hence, by the dominated convergence theorem and the right continuity of $b$, we have that
\begin{align*}
\lim_{\varepsilon \downarrow 0} \frac{1}{\varepsilon} \E_x \left(   \I_{\{\sigma_{\varepsilon}^-< \tau_{b}^{u,-\varepsilon} \}}[V(0,X_{\sigma_0^-})-V(0,X_{\sigma_0^-}-\varepsilon)] \right)=\E_x \left(   \I_{\{\sigma_{0}^-< \tau_{b}^{u,0} \}} \frac{\partial}{ \partial x} V(0,X_{\sigma_0^-}) \right).
\end{align*}
In particular, when $x=b(u)$, we have that equation above is equal to zero. On the other hand, by conditioning on $\sigma_{\varepsilon}^-$ we obtain that
\begin{align*}
\frac{1}{\varepsilon} & \E_x \left( \I_{\{\sigma_0^-< \tau_{b}^{g,-\varepsilon} \}}  \I_{\{\sigma_{\varepsilon}^-< \tau_{b}^{u,-\varepsilon} \}} \E_{X_{\sigma_0^-}} \left( [\tau_b^{g,-\varepsilon}-\tau_b^{g,0}] (\tau_b^{g,-\varepsilon})^{p-1}\right)\right)\\
&=\frac{1}{\varepsilon} \E_x \left(   \I_{\{\sigma_{\varepsilon}^-< \tau_{b}^{u,-\varepsilon} \}} f_2(\varepsilon, X_{\sigma_{\varepsilon}^-})\right),
\end{align*}
where
\begin{align*}
0 \leq f_2(\varepsilon, x)=\E_x\left(\I_{\{\sigma_0^-< \tau_{b}^{g,-\varepsilon} \}}  \E_{X_{\sigma_0^-}} \left( [\tau_b^{g,-\varepsilon}-\tau_b^{g,0}] (\tau_b^{g,-\varepsilon})^{p-1} \right)\right).
\end{align*}
We show that $f_2$ is a finite function. Conditioning with respect to $\tau_0^+$ and the strong Markov property of L\'evy processes we have for all $y\leq 0$,
\begin{align*}
\E_y&\left( [\tau_b^{g,-\varepsilon}-\tau_b^{g,0}] (\tau_b^{g,-\varepsilon})^{p-1} \right)\\
&\leq 2^p \E((\tau_b^{g,-\varepsilon})^{p})+2^p\E_y((\tau_0^+)^p) \leq 2^p \E((\tau_b^{g,-\varepsilon})^{p})+2^pA_p +2^pC_p |y|^p.
\end{align*}
where the last inequality follows from Lemma \ref{lemma:boundaryformomentsoftau0+p}. Hence, since $|X_{\sigma_0^-}|\leq |\underline{X}_{\infty}|$ on the event $\{\sigma_0^-<\infty \}$, we have that
\begin{align}
\label{eq:boundforauxiliaryfunctiong}
f_2(\varepsilon,x)\leq \left\{
\begin{array}{lc}
2^p \E((\tau_b^{g,-\varepsilon})^{p})+2^pA_p +2^pC_p \E_x( |\underline{X}_{\infty}|^p), & \text{ for }x>0,\\
2^p \E((\tau_b^{g,-\varepsilon})^{p})+2^pA_p +2^pC_p |x|^p, & \text{ for } x\leq 0.
\end{array}
\right.
\end{align}
From Lemmas \ref{lemma:finitenessofEgn} and \ref{lemma:finitenesofEtauD} we conclude that $f_2(\varepsilon,x)$ is a finite function. Moreover, from the fact that $b$ is continuous and $x\mapsto U_t^{(x)}$ is right continuous, we can show that $\lim_{\varepsilon \downarrow 0} \tau_b^{g,\varepsilon}=\tau_b^{g,0}$ a.s., and then, by the dominated convergence theorem, $\lim_{\varepsilon \downarrow 0}f_2(\varepsilon,x)=0$ for all $x\in \R$. Moreover, using the compensation formula for Poisson random measures, we get that
\begin{align*}
&\frac{1}{\varepsilon} \E_x \left(   \I_{\{\sigma_{\varepsilon}^-< \tau_{b}^{u,-\varepsilon} \}} f_2(\varepsilon, X_{\sigma_{\varepsilon}^-})\right)\\
&=\frac{1}{\varepsilon} \E_x \left(   \I_{\{\sigma_{\varepsilon}^-< \tau_{b(u)+\varepsilon}^+ \}} f_2(\varepsilon, X_{\sigma_{\varepsilon}^-})\right)\\
&=f_2(\varepsilon,\varepsilon)\frac{\P_x(\sigma_{\varepsilon}^-<\tau_{b(u)+\varepsilon}^+, X_{\sigma_{\varepsilon}^-}=\varepsilon)}{\varepsilon}\\
& \qquad +\frac{1}{\varepsilon} \E_x\left( \int_{[0,\infty)} \int_{(-\infty,0)} f_2(\varepsilon,X_{t-}+y) \I_{\{\overline{X}_{t-}< b(u)+\varepsilon \}}\I_{\{\underline{X}_{t-}> \varepsilon \}}\I_{\{X_{t-} +y \leq \varepsilon \}} N(\dd t,\dd y)  \right)\\
&\leq f_2(\varepsilon,\varepsilon)\frac{\P_x(\sigma_{\varepsilon}^-<\tau_{b(u)+\varepsilon}^+)}{\varepsilon}\\
& \qquad +\frac{1}{\varepsilon} \E_{x-\varepsilon}\left( \int_{0}^{\infty} \int_{(-\infty,0)} f_2(\varepsilon,X_{t}+\varepsilon+y) \I_{\{t< \tau_{b(u)}^+ \wedge \sigma_{0}^- \}}\I_{\{X_{t} +y \leq 0 \}} \Pi(\dd y) \dd t  \right).
\end{align*}
From the $0$-potential density of the process killed on exiting $[0,b(u)]$ (see equation (\ref{eq:qpotentialdensitytkillingonexiting0,a})) and from equation (\ref{eq:laplacetransformoftaua+beforecrossingthelevelzero}) we obtain
\begin{align*}
\frac{1}{\varepsilon} \E_x &\left(   \I_{\{\sigma_{\varepsilon}^-< \tau_{b}^{u,-\varepsilon} \}} f_2(\varepsilon, X_{\sigma_{\varepsilon}^-})\right)\\
& \leq  f_2(\varepsilon,\varepsilon)\frac{W(b(u))-W(x-\varepsilon)}{\varepsilon W(b(u))}\\
& \qquad +\frac{1}{\varepsilon} \int_{(0,b(u))} \int_{(-\infty,0)}  f_2(\varepsilon,z+\varepsilon+y)\I_{\{z +y \leq 0 \}} \Pi(\dd y)   \\
&\qquad\qquad \times   \int_{0}^{\infty} \P_{x-\varepsilon}(X_t \in \dd z,t< \tau_{b(u)}^+ \wedge \sigma_{0}^- ) \dd t    \\
&=f_2(\varepsilon,\varepsilon)\frac{W(b(u))-W(x-\varepsilon)}{\varepsilon W(b(u))}\\
& \qquad +\frac{1}{\varepsilon} \int_{0}^{(x-\varepsilon) \vee 0}  \left[ \frac{W(x-\varepsilon)W(b(u)-z)}{W(b(u))}-W(x-\varepsilon-z)\right] \\
&\qquad\qquad \times  \int_{(-\infty,-z)}  f_2(\varepsilon,z+\varepsilon+y)  \Pi(\dd y)\dd z  \\
& \qquad +\frac{1}{\varepsilon} \int_{(x-\varepsilon)\vee 0}^{b(u)}   \frac{W(x-\varepsilon)W(b(u)-z)}{W(b(u))}  \int_{(-\infty,-z)}  f_2(\varepsilon,z+\varepsilon+y)  \Pi(\dd y)\dd z.
\end{align*}
Note that since $\Pi$ is finite on sets of the form $(-\infty,-\delta)$, for all $\delta>0$, Lemma \ref{lemma:finitenessofEgn} and equation (\ref{eq:boundforauxiliaryfunctiong}), we have that the integrals above with respect to $\Pi$ are finite and bounded. Hence, taking $x=b(u)$ and from the dominated convergence theorem we conclude that
\begin{align*}
\lim_{\varepsilon \downarrow 0} \frac{1}{\varepsilon} \E_{b(u)} &\left(   \I_{\{\sigma_{\varepsilon}^-< \tau_{b}^{u,-\varepsilon} \}} g(\varepsilon, X_{\sigma_{\varepsilon}^-})\right) \leq 0.
\end{align*}
Therefore, we also have that
\begin{align*}
\lim_{\varepsilon  \downarrow 0} R_3{(\varepsilon)}(u,b(u))=0.
\end{align*}
The conclusion of the lemma holds.
\end{proof}

\begin{proof}[Proof of Lemma \ref{lemma:integrativilityconditionsofGbandVb}]
Let $(u,x)\in E$, we first show that (\ref{eq:integrabilityofGbelowb}) is satisfied. Indeed, using that $|G(u,x)|<u^{p-1}+\E_x(g^{p-1})$ and that $U_s\leq u+ s$ (under $\P_{u,x}$) we obtain that
\begin{align}
 \E_{u,x}& \left( \int_0^{\infty} |G(U_s,X_s)| \I_{\{ X_s <b(U_s) \}}\dd s \right) \nonumber\\
\label{eq:boundaryforGIbusinabsolutevalue}
& \leq   \E_{u,x} \left( \int_0^{\infty} (u+s)^{p-1} \I_{\{ X_s <b(U_s) \}}\dd s \right)+ \E_{x} \left( \int_0^{\infty} \E_{X_s}(g^{p-1})\dd s \right). 
\end{align}
From Lemma \ref{lemma:finitenessofintExgn} we know that the second integral above is finite. Now we check that the first integral above is also finite. Consider $\delta>0$ and let $g^{(b(\delta))}$ be the last time $X$ is below the level $b(\delta)$. Then, we have that $g^{(b(\delta))} \geq g$ and $X_{s+g^{(b(\delta))} +\delta} \geq b(\delta)\geq b(U_s)$ for all $s \geq 0$. Hence, since $b$ is non-increasing and $p>1$ we get
\begin{align}
\E_{u,x} \left( \int_0^{\infty} (u+s)^{p-1} \I_{\{ X_s <b(U_s) \}}\dd s \right) & = \E_{u,x} \left( \int_0^{g^{(b(\delta))}+\delta} (u+s)^{p-1} \I_{\{ X_s <b(U_s) \}}  \dd s \right) \nonumber\\
\label{eq:integrabilityofGIbusinabsolutevalue}
&\leq [\E_x((g^{(b(\delta))}+\delta+u)^p)-u^p]\\
&<\infty \nonumber,
\end{align}
where the last expectation is finite by Lemma \ref{lemma:finitenessofEgn}. Therefore, we conclude that (\ref{eq:integrabilityofGbelowb}) holds. Moreover, assume that $X$ is of finite variation with $\Pi (-\infty,0)<\infty$ and take $u>u_b$. Then, we have that $b(u+s)=0$ for all $s\geq 0$, and for any $\delta>0$ we have that 
\begin{align*}
\E_{u,x} \left( \int_0^{\infty}G(U_s,X_s) \I_{\{ X_s <b(U_s) \}}\dd s \right)
&=\E_{x} \left( \int_{\sigma_0^-}^{\infty}G(U_s,X_s) \I_{\{ X_s <b(U_s) \}}\dd s \right) \\
&\leq \E_{x} \left( \int_{0}^{\infty}|G(U_s,X_s)| \I_{\{ X_s <b(U_s) \}}\dd s \right)\\
&\leq \E_x((g^{(b(\delta))}+\delta)^p)+ \E_{x} \left( \int_0^{\infty} \E_{X_s}(g^{p-1})\dd s \right),
\end{align*}
where the last inequality follows from equations \eqref{eq:boundaryforGIbusinabsolutevalue} and \eqref{eq:integrabilityofGIbusinabsolutevalue}. From the fact that $x\mapsto \E_x(g^{p})$ is non increasing, that $\lim_{x\rightarrow \infty} \E_x(g^{p})=0$ (see Lemma \ref{lemma:propertiesofExgn}) and the dominated convergence theorem we then deduce that 
\begin{align*}
 \lim_{u,x \rightarrow \infty} & \E_{u,x} \left( \int_0^{\infty}G(U_s,X_s) \I_{\{ X_s <b(U_s) \}}\dd s \right) \\
 &\leq \lim_{x \rightarrow \infty} \E_x((g^{(b(\delta))}+\delta)^p)+ \lim_{x \rightarrow \infty}\E_{x} \left( \int_0^{\infty} \E_{X_s}(g^{p-1})\dd s \right) \\
 &= \delta
\end{align*}
for any $\delta>0$. Hence, by letting $\delta \downarrow 0$, we conclude that when $X$ is of finite variation with finite activity,
\begin{align*}
\lim_{u,x \rightarrow \infty}  \E_{u,x} \left( \int_0^{\infty}G(U_s,X_s) \I_{\{ X_s <b(U_s) \}}\dd s \right) =0.
\end{align*}

Next, we prove that (\ref{eq:integrabilityVPidy}) also holds. Take $(u,x)\in E$ and let $N>0$ be any positive number. Then, we have that
\begin{align}
\int_{(-\infty,0)} & \widetilde{V}(u,x+y)\Pi(\dd y)\I_{\{x>b(u) \}}\nonumber\\
&=\int_{(-\infty,0)}  \widetilde{V}(u,x+y)\Pi(\dd y)\I_{\{u\leq N \}}\I_{\{x>b(u) \}} \nonumber\\
&\qquad +\int_{(-\infty,0)}  \widetilde{V}(u,x+y)\Pi(\dd y) \I_{\{u>N \}}\I_{\{x>b(N) \}}\nonumber\\
\label{eq:decompositionofK2onthreefunctions}
&\qquad +\int_{(-\infty,0)}  \widetilde{V}(u,x+y)\Pi(\dd y) \I_{\{u>N \}}\I_{\{b(N)\geq x>b(u) \}}.
\end{align}
Using that $\int_{(-\infty,0)}V(u,x+y)\Pi(\dd y)+G(u,x)\geq 0$ for all $u>0$ and $x>b(u)$ (see Lemma \ref{lemma:auxiliaryfunctionLambda}), that $G(u,x)\leq u^{p-1}$ for all $(u,x)\in E$ and that $b$ is non increasing, we obtain the lower bound
\begin{align}
\int_{(-\infty,0)}  \widetilde{V}(u,x+y)\Pi(\dd y)\I_{\{u\leq N \}}\I_{\{x>b(u) \}} 
&\geq - G(u,x)  \I_{\{u\leq N \}}\I_{\{x>b(u) \}} \nonumber\\
&\geq - u^{p-1}  \I_{\{u\leq N \}}\I_{\{x>b(u) \}} \nonumber\\
\label{eq:lowerboundforK3}
&\geq - N^{p-1}  \I_{\{u\leq N \}} .
\end{align}
Then, we have that for any $(u,x)\in E$, 
\begin{align}
\E_{u,x}&\left(\int_0^{\infty} \int_{(-\infty,0)} \widetilde{V}(U_s,X_s+y) \Pi(\dd y) \I_{\{ U_s \leq  N\}}\I_{\{X_s >  b(U_s) \}} \dd s  \right)\nonumber\\
& \geq - N^{p-1} \E_{u,x}\left(\int_0^{\infty}  \I_{\{ U_s \leq  N\}} \dd s  \right)\nonumber\\
&\geq - N^{p-1} [\E_x(g)+N]\nonumber\\
\label{eq:integrabilityoflwforK3}
&>-\infty,
\end{align}
where in the second last inequality we used the fact that $U_s>N$ for all $s\geq g+N$. Similarly,
\begin{align}
\int_{(-\infty,0)}  \widetilde{V}(u,x+y)\Pi(\dd y) \I_{\{u>N \}}\I_{\{x>b(N) \}}
&\geq - u^{p-1} \I_{\{b(N)\geq x>b(u) \}} \nonumber\\
\label{eq:lowerboundforK4}
&\geq - u^{p-1}\I_{\{x\leq b(N) \}}.
\end{align}
Thus, we can see that for any $(u,x)\in E$,
\begin{align}
\E_{u,x} &\left(\int_0^{\infty} \int_{(-\infty,0)} \widetilde{V}(U_s,X_s+y)  \Pi(\dd y) \I_{\{ U_s >  N\}} \I_{\{b(N) \geq X_s >  b(U_s) \}} \dd s   \right)\nonumber\\
&\geq - \E_{u,x}\left(\int_0^{g^{(b(N)) }}  (u+s)^{p-1}   \I_{\{ X_s \leq b(N) \}} \dd s   \right)\nonumber\\
&\geq - \frac{1}{p}\E_x((u+g^{(b(N))})^p ) \nonumber\\
\label{eq:integrabilityoflwforK4}
&> -\infty,
\end{align}
where we used that $U_s\leq u+s$ and that $g^{(b(N))}=\sup\{t\geq 0: X_t \leq b(N) \}$ has moments of order $p$ (see Lemma \ref{lemma:finitenessofEgn}). Moreover, since $V$ is non-decreasing in each argument we have that
\begin{align}
\label{eq:lowerboundforK5}
\int_{(-\infty,0)}  \widetilde{V}(u,x+y)\Pi(\dd y) \I_{\{u>N \}}\I_{\{x>b(N) \}} \geq \int_{(-\infty,0)}  \widetilde{V}(N,x+y)\Pi(\dd y) \I_{\{x>b(N) \}} .
\end{align}

Hence, since $b$ is non-decreasing and by Fubini's theorem deduce that

\begin{align*}
\E_{u,x}&\left(\int_0^{\infty} \int_{(-\infty,0)} \widetilde{V}(U_s,X_s+y) \Pi(\dd y) \I_{\{ U_s >  N\}} \I_{\{X_s >  b(N) \}}\dd s   \right)\\
& \geq \E_x\left(\int_0^{\infty} \int_{(-\infty,0)} \widetilde{V}(N,X_s+y) \Pi(\dd y)  \I_{\{X_s >  b(N) \}}\dd s   \right)\\
&= \int_{(b(N),\infty)} \int_{(-\infty,0)} \widetilde{V}(N,z+y)  \Pi(\dd y) \int_0^{\infty}\P_x(X_s \in \dd z) \dd s \\
&\geq  \Phi'(0)  \int_{b(N)}^{\infty} \int_{(-\infty,0)} \widetilde{V}(N,z+y)  \Pi(\dd y)\dd z,
\end{align*}
where in the last inequality we used a density of the $0$-potential measure of $X$ without killing (see (\ref{eq:qpotentialdensitywithoutkilling})), that $V\leq 0$ and that $W$ is a non-negative function that vanishes on $(-\infty,0)$. From Fubini's theorem and since $V$ is a non-decreasing function in each argument that vanishes on $D$, we obtain that 

\begin{align}
\E_x&\left(\int_0^{\infty} \int_{(-\infty,0)} \widetilde{V}(U_s,X_s+y) \Pi(\dd y) \I_{\{ U_s >  N\}} \I_{\{X_s >  b(N) \}}\dd s   \right)\nonumber\\
&\geq \Phi'(0)  \int_{b(N)}^{b(N)+1} \int_{(-\infty,0)} \widetilde{V}(N,z+y)  \Pi(\dd y)\dd z \nonumber\\
&\qquad+\Phi'(0) \int_{b(N)+1}^{\infty} \int_{(-\infty,0)} \widetilde{V}(N,z+y)  \Pi(\dd y) \dd z	\nonumber \\
& \geq   \Phi'(0)\int_{(-\infty,0)}  \widetilde{V}(N,b(N)+y)  \Pi(\dd y) \nonumber\\
&\qquad +\Phi'(0) \int_{(-\infty,-1)} \int_{b(N)+1}^{b(N)-y}  \widetilde{V}(N,z+y) \dd z \Pi(\dd y)	\nonumber \\
& \geq   \Phi'(0) \int_{(-\infty,0)} \widetilde{V}(N,b(N)+y)  \Pi(\dd y)  -\Phi'(0) \int_{(-\infty,-1)} (y+1)  \widetilde{V}(0,y)  \Pi(\dd y)	\nonumber \\
\label{eq:integrabilityoflwforK5}
& >-\infty,
\end{align}
where the finiteness of the last integrals follow from Lemmas \ref{lemma:finitenessofEgn} and \ref{lemma:auxiliaryfunctionLambda} and equation (\ref{eq:lowerboundforV}). Therefore, we have that \eqref{eq:integrabilityVPidy} follows from equations \eqref{eq:decompositionofK2onthreefunctions}-\eqref{eq:integrabilityoflwforK5}. Moreover, from the dominated convergence theorem we have that 
\begin{align*}
&\lim_{u,x\rightarrow \infty} \E_{u,x}\left(\int_0^{\infty} \int_{(-\infty,0)} \widetilde{V}(U_s,X_s+y)\Pi(\dd y) \I_{\{X_s >  b(U_s) \}}   \right)\\
&=\lim_{u,x\rightarrow \infty} \E\left(\int_0^{\infty} \int_{(-\infty,0)} \widetilde{V}(u+s,X_s+x+y)\Pi(\dd y) \I_{\{X_s+x >  b(u+s) \}} \I_{\{s<\sigma_{-x}^- \}}   \right)\\
&\qquad +\lim_{x\rightarrow \infty} \E\left(\int_0^{\infty} \int_{(-\infty,0)} \widetilde{V}(U_s^{(-x)},X_s+x+y)\Pi(\dd y) \I_{\{X_s+x >  b(U_s^{(-x)}) \}} \I_{\{s\geq \sigma_{-x}^- \}}   \right)\\
&\geq  \E\left(\int_0^{\infty} \int_{(-\infty,0)}\lim_{u,x\rightarrow \infty} \widetilde{V}(u+s,X_s+x+y)\Pi(\dd y) \I_{\{X_s+x >  b(u+s) \}}   \right)\\
&\qquad +\E\left(\int_0^{\infty} \int_{(-\infty,0)} \lim_{x\rightarrow \infty} \widetilde{V}(U_s^{(-x)},X_s+x+y)\Pi(\dd y) \I_{\{X_s+x >  b(U_s^{(-x)}) \}}    \right).
\end{align*}
Note that $b$ is a decreasing function and then $\lim_{u,x\rightarrow \infty} V(u,x)=0$ and $\lim_{x\rightarrow \infty} V(u,x)=0$ for any $u>0$. Moreover, for any $s\geq 0$, $x\mapsto U_s^{(-x)}$ is increasing and bounded so then $\lim_{x\rightarrow \infty}U_s^{(-x)} $ exists. Then we have that 
\begin{align*}
\lim_{u,x\rightarrow \infty} &\E_{u,x}\left(\int_0^{\infty} \int_{(-\infty,0)} \widetilde{V}(U_s,X_s+y)\Pi(\dd y) \I_{\{X_s >  b(U_s) \}}   \right)=0
\end{align*}
as claimed.
\end{proof}

The next two Lemmas are auxiliary results for the proof of Lemma \ref{lemma:proofofequationforV0}. For ease of notation, we define for any $(u,x)\in E$,
\begin{align*}
K_1(u,x)&:=G(u,x)\I_{\{x\leq b(u) \}},\\
K_2(u,x)&:= -\int_{(-\infty,0)} \widetilde{V}(u,x+y)\Pi(\dd y)\I_{\{x>b(u) \}}.
\end{align*}
The following result is an application of formula \eqref{eq:calculationofUsXsintegral} to the functions $K_1$ and $K_2$.
\begin{lemma}
\label{lemma:formulaforK1andK2}
Assume that $\sigma>0$. Then for any $\delta>0$ we have that 
\begin{align*}
\E\bigg( \int_0^{\infty} [K_1&(U_s+\delta,X_s)+ K_2(U_s+\delta,X_s)]\dd s \bigg)\\
&= \lim_{\varepsilon \downarrow 0}\frac{ \E_{\varepsilon} \left(\I_{\{\tau_0^-<\infty \}}K^-(\delta, X_{\tau_0^-}-\varepsilon ) \right) +K^+(\delta,\varepsilon) }{\psi'(0+)W(\varepsilon)}, 
\end{align*}
where for all $\delta >  0$ and $x\leq 0$,
\begin{align*}
K^-(\delta, x)=\E_x\left(\int_0^{\tau_0^+} [K_1(\delta,X_r)+K_2(\delta,X_r)] \dd r \right),
\end{align*}
and for all $\delta,x>0$,
\begin{align*}
K^+(\delta,x)=\E_x\left(\int_0^{\tau_0^-} [K_1(\delta+s,X_r)+K_2(\delta+s,X_r)] \dd r \right).
\end{align*}
\end{lemma}
 \begin{proof}
 For $N>0$ fixed and $(u,x)\in E$, we define the functions 
\begin{align*}
K_3(u,x)&:=-\int_{(-\infty,0)}  \widetilde{V}(u,x+y)\Pi(\dd y)\I_{\{u\leq N \}}\I_{\{x>b(u) \}}, \\
K_4(u,x)&:=-\int_{(-\infty,0)}  \widetilde{V}(u,x+y)\Pi(\dd y) \I_{\{u>N \}}\I_{\{x>b(N) \}},\\
K_5(u,x)&:=-\int_{(-\infty,0)}  \widetilde{V}(u,x+y)\Pi(\dd y) \I_{\{u>N \}}\I_{\{b(N)\geq x>b(u) \}}. 
\end{align*} 
Then, it follows that for any $\delta>0$,
\begin{align*} 
 \E\bigg( \int_0^{\infty} [K_1&(U_s+\delta,X_s)+ K_2(U_s+\delta,X_s)]\dd s \bigg)\\
&=  \E\left( \int_0^{\infty} K_1(U_s+\delta,X_s)\dd s \right)+ \sum_{i=3}^5 \E\left( \int_0^{\infty} K_i(U_s+\delta,X_s)\dd s \right).
\end{align*}
Since $G$, $V$ and $b$ are continuous functions and $b$ is non-increasing, it is easy to see that $K_i$ is left-continuous in each argument, for each $i\in \{1,2,3,4,5\}$. Moreover, since $|G(u,x)|\leq u^{p-1}+\E_x(g^{p-1})$ for all $x\in \R$ and $u\geq 0$, and $b$ is non-increasing, we have that $|K_1(u+\delta,x)|\leq C_1(u,x)$, where $C_1(u,x):=  (u+\delta)^{p-1}\I_{\{x<b(\delta) \}}+\E_{x}(g^{p-1})$. Note that for each $x\in \R$, $u\mapsto C_1(u,x)$ is monotone and, for each $(u,x)\in E$ and $y\in \R$,
\begin{align*}
\E_{u,x}&\left( \int_0^{\infty}  C_1(U_r,X_r+y) \dd r\right)\\
&\leq \E_{x+y-b(\delta)}\left( \int_0^{\infty}  (u+r+\delta)^{p-1}\I_{\{X_r<0 \}} \dd r\right)+\E_{x+y}\left( \int_0^{\infty}  \E_{X_r}(g^{p-1}) \dd r\right)\\
&= \E_{x+y-b(\delta)}\left( \int_0^{g}  (u+r+\delta)^{p-1} \dd r\right)+\E_{x+y}\left( \int_0^{\infty}  \E_{X_r}(g^{p-1}) \dd r\right)\\
&\leq \E_{x+y-b(\delta)} \left((u+g+\delta)^{p}\right) +\E_{x+y}\left( \int_0^{\infty}  \E_{X_r}(g^{p-1}) \dd r\right)\\
&<\infty,
\end{align*}
where in the first inequality we used that $U_r\leq u+r$ for all $r\geq 0$ (under the measure $\P_{u,x}$) and the last inequality follows from Lemmas \ref{lemma:finitenessofEgn} and \ref{lemma:finitenessofintExgn}. Similarly, from equations \eqref{eq:lowerboundforK3}-\eqref{eq:integrabilityoflwforK5} we see that for each $i\in \{3,4,5\}$, there exists a non-negative function $C_i:\R_+ \times \R \mapsto \R$ such that $u\mapsto C_i(u,x)$ is a monotone function for all $x\in \R$, $|K_i(u+\delta,x)|\leq C_i(u,x)$ and $\E_{u,x}\left( \int_0^{\infty}  C_i(U_r,X_r+y) \dd r\right)<\infty$ for all $(u,x)\in E$ and $y\in \R$. \\
Hence, using formula (\ref{eq:calculationofUsXsintegral}), applied to the functions $K_i$, for $i \in \{1,3,4,5\}$, we get that 
\begin{align*}
\E\bigg( \int_0^{\infty} [K_1&(U_s+\delta,X_s)+ K_2(U_s+\delta,X_s)]\dd s \bigg)\\
&= \lim_{\varepsilon \downarrow 0}\frac{ \E_{\varepsilon} \left(\I_{\{\tau_0^-<\infty \}}K^-(\delta, X_{\tau_0^-}-\varepsilon ) \right) +K^+(\delta,\varepsilon) }{\psi'(0+)W(\varepsilon)}, 
\end{align*}
as claimed.
\end{proof}
The following result is also used in the proof of Lemma \ref{lemma:proofofequationforV0}. 

\begin{lemma}
\label{lemma:convergencederivativeofVXtaue}
Suppose that $\sigma>0$. Then we have that 
\begin{align*}
\lim_{\varepsilon \downarrow 0}\frac{\E(V(0,X_{\tau_{-\varepsilon}^-}+\varepsilon) \I_{\{\tau_{-\varepsilon}^-<\infty\}})-\E( V(0,X_{\tau_{-\varepsilon}^-})\I_{\{\tau_{-\varepsilon}^-<\infty\}})}{\varepsilon}=\frac{\partial}{\partial x} V_-(0,0).
\end{align*}
\end{lemma}

\begin{proof}
Recall from Lemma \ref{lemma:diferrentiabilityofV(0,x)} that $V(0,x)$ is  continuously differentiable on $(-\infty,0)$ with 
\begin{align*}
\frac{\partial}{\partial x} V(0,x)= \int_{0}^{\infty}  \E_{x-u}(g^{p-1})W'(u) \dd u
\end{align*} 
for $x<0$. Moreover, from the continuity of $x\mapsto \int_{0}^{\infty}  \E_{x-u}(g^{p-1})W'(u) \dd u$, we see that
\begin{align}
\label{eq:limittoleftderivative}
\frac{\partial}{\partial x} V_-(0,0)= \int_{0}^{\infty}  \E_{u}(g^{p-1})W'(u) \dd u=\lim_{x\uparrow 0}\frac{\partial}{\partial x} V(0,x).
\end{align} 
Hence, from \eqref{eq:expressionforV0xnegative} we have that for any $\varepsilon>0$ and $x\leq -\varepsilon$, 
\begin{align*}
V(0,x+\varepsilon)-V(0,x)&= \int_{-x-\varepsilon}^{-x} \int_{0}^{\infty}  \E_{-u-z}(g^{p-1})W'(u)\dd u  \dd z \\
&= \int_{-\varepsilon}^{0} \int_{0}^{\infty}  \E_{-u-z+x}(g^{p-1})W'(u)\dd u  \dd z \\
&= \int_{-\varepsilon}^{0}  \frac{\partial}{\partial x} V(0,x-z)\dd z.
\end{align*}
Thus, since $\frac{\partial}{\partial x} V(0,x)$ is non-increasing (see Lemma \ref{lemma:diferrentiabilityofV(0,x)}) we see that
\begin{align*}
\varepsilon \frac{\partial}{\partial x} V(0,x+\varepsilon)\leq V(0,x+\varepsilon)-V(0,x)\leq \varepsilon \frac{\partial}{\partial x} V(0,x)
\end{align*}
Hence, for any $\varepsilon>0$ we obtain 
\begin{align} 
\E&\left(\frac{\partial}{\partial x} V(0,X_{\tau_{-\varepsilon}^-}+\varepsilon)\I_{\{\tau_{-\varepsilon}^-<\infty\}} \right) \nonumber\\
& \qquad\leq 
\frac{   \E(V(0,X_{\tau_{-\varepsilon}^-}+\varepsilon) \I_{\{\tau_{-\varepsilon}^-<\infty\}})-\E( V(0,X_{\tau_{-\varepsilon}^-})\I_{\{\tau_{-\varepsilon}^-<\infty\}})}{\varepsilon}\nonumber\\
\label{eq:boundsforV(x+e)-V(x)/e}
&\qquad\qquad\leq \E\left(\frac{\partial}{\partial x} V(0,X_{\tau_{-\varepsilon}^-})\I_{\{\tau_{-\varepsilon}^-<\infty\}} \right).
\end{align}
Note that, under the event $\{ \tau_{-\varepsilon}^-<\infty\}$, we have $|X_{\tau_{-\varepsilon}^-}|\leq |\underline{X}_{\infty}|$. Then, from \eqref{eq:boundaryforderivativeofVbelowzero} we deduce that 
\begin{align*}
0\leq \frac{\partial}{\partial x} V(0,X_{\tau_{-\varepsilon}^-}+\varepsilon)\I_{\{\tau_{-\varepsilon}^-<\infty\}} \leq \frac{\partial}{\partial x} V(0,X_{\tau_{-\varepsilon}^-})\I_{\{\tau_{-\varepsilon}^-<\infty\}}\leq  \alpha_{p-1}+\gamma_{p-1} |\underline{X}_{\infty}|^{p-1}.
\end{align*}
Moreover, since $\E(|\underline{X}_{\infty}|^{p-1})<\infty$ (see Lemma \ref{lemma:finitenessofEgn}) we see that the random variables above are bounded by an integrable random variable. Furthermore, it is easy to show that $ \tau_{-\varepsilon}^- \downarrow \tau_0^-=0$ a.s., when $\varepsilon\downarrow 0$ (where the equality follows since $X$ is of infinite variation), and, since $X$ is right-continuous, we have that $\lim_{\varepsilon \downarrow 0}X_{\tau_{-\varepsilon}^-}=X_{0}=0$ almost surely. Therefore, by letting $\varepsilon \downarrow 0$ in \eqref{eq:boundsforV(x+e)-V(x)/e}, the dominated convergence theorem and \eqref{eq:limittoleftderivative}, we have that 
\begin{align*}
\lim_{\varepsilon \downarrow 0}\frac{\E(V(0,X_{\tau_{-\varepsilon}^-}+\varepsilon) \I_{\{\tau_{-\varepsilon}^-<\infty\}})-\E( V(0,X_{\tau_{-\varepsilon}^-})\I_{\{\tau_{-\varepsilon}^-<\infty\}})}{\varepsilon}=\frac{\partial}{\partial x} V_-(0,0)
\end{align*}
as claimed.
 
\end{proof}

\section{Variational inequalities for spectrally negative L\'evy processes}

\label{appendix:variationaliniequalities}
In \cite{lamberton2008critical} (see Section 2), variational inequalities in the sense of distributions are studied for general L\'evy processes, and such results are applied to characterise the price of an American option. It turns out that, although there are many similarities to our setting, their assumptions do not entirely fit our optimal stopping problem (cf. Proposition 2.5 in \cite{lamberton2008critical}), and their results are not directly applicable. This section is dedicated to relaxing the assumptions on Proposition 2.5 in \cite{lamberton2008critical} imposed to the value function. This extension is natural and most of the proofs remain the same, but for completeness, some of them are included in this section.\\

Let $X$ be a spectrally negative L\'evy process with the following representation:

\begin{align*}
X_t=-\mu t+\sigma B_t+\int_0^{t}\int_{(-\infty,-1)} x N(\dd s,\dd x)+\int_0^t\int_{(-1,0)} x(N(\dd s,\dd x)-\dd s\Pi(\dd x)) ,
\end{align*}
where $\mu  \in \R$, $\sigma\geq 0$, $\{ B_t, t\geq 0\}$ is a standard Brownian motion and $N$ is a Poisson random measure on $\R_+ \times \R \{0\}$ with intensity $\dd t \times \Pi(\dd y)$, where $\Pi$ is a L\'evy measure, i.e., $\Pi$ satisfies $\int_{\R} (1 \wedge |x|^2 )\Pi(\dd x)<\infty$.\\

Fix $f \in C^{1,2}_b (\R_+\times \R)$, the set of all bounded $C^{1,2}(\R_+\times	\R)$ functions with bounded derivatives. By applying It\^o formula we obtain the following decomposition 
\begin{align*}
f(t,X_t)=f(0,X_0)+M_t+\int_0^t \mathcal{A}_{(t,X)} (f)(s,X_s)\dd s,\qquad t\geq 0,
\end{align*}
where $M$ is a martingale starting at zero and $\mathcal{A}_{(t,X)}(f)$ is the infinitesimal generator of $(t,X)$, applied to $f$, given by 
\begin{align*}
\mathcal{A}_{(t,X)}(f)(t,x)
&=\frac{\partial }{\partial	t} f(t,x)  -\mu \frac{\partial }{\partial	x} f(t,x) +\frac{1}{2} \sigma^2  \frac{\partial^2}{\partial x^2} f(t,x) +B_X(f)(t,x),
\end{align*}
with
\begin{align*}
B_X(f)(t,x)=\int_{(-\infty,0)} \left( f(t,x+y) -f(t,x)-y\I_{\{y>-1\}}\frac{\partial }{\partial x} f(t,x) \right)\Pi(\dd y).
\end{align*}
Note that for the derivatives in the operator $\mathcal{A}_{(t,X)}$ to be defined it is only needed that $f\in C^{1,2}(\R_+\times \R$). In the next Lemma, we show that $B_X$ can be defined in a subset $B\subset \R_+ \times \R$ provided some conditions are met in the set $B$. 

\begin{lemma}
\label{lemma:jumpspartofgeneratorofX}
Let $B\subset \R_+ \times \R$ an open set. Assume that $f$ is a $C^{1,2}(\R_+\times \R)$ function and that 
\begin{align*}
 \int_{(-\infty,-1)} |f(t,x+y)|\Pi(\dd y)<\infty,
\end{align*}
for all $(x,y)\in B$. We have that $|B_X(f)(t,x)|<\infty$ for all $(t,x)\in B$. Moreover if $f$, its derivatives and $(t,x)\mapsto \int_{(-\infty,-1)} |f(t,x+y)|\Pi(\dd y)$ are bounded functions in $B$, we have that, for any $T>0$, $B_X(f)$ is bounded in $B\cap([0,T]\times \R)$ and continuous in $B$.
\end{lemma}

\begin{proof}
Take $(t,x)\in B$. By Taylor's theorem we know that for each $y\in (-1,0)$, there exists $c_y\in [x+y,x]\subset [x-1,x]$ such that 
\begin{align*}
f(t,x+y) -f(t,x)-y \frac{\partial }{\partial x} f(t,x)=y^2 \frac{1}{2} \frac{\partial^2}{\partial x^2} f(t,c_y).
\end{align*}
Hence, we have that for any $(t,x)\in B$, 
\begin{align*}
|B_X(f)(t,x)|&=\left|\int_{(-1,0)} \left( f(t,x+y) -f(t,x)-y\frac{\partial }{\partial x} f(t,x) \right)\Pi(\dd y)\right.\\
&\qquad+\left.\int_{(-\infty,-1]} \left( f(t,x+y) -f(t,x)\right)\Pi(\dd y)\right|\\
&\leq \int_{(-1,0)} y^2 \frac{1}{2} \left|\frac{\partial^2}{\partial x^2} f(t,c_y)\right|\Pi(\dd y)+\int_{(-\infty,-1]} \left| f(t,x+y) -f(t,x) \right|\Pi(\dd y)\\
&\leq  \sup_{z\in [x-1,x]}\frac{1}{2} \left|\frac{\partial^2}{\partial x^2} f(t,z)\right|   \int_{(-1,0)} y^2 \Pi(\dd y)+\int_{(-\infty,-1]} \left| f(t,x+y)\right| \Pi(\dd y)\\
&\qquad +|f(t,x)|\Pi(-\infty,-1]\\
&<\infty,
\end{align*}
where we used that $\Pi$ is finite on any set away from zero and that the derivatives of $f$ are continuous on $B$ (then bounded on compact sets). The second assertion follows by the fact that the second derivative is continuous and bounded on the compact set containing the set $\tilde{B}=\{(t,x-y)\in [0,T]\times \R: (t,x)\in B, \text{ }y\in (0,1)\text{ and } (t,x-y)\notin B \}$, and since $f$ and $(t,x)\mapsto \int_{(-\infty,-1)} |f(t,x+y)|\Pi(\dd y)$ are bounded in $B$. The continuity of $B_X(f)$ in $B$ follows from the fact that $f$ is $C^{1,2}$ and the dominated convergence theorem.
\end{proof}

Consider the stopping time $\tau_B$ as the first time the process $(t,X)$ exits the open set $B$, i.e.,
\begin{align*}
\tau_B^{(s,x)}=\inf \{t\geq 0: (s+t,X_t+x) \notin B \}.
\end{align*}

\begin{lemma}
\label{lemma:martingaledecompositionofftX}
Let $B\subset \R_+ \times \R$ an open set. Assume that $f$ is a $C^{1,2}(\R_+\times \R)$ function such that $f$, its derivatives and $(t,x)\mapsto \int_{(-\infty,-1)} |f(t,x+y)|\Pi(\dd y)$ are bounded in $B$. Then, for any $t\geq 0$, we have the following decomposition
\begin{align}
\label{eq:semimtgdecomp}
f(s+t\wedge \tau_B^{(s,x)},X_{t\wedge \tau_B^{(s,x)}}+x)=f(s,x)+M_t+\int_0^{t\wedge \tau_B^{(s,x)} }\mathcal{A}_{(t,X)}(f)(u+r,X_r+x)\dd r,
\end{align}
where $\{M_t, t\geq 0 \}$ is a zero mean $\P$-martingale.
\end{lemma}

\begin{proof}
Let $(s,x)\in B$ and $t\geq 0$. Since $f$ is a $C^{1,2}(\R_+\times \R)$ function we have by It\^o formula,
\begin{align*}
f(s+t,X_{t}+x)-f(s,x)
&=M_t^{(1)}+M_t^{(2)}+\int_0^t \mathcal{A}_{(t,X)}(f)(s+r,X_r+x)\dd r,
\end{align*}
where 
\begin{align*}
M_t^{(1)}&=\sigma \int_0^{t } \frac{\partial f}{\partial x} (s+r,X_{r-}+x) \dd B_r+\int_0^{t } \int_{(-1,0)} y \frac{\partial f}{\partial x} (s+r,X_{r-}+x)   \widetilde{N}(\dd r, \dd x),\\
M_t^{(2)}&=\int_0^t \int_{(-\infty,0)}  \bigg[ f(s+r,X_{r-}+x+y)-f(s+r,X_{r-}+x)\\
&\qquad \qquad\qquad -y \I_{\{y>-1 \}} \frac{\partial f}{\partial x} (s+r,X_{r-}+x)  \bigg] \widetilde{N}(\dd r, \dd y).
\end{align*}
Note that for any $r<\tau_B^{(s,x)}$ we have $(s+r,X_r+x)\in B$. Hence, since $\frac{\partial f}{ \partial x} $ is bounded in the set $B$, we have that the stopped process $\{ M^{(1)}_{t\wedge \tau_B^{(s,x)}}, t \geq 0\}$ is a martingale. Moreover, from Lemma \ref{lemma:jumpspartofgeneratorofX} we have that $B_X(f)$ is a bounded function on $B\cap ([0,t]\times \R)$. Hence, we have that
\begin{align*}
\E\left( \int_0^{t\wedge \tau_B^{(s,x)}} B_X(f)(s+r,X_r+x) \dd r \right)<\infty
\end{align*}
for all $t\geq 0$. Thus, the process $\{ M^{(2)}_{t\wedge \tau_B^{(s,x)}}, t \geq 0\}$ is also a martingale.
\end{proof}

Let $G$ be a continuous function. Define the process $Z^{(s,x)}=\{Z_t^{(s,x)},t\geq 0\}$, where 
\begin{align*}
Z_t^{(s,x)}=f(s+t,X_{t}+x)+\int_0^{t } G(r+s,X_r+x)\dd r, \qquad t\geq 0.
\end{align*}
We have the following proposition.

\begin{prop}
\label{prop:submtgcharactforc2}
Let $B\subset \R_+ \times \R$ an open set. Assume that $f$ is a $C^{1,2}(\R_+\times \R)$ function such that $f$, its derivatives and $(t,x)\mapsto \int_{(-\infty,-1)} |f(t,x+y)|\Pi(\dd y)$ are bounded in $B$, and $G:\R_+\times \R \mapsto \R$ is a continuous function bounded in $B$. Then, for all $(s,x)\in B$, the process $\{ Z_{t\wedge \tau_B^{(s,x)}}^{(s,x)}, t\geq 0\}$ is a submartingale if and only if $\mathcal{A}_{(t,X)}(f)+G \geq 0$ in $B$.
\end{prop}

\begin{proof}
Fix $(s,x) \in B$. Suppose that $\{ Z_{t\wedge \tau_B^{(s,x)}}^{(s,x)},t\geq 0\}$ is a submartingale. We prove that $\mathcal{A}_{(t,X)}(f)(s,x)+G(s,x)\geq 0$. From the submartingale property, we have that for every $t\geq 0$,

\begin{align*}
\E\left[ \frac{1}{t}(Z_{t\wedge \tau_B^{(s,x)}}^{(s,x)}-Z_0^{(s,x)}) \right]\geq 0
\end{align*}
which implies that 
\begin{align*}
\E&\left[ \frac{1}{t}[f(s+t\wedge \tau_B^{(s,x)},X_{t\wedge \tau_B^{(s,x)}}+x)-f(s,x)] \right]\\
&\qquad+ \E\left[\frac{1}{t}\int_0^{t \wedge \tau_B^{(s,x)}} G(s+r,X_r+x)\dd r \right]\geq 0.
\end{align*}
By the decomposition \eqref{eq:semimtgdecomp} we then get,
\begin{align*}
\E\left[ \frac{1}{t} \int_0^{t\wedge \tau_B^{(s,x)}} \mathcal{A}_{(t,X)} (f)(s+r,X_r+x)\dd s \right]+ \E\left[ \frac{1}{t}\int_0^{t \wedge \tau_B^{(s,x)}} G(s+r,X_r+x) \dd r \right] \geq 0.
\end{align*}

Due to the right continuity of $(t,X)$ and since $B$ is open, we have $\tau_B^{(s,x)}>0$ almost surely. Therefore, tending $t$ to zero in the above inequality, by Fubini's theorem and the fundamental theorem of calculus (since $r\mapsto X_r$ is right continuous) we deduce that
\begin{align*}
\mathcal{A}_{(t,X)}(f)(s,x)+G(s,x)\geq 0.
\end{align*}
Next, suppose that $\mathcal{A}_{(t,X)}(f)(s,x)+G(s,x) \geq 0$ for all $(s,x)\in B$. We show that, for any $(s,x)\in B$, the process $\{ Z_{t\wedge \tau_B^{(s,x)}}^{(s,x)},t\geq 0 \}$ is a submartingale. For any $(s,x) \in B$ and $0\leq r\leq t$, we have that

\begin{align*}
\E(Z_{t\wedge\tau_B^{(s,x)}}^{(s,x)}|\F_r)&=\E\left[ f(s,x)+M_{t\wedge \tau_B^{(s,x)}}+ \int_0^{t\wedge \tau_B^{(s,x)}} \mathcal{A}_{(t,X)} (f)(v+s,X_v+x)\dd v \bigg| \F_r \right]\\
&\qquad + \E\left[\int_0^{t \wedge \tau_B^{(s,x)}} G(v+s,X_v+x)\dd v\bigg| \F_r\right]\\
&=Z_{r\wedge \tau_B^{(s,x)}}^{(s,x)}+\E\left[ \int_{r\wedge \tau_B^{(s,x)}}^{t\wedge \tau_B^{(s,x)}} \mathcal{A}_{(t,X)} (f)(v+s,X_v+x)\dd v\bigg| \F_r\right]\\
&\qquad+ \E\left[ \int_{r\wedge \tau_B^{(s,x)}}^{t\wedge \tau_B^{(s,x)}}G(v+s,X_v+x)\dd v \bigg| \F_r\right]\\
& \geq Z_{s\wedge \tau_B^{(s,x)}}^{(s,x)},
\end{align*}
where the last inequality follows from the fact that $(v+s,X_v+x) \in B$ for all $v\in (r \wedge \tau_B^{(s,x)}, t\wedge \tau_B^{(s,x)})$ and that $\mathcal{A}_{(t,X)}(f)(s,x)+G (s,x)\geq 0$ in $B$. Therefore the process $Z_{t\wedge \tau_B}^{(s,x)}$ is a submartingale.
\end{proof}

It turns out that the above proposition can be extended to a more general class of functions, provided that the inequality $\mathcal{A}_{(t,X)}(f)+G\geq 0$ is taken in the sense of distributions. For this, we recall some facts and notation from the theory of distributions (see e.g. \cite{friedlander1998introduction} or \cite{rudin1991functional} for further details). We introduce the multi-index notation. A multi-index is a $n$-tuplet $\alpha=(\alpha_1,\ldots,\alpha_n)$ of non-negative integers with order $|\alpha|=\alpha_1+\cdots+\alpha_n$. We set the notation
\begin{align*}
\partial^{\alpha} f= \frac{\partial^{|\alpha|}}{\partial x_1^{\alpha_1} \cdots \partial x_n^{\alpha_n}}.
\end{align*}
If $\mathcal{O}$ is an open subset of $\R^d$, we denote by $\mathcal{D(O)}$ the set of test functions in $\mathcal{O}$, i.e., the set of all $C^{\infty}$ functions with compact support in $\mathcal{O}$, and by $\mathcal{D'(O)}$ the space of distributions on $\mathcal{O}$. That is, $\mathcal{D'(O)}$ is the space of linear forms $u$ in $\mathcal{D(O)}$ such that, for every compact set $K\subset O$, there is a real number $C\geq 0$ and a non-negative integer $N$ that satisfy 
\begin{align*}
|\langle u,\psi \rangle| \leq C \sum_{|\alpha|\leq N} \sup |\partial^{\alpha} \psi|
\end{align*}
for all $\psi \in \mathcal{D(O)}$, where $\langle u, \varphi\rangle$ denotes the evaluation of the distribution $u$ on the test function $\varphi$. Inspired by the integration by parts formula, the derivative of the distribution $u$ is defined by
\begin{align*}
\langle \partial ^{\alpha}u,\varphi \rangle =(-1)^{|\alpha|}  \langle  u, \partial^{\alpha} \varphi \rangle	 , \qquad \varphi \in \mathcal{D}(\mathcal{O}).
\end{align*}
If $u$ is a locally integrable function on $\mathcal{O}$ ($u$ is a measurable function and $\int_{K}|u(x)|\dd x<\infty$ for any compact set $K\subset \mathcal{O}$), we can define the distribution 
\begin{align*}
\langle u, \varphi \rangle = \int u(x) \varphi(x) \dd x, \qquad \varphi \in \mathcal{D}(\mathcal{O}),
\end{align*}
which is usually identified only with the function $u$. Hence, if $g$ is a locally integrable function on $(0,\infty)\times \R$, the differential operator, $\mathcal{A}_{(t,X)}^0(g)$, given by

\begin{align*}
\mathcal{A}_{(t,X)}^0(g)(u,x):= \frac{\partial }{\partial	t} g(t,x)  -\mu \frac{\partial }{\partial	x} g(t,x) +\frac{1}{2} \sigma^2  \frac{\partial^2}{\partial x^2} g(t,x) 
\end{align*}
can be defined in the sense of distributions. Indeed, for any test function $\varphi$ with compact support in $O\subset \R_+ \times \R$, we define

\begin{align}
\label{eq:generatorindistributionderivativspart}
\langle \mathcal{A}_{(t,X)}^0(g), \varphi \rangle:=\int_{\R_+} \int_{\R} g(t,x)\left[- \frac{\partial}{\partial t } \varphi(t,x)  +\mu \frac{\partial }{\partial x} \varphi(t,x)+\frac{1}{2}\sigma^2 \frac{\partial^2 }{\partial x^2} \varphi(t,x) \right] \dd x \dd t.
\end{align}
Moreover, \cite{lamberton2008critical} showed (see Proposition 2.1) that the operator defined by 
\begin{align*}
B_X(g)(t,x):=\int_{(-\infty,0)} (g(t,x+y)-g(t,x)-y\frac{ \partial}{\partial x} g(t,x)\I_{\{y>-1 \}})\Pi(\dd y)
\end{align*}
can be also defined in the sense of distributions, when $g$ is a bounded Borel measurable function. For $\varphi \in \mathcal{D}((0,\infty) \times \R)$, consider the operator $B^*_X$ given by
\begin{align}
\label{eq:definitionofBXstar}
B_X^*(\varphi)(t,x)=\int_{(-\infty,0)} [\varphi(t,x-y)-\varphi(t,x)+y \frac{\partial}{\partial x}\varphi(t,x) \I_{\{ y>- 1\}}]\Pi(\dd y),
\end{align}
for any $ (t,x)\in (0,\infty)\times \R$. From Proposition 2.1 in \cite{lamberton2008critical}, we know that $B^*_X(\varphi)$ is continuous and integrable on $(0,\infty)\times \R$ and that the operator
\begin{align}
\label{eq:jumpspartgeneratorindistrb}
\langle B_{X}(g), \varphi\rangle= \int_{\R_+}\int_{\R} g(u,x) B^*_X(\varphi)(u,x)\dd x \dd u
\end{align}
defines a distribution. The following lemma shows that the boundedness condition imposed on $g$ can be relaxed.

\begin{lemma}
\label{lemma:distributionofthejumpspart}
Let $g$ be a locally integrable function in $\R_+ \times \R$ such that 
\begin{align}
\label{eq:integconditionforjumpsdistribution}
(u,x) \mapsto \int_{(-\infty,-1)} |g(u,x+y)|\Pi(\dd y)
\end{align} 
is locally integrable. The linear operator $B_{X}(g)$ defined in \eqref{eq:jumpspartgeneratorindistrb} defines a distribution on any open set $ O\subset \R_+ \times \R$. Moreover, if in addition we suppose that $g$ is a $C^{1,2}(\R_+\times \R)$ function we have that 
\begin{align*}
\int_{\R_+}\int_{\R} B_X(g)(t,x)\varphi(t,x)\dd x \dd t=\int_{\R_+}\int_{\R} g(t,x) B_X^*(\varphi)(t,x)\dd x \dd t
\end{align*}
for any $\varphi \in \mathcal{D}(\mathcal{\R_+\times \R})$.
\end{lemma}

\begin{proof}
It is clear that the operator defined in \eqref{eq:jumpspartgeneratorindistrb} is linear. Take a test function $\varphi$ with support in a compact set $H \times K   \subset \R_+ \times \R$. We have

\begin{align*}
|\langle B_{X}(g), \varphi\rangle| &\leq \int_{\R_+} \int_{\R}| g(u,x)|| B^*_X(\varphi)(u,x)|\dd x \dd u\\
& \leq \int_{H} \int_{\R}| g(u,x)| \int_{(-1,0)}  |\varphi(u,x-y)-\varphi(u,x)+y \frac{\partial }{\partial x} \varphi(u,x)| \Pi(\dd y)     \dd x \dd u\\
& \qquad+ \int_{H}\int_{\R}| g(u,x)| \int_{(-\infty,-1)}  |\varphi(u,x-y)-\varphi(u,x)| \Pi(\dd y)     \dd x \dd u.
\end{align*}
Note that if $ x\notin K+(-1,0]$ we have that $x\notin K$ (if we assume that $x\in K$ then $x=x+0 \in K+(-1,0]$, which is a contradiction) and $x-y \notin K$ for all $y\in (-1,0)$ (if $z=x-y \in K$, then $x=z+y \in K+(-1,0)\subset K+(-1,0] $ and we have got a contradiction), which implies $\varphi(u,x-y)-\varphi(u,x)+y \frac{\partial }{\partial x} \varphi(u,x)=0$. Denote $k_*= \inf K$, since $x\mapsto \varphi(u,x)$ has support in $K$ and using Taylor's formula we obtain 
\begin{align*}
|&\langle B_{X}(g), \varphi\rangle| \\
& \leq \int_H \int_{K+(-1,0]}| g(u,x)| \int_{(-1,0)}  |\varphi(u,x-y)-\varphi(u,x)+y \frac{\partial }{\partial x} \varphi(u,x)| \Pi(\dd y)     \dd x \dd u\\
& \qquad+ \int_H \int_{K} |g(u,x)| \int_{(-\infty,-1)}  |\varphi(u,x-y)-\varphi(u,x)| \Pi(\dd y)     \dd x \dd u\\
& \qquad+ \int_H \int_{-\infty}^{k_*}| g(u,x)| \int_{(-\infty,-1)}  |\varphi(u,x-y)| \Pi(\dd y)     \dd x \dd u\\
& \leq  \frac{1}{2} \sup | \frac{\partial^2}{\partial x^2}\varphi| \int_{(-1,0)} y^2 \Pi(\dd y)   \int_H \int_{K+(-1,0]}| g(u,x)|    \dd x \dd u  \\
& \qquad+ 2\sup |\varphi| \Pi((-\infty,-1)) \int_H \int_{K}| g(u,x)| \dd x \dd u\\
& \qquad+  \sup |\varphi| \int_H \int_{K} \int_{(-\infty,-1)}  | g(u,x+y)| \Pi(\dd y)     \dd x \dd u,
\end{align*}
which proves the assertion, since $\Pi$ is a L\'evy measure and $(u,x)\mapsto \int_{(-\infty,-1)} |g(u,x+y)|\Pi(\dd y)$ is locally integrable by assumption. The last assertion follows the same argument in \cite{lamberton2008critical} (see Proposition 2.1), so the proof is omitted.
\end{proof}
Therefore, if $g$ is a locally integrable function in $\R_+ \times \R$ such that the function defined in  \eqref{eq:integconditionforjumpsdistribution} is locally integrable, we can define the distribution $\mathcal{A}_{(t,X)}(g)= \mathcal{A}_{(t,X)}^0(g)+B_{X}(g)$ in any open set $ \mathcal{O}\subset \R_+ \times \R$.\\

Let $u$ be a distribution and $\theta \in \mathcal{D}(\R_+\times \R)$. Then the function 

\begin{align*}
(\theta*u)(t,x)=\langle u(s,y), \theta(t-s,x-y) \rangle
\end{align*}
is a member of $C^{\infty}(\R_+\times \R)$ and defines a distribution given by 

\begin{align*}
\langle \theta *u, \phi \rangle=\int_{\R_+}\int_{\R} \langle u(s,y), \theta(t-s,x-y) \rangle \phi(t,x)\dd x \dd t,
\end{align*}
for any $\phi\in \mathcal{D}(\R_+\times\R)$. It turns out that Proposition 2.3 in \cite{lamberton2008critical} can also be extended to this case. The proof remains the same so it is omitted.

\begin{prop}
\label{prop:convolutionpropertyofdistributions}
Let $g$ be a Borel and locally integrable function in $\R_+ \times \R$ such that the function $\int_{(-\infty,-1)} |g(u,x+y)|\Pi(\dd y)$ is locally integrable. We have that for every $\theta$ and $\varphi$ in $ \mathcal{D}(\R_+\times \R)$,
\begin{align*}
\langle \mathcal{A}_{(t,X)}(g*\theta), \varphi \rangle=\langle \mathcal{A}_{(t,X)}(g), \varphi	* \check{\theta}\rangle  
&=\langle\mathcal{A}_{(t,X)}(g)*\theta, \varphi  \rangle,
\end{align*}
where $\check{\theta}(u,x)=\theta(-u,-x)$ for any $(u,x)\in \R_+\times \R$.
\end{prop}
Let $u$ be a distribution in $\mathcal{O}$, we say that $u$ is non-negative if for any non-negative test function $\varphi\in \mathcal{D}(\mathcal{O})$, 
\begin{align*}
\langle u , \varphi \rangle \geq 0.
\end{align*}
The next result is an extension of Proposition 2.5 in \cite{lamberton2008critical}. The proof is essentially the same but we include it for completeness. 
%
%
\begin{prop}
\label{prop:variationalinequality}
Let $B$ be an open set in $\R_+\times \R$. Suppose that $f:\R_+\times \R\mapsto \R$, $G:\R_+\times \R\mapsto \R$  and $(u,x) \mapsto \int_{(-\infty,-1)} |f(u,x+y)|\Pi(\dd y)$ are locally integrable functions in $\R_+\times \R$ and bounded in $B$. Assume that the process $\{ Z_{t\wedge \tau_B^{(s,x)}}^{(s,x)},t\geq 0\}$ is a submartingale for every $(s,x)\in B$, where $Z_t^{(s,x)}=f(s+t,X_t+x)+\int_0^t G(s+r,X_r+x)\dd r$ and $ \tau_B^{(s,x)}=\inf\{t\geq 0: (t+s,X_t+x) \notin B \}$. Then, $\mathcal{A}_{(t,X)}(f)+G$ is a non-negative distribution on $B$. 
\end{prop}
 
\begin{proof}
Take $z_0=(u_0,x_0)\in B$ and choose $a>0$ such that $\textbf{B}(z_0,2a)\subset B$, where $\textbf{B}(z_0,2a)$ is the open ball with center $z_0$ and radius $2a$. We define the stopping time 
\begin{align*}
\tau_B=\inf\{t\geq 0: \text{ there exists } z\in \textbf{B}(z_0,a) \text{ such that } z+(t,X_t)\notin B \}.
\end{align*}
Note that for every $(u,x) \in \textbf{B}(z_0,a/2)$ and $(v,y)\in \textbf{B}(0,a/2)$ we have that $(u-v,x-y)\in \textbf{B}(z_0,a)\subset B$ and then $\tau_B\leq \tau_B^{(u-v,x-y)}$. Hence, the process $\{ Z_{t\wedge \tau_B}^{(u-v,x-y)},t\geq 0\}$ is a submartingale, and then, for any $t\geq 0$,
\begin{align*}
f(u-v&,x-y)\\
&\leq \E\left( f(u-v+t\wedge \tau_B,X_{t\wedge \tau_B}+x-y)+\int_0^{t\wedge \tau_B} G(u-v+r,X_r+x-y)\dd r \right) .
\end{align*}
Next, we consider a sequence of even non-negative functions $\{ \rho_n,n\geq 1 \}$ in $C^{\infty}$ such that, for each $n\geq 1$, the support of $\rho_n$ is in $\textbf{B}(0,a/(2n))$ and $\int_{\R^2}\rho_n(v,y)\dd v \dd y=1$. Then, by integrating the equation above with respect to $\rho_n(v,y)$ and Fubini's theorem, we get that 
\begin{align}
(f*\rho_n)&(u,x) \nonumber\\
\label{eq:submartingfstarrhoninequality}
&\leq \E\left( (f*\rho_n)(u+t\wedge \tau_B,X_{t\wedge \tau_B}+x)\right)+\E\left(\int_0^{t\wedge \tau_B} (G*\rho_n)(u+r,X_r+x)\dd r \right).
\end{align}
Fix $(u,x) \in \textbf{B}(z_0,a/2)$. Note that, since $f$ is bounded, we have that for all $n\geq 1$, the function $(s,w)\mapsto	f*\rho_n (u+s,w+x)$ is $C^{\infty}(\R_+\times \R)$ and has bounded derivatives in the open set $\widetilde{B}=\{ (s,w) \in \R_+\times \R: z+(s,w)\in B \text{ for all }z\in \textbf{B}(z_0,a)  \}$.
Moreover, since $(u,x) \mapsto \int_{(-\infty,-1)} |f(u,x+y)|\Pi(\dd y)$ is bounded in $B$, the function $(s,w)\mapsto \int_{(-\infty,-1)} (f*\rho_n) (u+s,w+x+y)\Pi(\dd y)$ is bounded in $\widetilde{B}$. Thus, since $\tau_B$ is the first exit time of $(s,X_s)$ from the set $\widetilde{B}$, we get from Lemma \ref{lemma:martingaledecompositionofftX} that 
\begin{align*}
(f*\rho_n)&(u+t\wedge \tau_B,X_{t\wedge \tau_B}+x)\\
&\qquad=(f*\rho_n)(u,x)+M_{t}^{(u,x)}+\int_0^{t\wedge \tau_B} \mathcal{A}_{(t,X)}(f*\rho_n)(u+s,X_s+x) \dd s,
\end{align*}
where $\{ M_{t}^{(u,x)}, t\geq 0\}$ is a martingale. Therefore equation \eqref{eq:submartingfstarrhoninequality} reads 
\begin{align*}
\E\left( \int_0^{t\wedge \tau_B} \left[\mathcal{A}_{(t,X)}(f*\rho_n)(u+r,X_r+x) +(G*\rho_n)(u+r,X_r+x)\right]\dd r \right) \geq 0.
\end{align*}
Note that $\tau_B>0$ a.s. (since $\textbf{B}(0,a)\subset \widetilde{B}$), dividing by $t>0$ the equation above and taking $t\downarrow 0$ we obtain that
\begin{align}
\label{eq:freeboundaryproblemconvolution}
 \mathcal{A}_{(t,X)}(f*\rho_n)(u,x)+ (G*\rho_n)(u,x) \geq 0
\end{align}
for all $n\geq 1$ and $(u,x)\in \textbf{B}(z_0,a/2)$. That implies that for any non-negative test function $\psi$ in $\textbf{B}(z_0,a/2)$ 
\begin{align*}
\langle \mathcal{A}_{(t,x)}(f*\rho_n)+ G*\rho_n, \psi \rangle \geq 0.
\end{align*}
Then, from Proposition \ref{prop:convolutionpropertyofdistributions} we conclude that $\mathcal{A}_{(t,X)}(f)*\rho_n+ G* \rho_n\geq 0$ in the sense of distributions on $\textbf{B}(z_0,a/2)$. By letting $n$ go to infinity, we conclude that $\mathcal{A}_{(t,X)}(f)+ G \geq 0$ on $\textbf{B}(z_0,a/2)$ in the sense of distributions. Since $z_0$ is any arbitrary point in $B$, using a partition of unity argument, we conclude that $\mathcal{A}_{(t,X)}(f)+ G\geq 0$ in the sense of distributions on $B$.

\end{proof}

\end{appendix}
%
%

\begin{acks}[Acknowledgments]
Support from the Department of Statistics of LSE and the LSE PhD Studentship is gratefully acknowledged by Jos\'e M. Pedraza. The authors would like to thank two anonymous referees for their helpful comments and suggestions.
\end{acks}
\bibliographystyle{imsart-nameyear.bst} 
\bibliography{Lppredicintg}       

\begin{thebibliography}{37}

\bibitem[\protect\citeauthoryear{Az\'ema and Yor}{1989}]{AzemaYor1989}
\begin{barticle}[author]
\bauthor{\bsnm{Az\'ema},~\bfnm{Jacques}\binits{J.}} \AND
  \bauthor{\bsnm{Yor},~\bfnm{Marc}\binits{M.}}
(\byear{1989}).
\btitle{{\'Etude d'une martingale remarquable}}.
\bjournal{S\'eminaire de probabilit\'es de Strasbourg}
\bvolume{23}
\bpages{88-130}.
\end{barticle}
\endbibitem

\bibitem[\protect\citeauthoryear{Barker and Newby}{2009}]{BARKER200933}
\begin{barticle}[author]
\bauthor{\bsnm{Barker},~\bfnm{C.~T.}\binits{C.~T.}} \AND
  \bauthor{\bsnm{Newby},~\bfnm{M.~J.}\binits{M.~J.}}
(\byear{2009}).
\btitle{{Optimal non-periodic inspection for a multivariate degradation
  model}}.
\bjournal{Reliability Engineering \& System Safety}
\bvolume{94}
\bpages{33 - 43}.
\bnote{Maintenance Modeling and Application}.
\bdoi{https://doi.org/10.1016/j.ress.2007.03.015}
\end{barticle}
\endbibitem

\bibitem[\protect\citeauthoryear{Baurdoux, Kyprianou and
  Ott}{2016}]{baurdoux2016optimal}
\begin{barticle}[author]
\bauthor{\bsnm{Baurdoux},~\bfnm{Erik~J.}\binits{E.~J.}},
  \bauthor{\bsnm{Kyprianou},~\bfnm{Andreas~E.}\binits{A.~E.}} \AND
  \bauthor{\bsnm{Ott},~\bfnm{Curdin}\binits{C.}}
(\byear{2016}).
\btitle{{Optimal prediction for positive self-similar Markov processes}}.
\bjournal{Electron. J. Probab.}
\bvolume{21}
\bpages{24 pp.}
\bdoi{10.1214/16-EJP4280}
\end{barticle}
\endbibitem

\bibitem[\protect\citeauthoryear{Baurdoux and
  Pedraza}{2020}]{baurdoux2018predicting}
\begin{binproceedings}[author]
\bauthor{\bsnm{Baurdoux},~\bfnm{Erik~J}\binits{E.~J.}} \AND
  \bauthor{\bsnm{Pedraza},~\bfnm{Jos{\'e}~M}\binits{J.~M.}}
(\byear{2020}).
\btitle{{Predicting the last zero of a spectrally negative L{\'e}vy process}}.
In \bbooktitle{XIII Symposium on Probability and Stochastic Processes}
\bpages{77--105}.
\bpublisher{Springer}.
\end{binproceedings}
\endbibitem

\bibitem[\protect\citeauthoryear{Baurdoux and
  Pedraza}{2022}]{baurdoux2019Itolastzero}
\begin{barticle}[author]
\bauthor{\bsnm{Baurdoux},~\bfnm{Erik~J.}\binits{E.~J.}} \AND
  \bauthor{\bsnm{Pedraza},~\bfnm{J.~M.}\binits{J.~M.}}
(\byear{2022}).
\btitle{{On the last zero process with applications in corporate bankruptcy}}.
\bjournal{arXiv preprint arXiv:2003.06871}.
\end{barticle}
\endbibitem

\bibitem[\protect\citeauthoryear{Baurdoux and van
  Schaik}{2014}]{baurdoux2014predicting}
\begin{barticle}[author]
\bauthor{\bsnm{Baurdoux},~\bfnm{Erik~J.}\binits{E.~J.}} \AND
  \bauthor{\bparticle{van} \bsnm{Schaik},~\bfnm{Kees}\binits{K.}}
(\byear{2014}).
\btitle{{Predicting the time at which a L{\'{e}}vy process attains its ultimate
  supremum}}.
\bjournal{Acta Applicandae Mathematicae}
\bvolume{134}
\bpages{21--44}.
\bdoi{10.1007/s10440-014-9867-2}
\end{barticle}
\endbibitem

\bibitem[\protect\citeauthoryear{Bayraktar and
  Xing}{2012}]{bayraktar2012regularity}
\begin{barticle}[author]
\bauthor{\bsnm{Bayraktar},~\bfnm{Erhan}\binits{E.}} \AND
  \bauthor{\bsnm{Xing},~\bfnm{Hao}\binits{H.}}
(\byear{2012}).
\btitle{Regularity of the optimal stopping problem for jump diffusions}.
\bjournal{SIAM Journal on Control and Optimization}
\bvolume{50}
\bpages{1337--1357}.
\end{barticle}
\endbibitem

\bibitem[\protect\citeauthoryear{Bertoin}{1998}]{bertoin1998levy}
\begin{bbook}[author]
\bauthor{\bsnm{Bertoin},~\bfnm{Jean}\binits{J.}}
(\byear{1998}).
\btitle{L{\'e}vy processes}
\bvolume{121}.
\bpublisher{Cambridge university press}.
\end{bbook}
\endbibitem

\bibitem[\protect\citeauthoryear{Bichteler}{2002}]{bichteler2002stochastic}
\begin{bbook}[author]
\bauthor{\bsnm{Bichteler},~\bfnm{Klaus}\binits{K.}}
(\byear{2002}).
\btitle{{Stochastic integration with jumps}}.
\bseries{Encyclopedia of Mathematics and its Applications}.
\bpublisher{Cambridge University Press}.
\bdoi{10.1017/CBO9780511549878}
\end{bbook}
\endbibitem

\bibitem[\protect\citeauthoryear{Chiu and Yin}{2005}]{chiu2005passage}
\begin{barticle}[author]
\bauthor{\bsnm{Chiu},~\bfnm{Sung~Nok}\binits{S.~N.}} \AND
  \bauthor{\bsnm{Yin},~\bfnm{Chuancun}\binits{C.}}
(\byear{2005}).
\btitle{{Passage times for a spectrally Negative L\'evy Process with
  applications to risk theory}}.
\bjournal{Bernoulli}
\bvolume{11}
\bpages{511--522}.
\end{barticle}
\endbibitem

\bibitem[\protect\citeauthoryear{Cont and Voltchkova}{2005}]{cont2005integro}
\begin{barticle}[author]
\bauthor{\bsnm{Cont},~\bfnm{Rama}\binits{R.}} \AND
  \bauthor{\bsnm{Voltchkova},~\bfnm{Ekaterina}\binits{E.}}
(\byear{2005}).
\btitle{Integro-differential equations for option prices in exponential
  L{\'e}vy models}.
\bjournal{Finance and Stochastics}
\bvolume{9}
\bpages{299--325}.
\end{barticle}
\endbibitem

\bibitem[\protect\citeauthoryear{Doney and Maller}{2004}]{doney2004moments}
\begin{barticle}[author]
\bauthor{\bsnm{Doney},~\bfnm{R.~A.}\binits{R.~A.}} \AND
  \bauthor{\bsnm{Maller},~\bfnm{R.~A.}\binits{R.~A.}}
(\byear{2004}).
\btitle{{Moments of passage times for Lévy processes}}.
\bjournal{Annales de l'Institut Henri Poincare (B) Probability and Statistics}
\bvolume{40}
\bpages{279 - 297}.
\bdoi{https://doi.org/10.1016/j.anihpb.2003.10.005}
\end{barticle}
\endbibitem

\bibitem[\protect\citeauthoryear{du~Toit and Peskir}{2008}]{duToit2008}
\begin{bincollection}[author]
\bauthor{\bparticle{du} \bsnm{Toit},~\bfnm{Jacques}\binits{J.}} \AND
  \bauthor{\bsnm{Peskir},~\bfnm{Goran}\binits{G.}}
(\byear{2008}).
\btitle{{Predicting the time of the ultimate maximum for Brownian motion with
  drift}}.
In \bbooktitle{Mathematical Control Theory and Finance}
\bpages{95--112}.
\bpublisher{Springer Berlin Heidelberg}.
\bdoi{10.1007/978-3-540-69532-5_6}
\end{bincollection}
\endbibitem

\bibitem[\protect\citeauthoryear{du~Toit, Peskir and
  Shiryaev}{2008}]{du2008predicting}
\begin{barticle}[author]
\bauthor{\bparticle{du} \bsnm{Toit},~\bfnm{J.}\binits{J.}},
  \bauthor{\bsnm{Peskir},~\bfnm{G.}\binits{G.}} \AND
  \bauthor{\bsnm{Shiryaev},~\bfnm{A.~N.}\binits{A.~N.}}
(\byear{2008}).
\btitle{{Predicting the last zero of Brownian motion with drift}}.
\bjournal{Stochastics}
\bvolume{80}
\bpages{229-245}.
\bdoi{10.1080/17442500701840950}
\end{barticle}
\endbibitem

\bibitem[\protect\citeauthoryear{Egami and Kevkhishvili}{2020}]{egami2017time}
\begin{barticle}[author]
\bauthor{\bsnm{Egami},~\bfnm{Masahiko}\binits{M.}} \AND
  \bauthor{\bsnm{Kevkhishvili},~\bfnm{Rusudan}\binits{R.}}
(\byear{2020}).
\btitle{{Time reversal and last passage time of diffusions with applications to
  credit risk management}}.
\bjournal{Finance and Stochastics}
\bvolume{24}
\bpages{795-825}.
\bdoi{10.1007/s00780-020-00423-}
\end{barticle}
\endbibitem

\bibitem[\protect\citeauthoryear{Friedlander
  et~al.}{1998}]{friedlander1998introduction}
\begin{bbook}[author]
\bauthor{\bsnm{Friedlander},~\bfnm{Friedrich~Gerard}\binits{F.~G.}},
  \bauthor{\bsnm{Friedlander},~\bfnm{G}\binits{G.}},
  \bauthor{\bsnm{Joshi},~\bfnm{Mark~Suresh}\binits{M.~S.}},
  \bauthor{\bsnm{Joshi},~\bfnm{M}\binits{M.}} \AND
  \bauthor{\bsnm{Joshi},~\bfnm{Mohan~C}\binits{M.~C.}}
(\byear{1998}).
\btitle{Introduction to the Theory of Distributions}.
\bpublisher{Cambridge University Press}.
\end{bbook}
\endbibitem

\bibitem[\protect\citeauthoryear{Garroni and Menaldi}{2002}]{garroni2002second}
\begin{bbook}[author]
\bauthor{\bsnm{Garroni},~\bfnm{Maria~Giovanna}\binits{M.~G.}} \AND
  \bauthor{\bsnm{Menaldi},~\bfnm{Jose~Luis}\binits{J.~L.}}
(\byear{2002}).
\btitle{Second order elliptic integro-differential problems}.
\bpublisher{Chapman and Hall/CRC}.
\end{bbook}
\endbibitem

\bibitem[\protect\citeauthoryear{Gerber}{1990}]{GERBER1990115}
\begin{barticle}[author]
\bauthor{\bsnm{Gerber},~\bfnm{Hans~U.}\binits{H.~U.}}
(\byear{1990}).
\btitle{When does the surplus reach a given target?}
\bjournal{Insurance: Mathematics and Economics}
\bvolume{9}
\bpages{115-119}.
\bdoi{https://doi.org/10.1016/0167-6687(90)90022-6}
\end{barticle}
\endbibitem

\bibitem[\protect\citeauthoryear{Glover, Hulley and
  Peskir}{2013}]{glover2013three}
\begin{barticle}[author]
\bauthor{\bsnm{Glover},~\bfnm{Kristoffer}\binits{K.}},
  \bauthor{\bsnm{Hulley},~\bfnm{Hardy}\binits{H.}} \AND
  \bauthor{\bsnm{Peskir},~\bfnm{Goran}\binits{G.}}
(\byear{2013}).
\btitle{{Three-dimensional Brownian motion and the golden ratio rule}}.
\bjournal{Ann. Appl. Probab.}
\bvolume{23}
\bpages{895--922}.
\bdoi{10.1214/12-AAP859}
\end{barticle}
\endbibitem

\bibitem[\protect\citeauthoryear{Glover and Hulley}{2014}]{glover2014optimal}
\begin{barticle}[author]
\bauthor{\bsnm{Glover},~\bfnm{Kristoffer}\binits{K.}} \AND
  \bauthor{\bsnm{Hulley},~\bfnm{Hardy}\binits{H.}}
(\byear{2014}).
\btitle{{Optimal prediction of the last-passage time of a transient
  diffusion}}.
\bjournal{SIAM Journal on Control and Optimization}
\bvolume{52}
\bpages{3833-3853}.
\bdoi{10.1137/130950719}
\end{barticle}
\endbibitem

\bibitem[\protect\citeauthoryear{Huzak et~al.}{2004}]{huzak2004ruin}
\begin{barticle}[author]
\bauthor{\bsnm{Huzak},~\bfnm{Miljenko}\binits{M.}},
  \bauthor{\bsnm{Perman},~\bfnm{Mihael}\binits{M.}},
  \bauthor{\bsnm{{\v{S}}iki{\'c}},~\bfnm{Hrvoje}\binits{H.}} \AND
  \bauthor{\bsnm{Vondra{\v{c}}ek},~\bfnm{Zoran}\binits{Z.}}
(\byear{2004}).
\btitle{Ruin probabilities and decompositions for general perturbed risk
  processes}.
\bjournal{The Annals of Applied Probability}
\bvolume{14}
\bpages{1378--1397}.
\end{barticle}
\endbibitem

\bibitem[\protect\citeauthoryear{Kl{\"u}ppelberg, Kyprianou and
  Maller}{2004}]{kluppelberg2004ruin}
\begin{barticle}[author]
\bauthor{\bsnm{Kl{\"u}ppelberg},~\bfnm{Claudia}\binits{C.}},
  \bauthor{\bsnm{Kyprianou},~\bfnm{Andreas~E}\binits{A.~E.}} \AND
  \bauthor{\bsnm{Maller},~\bfnm{Ross~A}\binits{R.~A.}}
(\byear{2004}).
\btitle{Ruin probabilities and overshoots for general L{\'e}vy insurance risk
  processes}.
\bjournal{The Annals of Applied Probability}
\bvolume{14}
\bpages{1766--1801}.
\end{barticle}
\endbibitem

\bibitem[\protect\citeauthoryear{Kuznetsov, Kyprianou and
  Rivero}{2013}]{kyprianou2011theory}
\begin{binbook}[author]
\bauthor{\bsnm{Kuznetsov},~\bfnm{Alexey}\binits{A.}},
  \bauthor{\bsnm{Kyprianou},~\bfnm{Andreas~E.}\binits{A.~E.}} \AND
  \bauthor{\bsnm{Rivero},~\bfnm{Victor}\binits{V.}}
(\byear{2013}).
\btitle{{The theory of scale functions for spectrally negative L{\'e}vy
  processes}}
In \bbooktitle{L{\'e}vy Matters II: Recent Progress in Theory and Applications:
  Fractional L{\'e}vy Fields, and Scale Functions}
\bpages{97--186}.
\bpublisher{Springer Berlin Heidelberg}, \baddress{Berlin, Heidelberg}.
\bdoi{10.1007/978-3-642-31407-0_2}
\end{binbook}
\endbibitem

\bibitem[\protect\citeauthoryear{Kuznetsov et~al.}{2011}]{kuznetsov2011wiener}
\begin{barticle}[author]
\bauthor{\bsnm{Kuznetsov},~\bfnm{A}\binits{A.}},
  \bauthor{\bsnm{Kyprianou},~\bfnm{{Andreas E}}\binits{A.}},
  \bauthor{\bsnm{Pardo},~\bfnm{Juan-Carlos}\binits{J.-C.}} \AND
  \bauthor{\bsnm{{Van Schaik}},~\bfnm{Kees}\binits{K.}}
(\byear{2011}).
\btitle{{A Wiener--Hopf Monte Carlo simulation technique for L{\'e}vy
  process}}.
\bjournal{Annals of Applied Probability}
\bvolume{21}
\bpages{2171--2190}.
\bdoi{10.1214/10-AAP746}
\end{barticle}
\endbibitem

\bibitem[\protect\citeauthoryear{Kyprianou}{2014}]{kyprianou2014fluctuations}
\begin{bbook}[author]
\bauthor{\bsnm{Kyprianou},~\bfnm{Andreas~E.}\binits{A.~E.}}
(\byear{2014}).
\btitle{{Fluctuations of L{\'{e}}vy processes with applications}}.
\bpublisher{Springer Berlin Heidelberg}.
\bdoi{10.1007/978-3-642-37632-0}
\end{bbook}
\endbibitem

\bibitem[\protect\citeauthoryear{Lamberton and
  Mikou}{2008}]{lamberton2008critical}
\begin{barticle}[author]
\bauthor{\bsnm{Lamberton},~\bfnm{Damien}\binits{D.}} \AND
  \bauthor{\bsnm{Mikou},~\bfnm{Mohammed}\binits{M.}}
(\byear{2008}).
\btitle{{The critical price for the American put in~an~exponential L{\'{e}}vy
  model}}.
\bjournal{Finance and Stochastics}
\bvolume{12}
\bpages{561--581}.
\bdoi{10.1007/s00780-008-0073-9}
\end{barticle}
\endbibitem

\bibitem[\protect\citeauthoryear{Lamberton and Mikou}{2013}]{Lamberton2013}
\begin{barticle}[author]
\bauthor{\bsnm{Lamberton},~\bfnm{Damien}\binits{D.}} \AND
  \bauthor{\bsnm{Mikou},~\bfnm{Mohammed~Adam}\binits{M.~A.}}
(\byear{2013}).
\btitle{{Exercise boundary of the American put near maturity in an exponential
  L{\'e}vy model}}.
\bjournal{Finance and Stochastics}
\bvolume{17}
\bpages{355--394}.
\bdoi{10.1007/s00780-012-0194-z}
\end{barticle}
\endbibitem

\bibitem[\protect\citeauthoryear{Laue}{1980}]{laue1980remarks}
\begin{barticle}[author]
\bauthor{\bsnm{Laue},~\bfnm{G.}\binits{G.}}
(\byear{1980}).
\btitle{{Remarks on the relation between fractional moments and fractional
  derivatives of characteristic functions}}.
\bjournal{Journal of Applied Probability}
\bvolume{17}
\bpages{456--466}.
\end{barticle}
\endbibitem

\bibitem[\protect\citeauthoryear{Paroissin and
  Rabehasaina}{2013}]{paroissin2013first}
\begin{barticle}[author]
\bauthor{\bsnm{Paroissin},~\bfnm{Christian}\binits{C.}} \AND
  \bauthor{\bsnm{Rabehasaina},~\bfnm{Landy}\binits{L.}}
(\byear{2013}).
\btitle{{First and last Passage times of spectrally positive L{\'{e}}vy
  processes with application to reliability}}.
\bjournal{Methodology and Computing in Applied Probability}
\bvolume{17}
\bpages{351--372}.
\bdoi{10.1007/s11009-013-9360-9}
\end{barticle}
\endbibitem

\bibitem[\protect\citeauthoryear{Peskir and Shiryaev}{2006}]{peskir2006optimal}
\begin{bbook}[author]
\bauthor{\bsnm{Peskir},~\bfnm{Goran}\binits{G.}} \AND
  \bauthor{\bsnm{Shiryaev},~\bfnm{Albert}\binits{A.}}
(\byear{2006}).
\btitle{{Optimal stopping and free-boundary problems}}.
\bpublisher{Birkhäuser Basel}.
\bdoi{10.1007/978-3-7643-7390-0}
\end{bbook}
\endbibitem

\bibitem[\protect\citeauthoryear{Protter}{2005}]{protter2005}
\begin{bbook}[author]
\bauthor{\bsnm{Protter},~\bfnm{Philip~E.}\binits{P.~E.}}
(\byear{2005}).
\btitle{Stochastic integration and differential equations}.
\bpublisher{Springer Berlin Heidelberg}.
\bdoi{10.1007/978-3-662-10061-5}
\end{bbook}
\endbibitem

\bibitem[\protect\citeauthoryear{Rudin}{1991}]{rudin1991functional}
\begin{bbook}[author]
\bauthor{\bsnm{Rudin},~\bfnm{W.}\binits{W.}}
(\byear{1991}).
\btitle{Functional Analysis}.
\bseries{International series in pure and applied mathematics}.
\bpublisher{McGraw-Hill}.
\end{bbook}
\endbibitem

\bibitem[\protect\citeauthoryear{Sato}{1999}]{sato1999levy}
\begin{bbook}[author]
\bauthor{\bsnm{Sato},~\bfnm{Ken-iti}\binits{K.-i.}}
(\byear{1999}).
\btitle{{L{\'e}vy processes and infinitely divisible distributions}}.
\bpublisher{Cambridge university press}.
\end{bbook}
\endbibitem

\bibitem[\protect\citeauthoryear{Shiryaev}{2002}]{Shiryaev2002}
\begin{bincollection}[author]
\bauthor{\bsnm{Shiryaev},~\bfnm{Albert~N.}\binits{A.~N.}}
(\byear{2002}).
\btitle{{Quickest detection problems in the technical analysis of the financial
  data}}.
In \bbooktitle{Mathematical Finance --- Bachelier Congress 2000: Selected
  Papers from the First World Congress of the Bachelier Finance Society, Paris,
  June 29--July 1, 2000}
(\beditor{\bfnm{H{\'e}lyette}\binits{H.}~\bsnm{Geman}},
  \beditor{\bfnm{Dilip}\binits{D.}~\bsnm{Madan}},
  \beditor{\bfnm{Stanley~R.}\binits{S.~R.}~\bsnm{Pliska}} \AND
  \beditor{\bfnm{Ton}\binits{T.}~\bsnm{Vorst}}, eds.)
\bpages{487--521}.
\bpublisher{Springer Berlin Heidelberg}, \baddress{Berlin, Heidelberg}.
\bdoi{10.1007/978-3-662-12429-1_22}
\end{bincollection}
\endbibitem

\bibitem[\protect\citeauthoryear{Shiryaev}{2009}]{shiryaev2009}
\begin{barticle}[author]
\bauthor{\bsnm{Shiryaev},~\bfnm{A.~N.}\binits{A.~N.}}
(\byear{2009}).
\btitle{{On conditional-extremal problems of the quickest detection of
  nonpredictable times of the observable Brownian motion}}.
\bjournal{Theory of Probability \& Its Applications}
\bvolume{53}
\bpages{663-678}.
\bdoi{10.1137/S0040585X97983882}
\end{barticle}
\endbibitem

\bibitem[\protect\citeauthoryear{Spindler}{2005}]{spindler2005short}
\begin{barticle}[author]
\bauthor{\bsnm{Spindler},~\bfnm{Karlheinz}\binits{K.}}
(\byear{2005}).
\btitle{{A short proof of the formula of Fa{\`{a}} di Bruno}}.
\bjournal{Elemente der Mathematik}
\bpages{33--35}.
\bdoi{10.4171/em/5}
\end{barticle}
\endbibitem

\bibitem[\protect\citeauthoryear{Urusov}{2005}]{urusov2005property}
\begin{barticle}[author]
\bauthor{\bsnm{Urusov},~\bfnm{M.~A.}\binits{M.~A.}}
(\byear{2005}).
\btitle{{On a property of the moment at which Brownian motion attains its
  maximum and some optimal stopping problems}}.
\bjournal{Theory of Probability \& Its Applications}
\bvolume{49}
\bpages{169-176}.
\bdoi{10.1137/S0040585X97980956}
\end{barticle}
\endbibitem

\end{thebibliography}


\end{document}